\documentclass[12pt,amsymb,fullpage]{amsart}
\usepackage{amssymb,amscd,pstricks}
\newtheorem{Theorem}{Theorem}[section]

\newtheorem{theorem}{Theorem}[section]
\newtheorem{defn}[theorem]{Definition}

\newtheorem{lemma}[theorem]{Lemma}

\newtheorem{eple}[theorem]{Example}
\newtheorem{rmk}[theorem]{Remarks}
\newtheorem{dsc}[theorem]{Discussion}
\newtheorem{nota}[theorem]{Notation}

\newsavebox{\indbin}
\savebox{\indbin}{\begin{picture}(0,0)
\newlength{\gnu}
\settowidth{\gnu}{$\smile$} \setlength{\unitlength}{.5\gnu}
\put(-1,-.65){$\smile$} \put(-.25,.1){$|$}
\end{picture}}

\newcommand{\be}{\begin{enumerate}}
\newcommand{\bd}{\begin{defn}}
\newcommand{\bt}{\begin{theorem}}
\newcommand{\bl}{\begin{lemma}}
\newcommand{\ee}{\end{enumerate}}
\newcommand{\ed}{\end{defn}}
\newcommand{\et}{\end{theorem}}
\newcommand{\el}{\end{lemma}}

\begin{document}
\title{A Theory of Duality for Algebraic Curves}
\author{Tristram de Piro}
\address{Flat 1, 98 Prestbury Road, Cheltenham, GL52 2DJ}
 \email{depiro100@gmail.com}
\thanks{}
\begin{abstract}
We formulate a refined theory of $g_{n}^{r}$, using the methods of
\cite{depiro6}, and use the theory to give a geometric interpretation
of the genus of an algebraic curve. Using principles of duality, we prove
generalisations of Plucker's formulae for algebraic curves. The results hold
for arbitrary characteristic of the base field $L$, with some occasional
exceptions when $char(L)=2$, which we observe in the course of the paper.
\end{abstract}
\maketitle
\begin{section}{A refined theory of $g_{n}^{r}$}

The purpose of this section is to refine the general theory of
$g_{n}^{r}$, in order to take into account the notion of a branch
for a projective algebraic curve. We will rely heavily on results
proved in \cite{depiro6}. We also refer the reader there for the
relevant notation. Unless otherwise stated, we will assume that
the characteristic of the field $L$ is zero, making the modifications
for non-zero characteristic in the final section. By an algebraic curve, we always mean
a projective irreducible variety of dimension $1$.\\

\begin{defn}

Let $C\subset P^{w}$ be a projective algebraic curve of degree $d$
and let $\Sigma$ be a linear system of dimension $R$, contained in
the space of algebraic forms of degree $e$ on $P^{w}$. Let
$\phi_{\lambda}$ belong to $\Sigma$, having finite intersection
with $C$. Then, if $p\in C\cap\phi_{\lambda}$ and $\gamma_{p}$ is
a branch centred at $p$, we define;\\

$I_{p}(C,\phi_{\lambda})=I_{italian}(p,C,\phi_{\lambda})$\\

$I_{p}^{\Sigma}(C,\phi_{\lambda})=I_{italian}^{\Sigma}(p,C,\phi_{\lambda})$\\

$I_{p}^{\Sigma,mobile}(C,\phi_{\lambda})=I_{italian}^{\Sigma,mobile}(p,C,\phi_{\lambda})$\\

$I_{\gamma_{p}}(C,\phi_{\lambda})=I_{italian}(p,\gamma_{p},C,\phi_{\lambda})$\\

$I_{\gamma_{p}}^{\Sigma}(C,\phi_{\lambda})=I_{italian}^{\Sigma}(p,\gamma_{p},C,\phi_{\lambda})$\\

$I_{\gamma_{p}}^{\Sigma,mobile}(C,\phi_{\lambda})=I_{italian}^{\Sigma,mobile}(p,\gamma_{p},C,\phi_{\lambda})$\\

where $I_{italian}$ was defined in \cite{depiro6}.

\end{defn}

It follows that, as $\lambda$ varies in $Par_{\Sigma}$, we obtain
a series of weighted sets;\\

$W_{\lambda}=\{n_{\gamma_{p_{1}}^{1}},\ldots,n_{\gamma_{p_{1}}^{n_{1}}},\ldots,n_{\gamma_{p_{m}}^{1}},\ldots,n_{\gamma_{p_{m}}^{n_{m}}}\}$\\

where;\\

$\{p_{1},\ldots,p_{i},\ldots,p_{m}\}=C\cap\phi_{\lambda}$, for $1\leq i\leq m$,\\

$\{\gamma_{p_{i}}^{1},\ldots,\gamma_{p_{i}}^{j(i)},\ldots,\gamma_{p_{i}}^{n_{i}}\}$,
for $1\leq j(i)\leq n_{i}$, consists of the branches of $C$ centred at $p_{i}$\\

and\\

$I_{\gamma_{p_{i}}^{j(i)}}(C,\phi_{\lambda})=n_{\gamma_{p_{i}}^{j(i)}}$\\

By the branched version of the Hyperspatial Bezout Theorem, see
\cite{depiro6}, the total weight of any of these sets, which we
will occasionally abbreviate by $C\sqcap\phi_{\lambda}$, is always
equal to $de$. Let $r$ be the least integer such that every
weighted set $W_{\lambda}$ is defined by a linear subsystem
$\Sigma'\subset\Sigma$ of dimension $r$.\\

\begin{defn}

We define;\\

$Series(\Sigma)=\{W_{\lambda}:\lambda\in Par_{\Sigma}\}$\\

$dimension(Series(\Sigma))=r$\\

$order(Series(\Sigma))=de$\\

\end{defn}

We then claim the following;\\

\begin{theorem}.\\

(i). $r\leq R$, with equality iff every weighted set $W_{\lambda}$
of the series is cut out by a \emph{single} form of $\Sigma$.\\

(ii). $r\lneqq R$ iff there exists a form $\phi_{\lambda}$ in
$\Sigma$, containing all of $C$.
\end{theorem}

\begin{proof}
We first show the equivalence of $(i)$ and $(ii)$. Suppose that
$(i)$ holds and $r\lneqq R$. Then, we can find a weighted set $W$
and distinct elements $\{\lambda_{1},\lambda_{2}\}$ of
$Par_{\Sigma}$ such that $W=W_{\lambda_{1}}=W_{\lambda_{2}}$. Let
$\{\phi_{\lambda_{1}},\phi_{\lambda_{2}}\}$ be the corresponding
algebraic forms of $\Sigma$ and consider the pencil
$\Sigma_{1}\subset \Sigma$ defined by these forms. We claim
that;\\

$W=C\sqcap(\mu_{1}\phi_{\lambda_{1}}+\mu_{2}\phi_{\lambda_{2}})$,
for $[\mu_{1}:\mu_{2}]\in P^{1}$ $(*)$\\

This follows immediately from the results in \cite{depiro6} that
the condition of multiplicity at a branch is \emph{linear} and the
branched version of the Hyperspatial Bezout Theorem. Now choose a
point $p\in C$ not contained in $W$. Then, the condition that an
algebraic form $\phi_{\lambda}$ passes through $p$ defines a
hyperplane condition on $Par_{k}$, hence, intersects
$Par_{\Sigma_{1}}$ in a point. Let $\phi_{\lambda_{0}}$ be the
algebraic form in $\Sigma_{1}$ defined by this parameter. Then, by
$(*)$, we have that;\\

$W\cup\{p\}\subseteq C\sqcap\phi_{\lambda_{0}}$\\

Hence, the total multiplicity of intersection of
$\phi_{\lambda_{0}}$ with $C$ is at least equal to $de+1$. By the
branched version of the Hyperspatial Bezout Theorem, $C$ must be
contained in $\phi_{\lambda_{0}}$. Conversely, suppose that $(i)$
holds and there exists a form $\phi_{\lambda_{0}}$ in $\Sigma$
containing all of $C$. Let $W$ be cut out by $\phi_{\lambda_{1}}$
and consider the pencil $\Sigma_{1}\subset\Sigma$ generated by
$\{\phi_{\lambda_{0}},\phi_{\lambda_{1}}\}$. By the same argument
as above, we can find $\phi_{\lambda_{2}}$ in $\Sigma_{1}$,
distinct from $\phi_{\lambda_{1}}$, which also cuts out $W$.
Hence, by $(i)$, we must have that $r\lneq R$. Therefore, $(ii)$
holds. \\

The argument that $(ii)$ implies $(i)$ is similar.\\

We now prove that $(ii)$ holds. Using the Hyperspatial Bezout
Theorem, the condition on $Par_{\Sigma}$ that a form
$\phi_{\lambda}$ contains $C$ is linear. Let $H$ be the linear
subsystem of $\Sigma$, consisting of forms containing $C$ and let
$h=dim(H)$. Let $K\subset\Sigma$ be a maximal linear subsystem,
having finite intersection with $C$. Then $K$ has no form in
common with $H$ and $dim(K)=R-h-1$. We claim that every weighted
set in $Series(\Sigma)$ is cut out by a unique form from $K$. For
suppose that $W=C\sqcap\phi_{\lambda}$ is such a weighted set and
consider the linear system defined by $<H,\phi_{\lambda}>$. If
$\phi_{\mu}$ belongs to this system and has finite intersection
with $C$, then clearly $(C\cap\phi_{\lambda})=(C\cap\phi_{\mu})$.
Using linearity of multiplicity at a branch and the Hyperspatial
Bezout Theorem again (by convention, a form containing $C$ has
infinite multiplicity at a branch), we must have that
$(C\sqcap\phi_{\lambda})=(C\sqcap\phi_{\mu})$. Now consider
$K\cap <H,\phi_{\lambda}>$. We have that;\\

$codim(K\cap <H,\phi_{\lambda}>)\leq
codim(K)+codim(<H,\phi_{\lambda}>)$\\
$. \ \ \ \ \ \ \ \ \ \ \ \ \ \ \ \ \ \ \ \ \ \ \ \ \ \ \ \ \ \ \ \ =(h+1)+(R-(h+1))=R$.\\

 Hence,
$dim(K\cap <H,\phi_{\lambda}>)\geq 0$. We can, therefore, find a
form $\phi_{\mu}$ belonging to $K$ such that
$W=(C\sqcap\phi_{\mu})$. We need to show that $\phi_{\mu}$ is the
unique form in $K$ defining $W$. This follows by the argument
given above. It follows immediately that $r=dim(K)=R-h-1$. Hence,
$r\lneq R$ iff $h\geq 0$. Therefore, $(ii)$ is shown.

\end{proof}

Using this theorem, we give a more refined definition of a
$g_{n}^{r}$.\\

\begin{defn}
Let $C\subset P^{w}$ be a projective algebraic curve. By a
$g_{n}^{r}$ on $C$, we mean the collection of weighted sets,
without repetitions, defined by $Series(\Sigma)$ for \emph{some}
linear system $\Sigma$, such that $r=dimension(Series(\Sigma))$
and $n=order(Series(\Sigma))$. If a branch $\gamma_{p}^{j}$
appears with multiplicity at least $s$ in every weighted set of a
$g_{n}^{r}$, as just defined, then we allow the possibility of
removing some multiplicity contribution $s'\leq s$ from each
weighted set and adjusting $n$ to $n'=n-s'$.

\end{defn}

\begin{rmk}
The reader should observe carefully that a $g_{n}^{r}$ is defined
independently of a particular linear system. However, by the
previous theorem, for any $g_{n}^{r}$, there exists a $g_{n'}^{r}$
with $n\leq n'$ such that the following property holds. The
$g_{n'}^{r}$ is defined by a linear system of dimension $r$,
having finite intersection with $C$, such that each there is a
bijection between the weighted sets $W$ in the $g_{n'}^{r}$ and
the $W_{\lambda}$ in $Series(\Sigma)$. The original $g_{n}^{r}$ is
obtained from the $g_{n'}^{r}$ by removing some fixed point
contribution.
\end{rmk}

We now reformulate the results of Section 2 and Section 5 in
\cite{depiro6} for this new definition of a $g_{n}^{r}$. In order
to do this, we require the following definition;\\

\begin{defn}
Suppose that $C\subset P^{w}(L)$ is a projective algebraic curve
and $C^{ext}\subset P^{w}(K)$ is its non-standard model. Let a
$g_{n}^{r}$ be given on $C$, defined by a linear system $\Sigma$
after removing some fixed point contribution. We define the
extension $g_{n}^{r,ext}$ of the $g_{n}^{r}$ to the nonstandard
model $C^{ext}$ to be the collection of weighted sets, without
repetitions, defined by $Series(\Sigma)$ on $C^{ext}$, after
removing the same fixed point contribution. Note that, by
definability of multiplicity at a branch, see Theorem 6.5 of
\cite{depiro6}, if $\gamma_{p}^{j}$ is a branch of $C$ and;\\

$I_{italian}(p,\gamma_{p}^{j},C,\phi_{\lambda})\geq k$,
$(\lambda\in Par_{\Sigma(L)})$\\

then;\\

$I_{italian}(p,\gamma_{p}^{j},C,\phi_{\lambda})\geq k$,
$(\lambda\in Par_{\Sigma(K)})$\\

Hence, it \emph{is} possible to remove the same fixed point
contribution
of $Series(\Sigma)$ on $C^{ext}$. See also the proof of Lemma 1.7.\\

\end{defn}

It is a remarkable fact that, after introducing the notion of a
branch, the definition is independent of the particular linear
system $\Sigma$. This is the content of the following lemma;\\

\begin{lemma}
The previous definition is independent of the particular choice of
linear system $\Sigma$ defining the $g_{n}^{r}$.
\end{lemma}

\begin{proof}

We divide the proof into the following cases;\\

Case 1. $\Sigma\subset\Sigma'$;\\

By the proof of Theorem 1.3, we can find a linear system
$\Sigma_{0}\subset\Sigma\subset\Sigma'$ of dimension $r$, having
finite intersection with $C$, such that the $g_{n}^{r}$ is defined
by removing some fixed contribution from $\Sigma_{0}$. Here, we
have also used the fact that the base point contributions (at a
branch) of $\{\Sigma_{0},\Sigma,\Sigma'\}$ are the same. Again, by
Theorem 1.3, if $W_{\lambda'}$ is a weighted set defined by
$\Sigma'$ on $C^{ext}$, then it appears as a weighted set
$V_{\lambda''}$ defined by $\Sigma_{0}$ on $C^{ext}$. Hence, it
appears as a weighted set $V_{\lambda''}$ defined by $\Sigma$ on
$C^{ext}$. By the converse argument and the remark on base point
contributions, the proof is shown.\\

Case 2. $\Sigma$ are $\Sigma'$ are both linear systems of
dimension $r$, having finite intersection with $C$, such that $degree(\Sigma)=degree(\Sigma')=n$;\\

By Theorem 1.3, every weighted set $W$ in the $g_{n}^{r}$ is
defined uniquely by weighted sets $W_{\lambda_{1}}$ and
$V_{\lambda_{2}}$ in $Series(\Sigma_{1})$ and $Series(\Sigma_{2})$
respectively. Let $(C^{ns},\Phi^{ns})$ be a non-singular model of
$C$. Using the method of Section 5 in \cite{depiro6} to avoid the
technical problem of presentations of $\Phi^{ns}$ and base point
contributions, we may, without loss of generality, assume that
there exist finite covers $W_{1}\subset Par_{\Sigma}\times C^{ns}$ and $W_{2}\subset Par_{\Sigma'}\times C^{ns}$ such that;\\

$j_{k,\Sigma}(\lambda,p_{j})\equiv
Mult_{(W_{1}/Par_{\Sigma})}(\lambda,p_{j})\geq k$ iff
$I_{italian}(p,\gamma_{p}^{j},C,\phi_{\lambda})\geq k$\\

$j_{k,\Sigma'}(\lambda',p_{j})\equiv
Mult_{(W_{2}/Par_{\Sigma'})}(\lambda',p_{j})\geq k$ iff
$I_{italian}(p,\gamma_{p}^{j},C,\psi_{\lambda'})\geq k$\\

Then consider the sentences;\\

$(\forall \lambda\in Par_{\Sigma})(\exists!\lambda'\in
Par_{\Sigma'})\forall x\in
C^{ns}[\bigwedge_{k=1}^{n}(j_{k}(\lambda,x)\leftrightarrow
j_{k}(\lambda',x))]$\\

$(\forall \lambda'\in Par_{\Sigma})(\exists!\lambda\in
Par_{\Sigma})\forall x\in
C^{ns}[\bigwedge_{k=1}^{n}(j_{k}(\lambda',x)\leftrightarrow
j_{k}(\lambda,x))]$ (*)\\

in the language of $<P^{1}(L),C_{i}>$, considered as a Zariski
structure with predicates $\{C_{i}\}$ for Zariski closed subsets
defined over $L$, (see \cite{Z}). We have, again by results of
\cite{Z} or \cite{depiro4}, that $<P^{1}(L),C_{i}>\prec
<P^{1}(K),C_{i}>$, for the nonstandard model $P(K)$ of $P(L)$. It
follows immediately from the algebraic definition of $j_{k}$ in
\cite{Z}, that, for any weighted set $W_{\lambda_{1}'}$ defined by
$Series(\Sigma)$ on $C^{ext}$, there exists a unique weighted set
$V_{\lambda_{2}'}$ defined by $Series(\Sigma')$ on $C^{ext}$ such
that $W_{\lambda_{1}'}=V_{\lambda_{2}'}$, and conversely. Hence,
the proof is shown.\\

Case 3. $\Sigma$ are $\Sigma'$ are both linear systems of
dimension $r$, having finite intersection with $C$;\\

Let $n_{1}=degree(\Sigma)$ and $n_{2}=degree(\Sigma')$. Then the
original $g_{n}^{r}$ is obtained from $Series(\Sigma)$, by
removing a fixed point contribution of multiplicity $n_{1}-n$,
and, is obtained from $Series(\Sigma')$, by removing a fixed point
contribution of multiplicity $n_{2}-n$. We now imitate the proof
of Case 2, with the slight modification that, in the construction
of the sentences given by $(*)$, we make an adjustment of the
multiplicity statement at the finite number of branches where a
fixed point contribution has been removed. The details are left to
the reader.
\end{proof}

Now, using Definition 1.6, we construct a specialisation operator
$sp:g_{n}^{r,ext}\rightarrow g_{n}^{r}$. We first require the following simple
lemma;\\

\begin{lemma}
Let $C\subset P^{w}(L)$ be a projective algebraic curve and let
$C^{ext}\subset P^{w}(K)$ be its nonstandard model. Let $p'\in
C^{ext}$ be a non-singular point, with specialisation $p\in C$.
Then there exists a unique branch $\gamma_{p}^{j}$ such that
$p'\in\gamma_{p}^{j}$.

\end{lemma}

\begin{proof}
We may assume that $p'\neq p$, otherwise $p$ would be non-singular
and, by Lemma 5.4 of \cite{depiro6}, would be the origin of a
single branch $\gamma_{p}$. Let $(C^{ns},\Phi)$ be a non-singular
model of $C$, then $p'$ must belong to the canonical set
$V_{[\Phi]}$, hence there exists a unique $p''\in C^{ns}$ such
that $\Phi(p'')=p'$. By properties of specialisations, $p''\in
C^{ns}\cap{\mathcal V}_{p_{j}}$ for some
$p_{j}\in\Gamma_{[\Phi]}(x,p)$. Hence, by definition of a branch
given in Definition 5.15 of \cite{depiro6}, we must have that
$p'\in\gamma_{p}^{j}$. The uniqueness statement follows as well.
\end{proof}

We now make the following definition;\\

\begin{defn}
Let $C\subset P^{w}(L)$ be a projective algebraic curve and let
$C^{ext}\subset P^{w}(K)$ be its non-standard model. Given a
$g_{n}^{r}$ on $C$ with extension $g_{n}^{r,ext}$ on $C^{ext}$,
we define the specialisation operator;\\

$sp:g_{n}^{r,ext}\rightarrow g_{n}^{r}$\\

by;\\

$sp(\gamma_{p'})=\gamma_{p}^{j}$, for $p'\in NonSing(C^{ext})$ and $\gamma_{p}^{j}$ as in Lemma 1.8.\\

$sp(\gamma_{p}^{j})=\gamma_{p}^{j}$, for $p\in
Sing(C^{ext})=Sing(C)$ and
$\{\gamma_{p}^{1},\ldots,\gamma_{p}^{j},\ldots,\gamma_{p}^{s}\}$
\indent\indent\indent\indent\indent\indent\  enumerating the branches at $p$.\\

$sp(n_{1}\gamma_{p_{1}}^{j_{1}}+\ldots+n_{r}\gamma_{p_{r}}^{j_{r}})=n_{1}sp(\gamma_{p_{1}}^{j_{1}})+\ldots+n_{r}sp(\gamma_{p_{r}}^{j_{r}})$,\\

for a linear combination of branches with $n_{1}+\ldots+n_{r}=n$\\

\end{defn}

It is also a remarkable fact that, after introducing the notion of
a branch, the specialisation operator $sp$ is well defined. This is the content of the following lemma;\\

\begin{lemma}
Let hypotheses be as in the previous definition, then, if $W$ is a
weighted set belonging to $g_{n}^{r,ext}$, its specialisation
$sp(W)$ belongs to $g_{n}^{r}$.

\end{lemma}

\begin{proof}
We may assume that there exists a linear system $\Sigma$, having
finite intersection with $C$, such that $dimension(\Sigma)=r$ and
$degree(\Sigma)=n_{1}$, with the $g_{n}^{r}$ and $g_{n}^{r,ext}$
both defined by $Series(\Sigma)$, after removing some fixed point
contribution $W_{0}$ of multiplicity $n_{1}-n$. Let $W$ be a
weighted set of the $g_{n}^{r,ext}$, then $W\cup
W_{0}=(C\sqcap\phi_{\lambda'})$, for some unique $\lambda'\in
Par_{\Sigma}$. We claim that $sp(W\cup
W_{0})=C\sqcap\phi_{\lambda}$, for the specialisation $\lambda\in
Par_{\Sigma}$ of $\lambda'$ $(*)$. As $sp(W_{0})=W_{0}$, it then
follows immediately from linearity of $sp$, that $sp(W)$ belongs
to the $g_{n}^{r}$ as required. We now show $(*)$. Let $p\in C$ and
let $\gamma_{p}$ be a branch centred at $p$. By $\gamma_{p}^{ext}$, we mean the branch
at $p$, where $p$ is considered as an element of $C^{ext}$. We now claim that;\\

$I_{\gamma_{p}}(C,\phi_{\lambda})=I_{\gamma_{p}^{ext}}(C,\phi_{\lambda'})+\sum_{p'\in(\gamma_{p}\setminus
p)}I_{\gamma_{p'}^{ext}}(C,\phi_{\lambda'})$ $(**)$\\

Let $(C^{ns},\Phi)\subset P^{w'}(L)$ be a non-singular model of
$C$, such that $\gamma_{p}$ corresponds to $C^{ns,ext}\cap
{\mathcal V}_{q}$, where $q\in\Gamma_{[\Phi]}(x,p)$ and ${\mathcal
V}_{q}$ is defined relative to the specialisation from $P(K)$ to
$P(L)$. Let $C^{ns,ext,ext}\subset P^{w'}(K')$ be a non-standard
model of $C^{ns,ext}$, such that $\gamma_{q}^{ext}$ corresponds to
$C^{ns,ext,ext}\cap {\mathcal V}_{q}$, where ${\mathcal V}_{q}$ is
defined relative to the specialisation from $P(K')$ to $P(K)$.
Then, for $p'\in {(\gamma_{p}\setminus p)}$, we can find $q'\in
{\mathcal V}_{q}\cap C^{ns,ext}$ such that $\gamma_{p'}$
corresponds to ${\mathcal V}_{q'}\cap C^{ns,ext,ext}$. We may
choose a suitable presentation $\Phi_{\Sigma_{1}}$ of $\Phi$, such
that $Base(\Sigma_{1})$ is disjoint from $\Gamma_{[\Phi]}(x,p)$,
and, therefore, disjoint from $\Gamma_{[\Phi]}(x,p')$, for
$p'\in{(\gamma_{p}\setminus p)}$. Let
$\{\overline{\phi_{\lambda}}\}$ denote the lifted family of on
$C^{ns}$ from the presentation $\Phi_{\Sigma'}$. In this case,
we have, by results of \cite{depiro6}, that;\\

$I_{\gamma_{p}}(C,\phi_{\lambda})=I_{q}(C^{ns},\overline{\phi_{\lambda}})$\\

$I_{\gamma_{p}^{ext}}(C,\phi_{\lambda'})=I_{q}(C^{ns},\overline{\phi_{\lambda'}})$\\

$I_{\gamma_{p'}^{ext}}(C,\phi_{\lambda'})=I_{q'}(C^{ns},\overline{\phi_{\lambda'}})$ $(1)$\\

By summability of specialisation, see \cite{depiro4};\\

$I_{q}(C^{ns},\overline{\phi_{\lambda}})=I_{q}(C^{ns},\overline{\phi_{\lambda'}})+\sum_{q'\in
C^{ns}\cap{({\mathcal
V}_{q}}\setminus q)}I_{q'}(C^{ns},\overline{\phi_{\lambda'}})$ $(2)$\\

Combining $(1)$ and $(2)$, the result $(**)$ follows, as required.
Now, suppose that a branch $\gamma_{p}$ occurs with non-trivial
multiplicity in\\ $sp(C\sqcap\phi_{\lambda'})$. By Definition 1.9,
the contribution must come from either
$I_{\gamma_{p}^{ext}}(C,\phi_{\lambda'})$ or
$I_{\gamma_{p'}^{ext}}(C,\phi_{\lambda'})$, for some
$p'\in({\gamma_{p}\setminus p})$. Applying $sp$ to $(**)$, one
sees that the branch $\gamma_{p}$ occurs with multiplicity
$I_{\gamma_{p}}(C,\phi_{\lambda})$. It follows that
$sp(C\sqcap\phi_{\lambda'})=C\sqcap\phi_{\lambda}$, hence $(*)$ is
shown. The lemma then follows.

\end{proof}

We can now reformulate the results of Section 2 and Section 5 of
\cite{depiro6} in the language of this refined theory of $g_{n}^{r}$. We first make the
following definition;\\

\begin{defn}
Let $C\subset P^{w}$ be a projective algebraic curve and let a
$g_{n}^{r}$ be given on $C$. Let $W$ be a weighted set in this
$g_{n}^{r}$ or its extension $g_{n}^{r,ext}$ and let $\gamma_{p}$
be a branch centred at $p$. Then we say
that;\\

$\gamma_{p}$ is $s$-fold ($s$-plo) for $W$ if it appears with
multiplicity at least $s$.\\

$\gamma_{p}$ is multiple for $W$ if it appears with multiplicity
at least $2$.\\

$\gamma_{p}$ is simple for $W$ if it is not multiple.\\

$\gamma_{p}$ is counted (contato) $s$-times in $W$ if it appears
with multiplicity exactly $s$.\\

$\gamma_{p}$ is a base branch of the $g_{n}^{r}$ if it
appears in \emph{every} weighted set.\\

$\gamma_{p}$ is $s$-fold for the $g_{n}^{r}$ if it is $s$-fold in
$W$ for \emph{every} weighted set $W$ of the $g_{n}^{r}$.\\

$\gamma_{p}$ is counted $s$-times for the $g_{n}^{r}$ if it is
$s$-fold for the $g_{n}^{r}$ and is counted $s$-times in some
weighted set $W$ of the $g_{n}^{r}$.
\end{defn}

We then have the following;\\

\begin{theorem}{Local Behaviour of a $g_{n}^{r}$}\\

Let $C$ be a projective algebraic curve and let a $g_{n}^{r}$ be
given on $C$. Let $\gamma_{p}$ be a branch centred at $p$, such
that $\gamma_{p}$ is counted $s$-times for the $g_{n}^{r}$. If
$\gamma_{p}$ is counted $t$ times in a given weighted set $W$,
then there exists a weighted set $W'$ in $g_{n}^{r,ext}$ such that
$sp(W')=W$ and $sp^{-1}(t\gamma_{p})$ consists of the branch
$\gamma_{p}$ counted $s$-times and $t-s$ other distinct branches
$\{\gamma_{p_{1}},\ldots,\gamma_{p_{t-s}}\}$, each counted once in
$W'$.
\end{theorem}

\begin{proof}
Without loss of generality, we may assume that the $g_{n}^{r}$ is
defined by a linear system $\Sigma$ of dimension $r$, having
finite intersection with $C$. Let $W$ be the weighted set defined
by $\phi_{\lambda}$ in $\Sigma$. Suppose that $s=0$, then
$\gamma_{p}$ is not a base branch for $\Sigma$. Hence, by Lemma
5.25 of \cite{depiro6}, we can find $\lambda'\in{\mathcal
V}_{\lambda}$, generic in $Par_{\Sigma}$, and distinct
$\{p_{1},\ldots,p_{t}\}=C^{ext}\cap\phi_{\lambda'}\cap
({\gamma_{p}\setminus p})$ such that the intersections at these
points are transverse. Let $W'$ be the weighted set defined by
$\phi_{\lambda'}$ in $g_{n}^{r,ext}$. By the proof of $(*)$ in
Lemma 1.10, we have that $sp(W')=W$. By the construction of $sp$
in Definition 1.9, we have that $sp^{-1}(t\gamma_{p})$ consists of
the distinct branches $\{\gamma_{p_{1}},\ldots,\gamma_{p_{t}}\}$,
each counted once in $W'$. If $s\geq 1$, then $\gamma_{p}$ is a
base branch for $\Sigma$. By Lemma 5.27 of \cite{depiro6}, we have
that
$I^{\Sigma,mobile}_{italian}(p,\gamma_{p},C,\phi_{\lambda})=t-s$.
The result then follows by application of Lemma 5.28 in
\cite{depiro6} and the argument given above.
\end{proof}

We now note the following;\\

\begin{lemma}
Let a $g_{n}^{r}$ be given on a projective algebraic curve $C$.
Let $W_{0}$ be \emph{any} weighted set on $C$ with total
multiplicity $n'$. Then the collection of weighted sets given by
$\{W\cup W_{0}\}$ for the weighted sets $W$ in the $g_{n}^{r}$
defines a $g_{n+n'}^{r}$.

\end{lemma}

\begin{proof}
Let the original $g_{n}^{r}$ be obtained from a linear system
$\Sigma$ of dimension $r$ and degree $n''$, having finite
intersection with $C$, after removing some fixed point
contribution $J$ of total multiplicity $n''-n$. Let
$\{\phi_{0},\ldots,\phi_{r}\}$ be a basis for $\Sigma$ and let
$\{n_{1}\gamma_{p_{1}}^{j_{1}},\ldots,n_{m}\gamma_{p_{m}}^{j_{m}}\}$
be the branches appearing in $W_{0}$ with total multiplicity
$n_{1}+\ldots+n_{m}=n'$ $(\dag)$. Let $\{H_{1},\ldots,H_{m}\}$ be
hyperplanes passing through the points $\{p_{1},\ldots,p_{m}\}$
and let $G$ be the algebraic form of degree $n'$ defined by
$H_{1}^{n_{1}}\centerdot\ldots\centerdot H_{m}^{n_{m}}$. Let
$\Sigma'$ be the linear system of dimension $r$ defined by the
basis\\
 $\{G\centerdot\phi_{0},\ldots,G\centerdot\phi_{r}\}$. As
we may assume that $C$ is not contained in any hyperplane section,
$\Sigma'$ has finite intersection with $C$. We claim that
$g_{n''}^{r}(\Sigma)\subset g_{n''+n'deg(C)}^{r}(\Sigma')$, in the
sense that every weighted set $W_{\lambda}$ defined by
$g_{n''+n'deg(C)}^{r}(\Sigma')$ is obtained from the corresponding
$V_{\lambda}$ in $g_{n''}^{r}(\Sigma)$ by adding a \emph{fixed}
weighted set $W_{1}\supset W_{0}$ of total multiplicity $n'deg(C)$
$(*)$. The proof then follows as we can recover the original
$g_{n}^{r}$ by removing the fixed point contribution
$J\cup(W_{1}\setminus W_{0})$ from
$g_{n''+n'deg(C)}^{r}(\Sigma')$. In order to show $(*)$, let
$W_{1}$ be the weighted set defined by $C\sqcap G$. By the
branched version of the Hyperspatial Bezout Theorem, see Theorem
5.13 of \cite{depiro6}, this has total multiplicity $n'deg(C)$. We
claim that $W_{0}\subset W_{1}$ $(**)$. Let $\gamma_{p}^{j}$ be a
branch appearing in $(\dag)$ with multiplicity $s$. By
construction, we can factor $G$ as $H^{s}\centerdot R$, where $H$
is a hyperplane
passing through $s$. We need to show that;\\

$I_{\gamma_{p}^{j}}(C,H^{s}\centerdot R)\geq s$\\

or equivalently,\\

$I_{p_{j}}(C^{ns},\overline{H^{s}\centerdot R})=I_{p_{j}}(C^{ns},{\overline H}^{s}\centerdot\overline{R})\geq s$\\

for a suitable presentation $C^{ns}$ of a non-singular model of
$C$, see Lemma 5.12 of \cite{depiro6}, where we have used the
"lifted" form notation there. Using the method of conic
projections, see section 4 of \cite{depiro6}, we can find a plane
projective curve $C'$ birational to $C^{ns}$, such that the point
$p_{j}$ corresponds to a
non-singular point $q$ of $C'$ and;\\

$I_{p_{j}}(C^{ns},\overline{H}^{s}\centerdot\overline
R)=I_{q}(C',\overline{{\overline
H}^{s}\centerdot\overline{R}})=I_{q}(C',\overline{{\overline
H}}^{s}\centerdot\overline{\overline{R}})$\\

The result then follows by results of the paper \cite{depiro5} for
the intersections of plane projective curves. This shows $(**)$.
We now need to prove that, for an algebraic form $\phi_{\lambda}$
in $\Sigma$ and a branch $\gamma_{p}^{j}$ of $C$;\\

$I_{\gamma_{p}^{j}}(C,\phi_{\lambda}\centerdot
G)=I_{\gamma_{p}^{j}}(C,\phi_{\lambda})+I_{\gamma_{p}^{j}}(C,G)$\\

This follows by exactly the same argument, reducing to the case of
intersections between plane projective curves and using the
results of \cite{depiro5}. The result is then shown.

\end{proof}

\begin{theorem}{Birational Invariance of a $g_{n}^{r}$}\\

Let $\Phi:C_{1}\leftrightsquigarrow C_{2}$ be a birational map
between projective algebraic curves. Then, given a $g_{n}^{r}$ on
$C_{2}$, there exists a canonically defined $g_{n}^{r}$ on
$C_{1}$, depending only on the class $[\Phi]$ of the birational
map. Conversely, given a $g_{n}^{r}$ on $C_{1}$, there exists a
canonically defined $g_{n}^{r}$ on $C_{2}$, depending only on the
class $[\Phi^{-1}]$ of the birational map. Moreover, these
correspondences are inverse.

\end{theorem}

\begin{proof}
By Lemma 5.7 of \cite{depiro6}, $[\Phi]$ induces a bijection;\\

$[\Phi]^{*}:\bigcup_{O\in C_{2}}\gamma_{O}\rightarrow\bigcup_{O\in
C_{1}}\gamma_{O}$\\

of branches, with inverse given by ${[\Phi^{-1}]}^{*}$.

Then $[\Phi]^{*}$ extends naturally to a map on weighted sets of
degree $n$ by the formula;\\

$[\Phi]^{*}(n_{1}\gamma_{p_{1}}^{j_{1}}+\ldots+n_{r}\gamma_{p_{r}}^{j_{r}})=n_{1}[\Phi]^{*}(\gamma_{p_{1}}^{j_{1}})+\ldots+n_{r}[\Phi]^{*}(\gamma_{p_{r}}^{j_{r}})$\\

for a linear combination of branches
$\{\gamma_{p_{1}}^{j_{1}},\ldots,\gamma_{p_{r}}^{j_{r}}\}$ with\\
$n=n_{1}+\ldots+n_{r}$. Therefore, given a $g_{n}^{r}$ on $C_{2}$,
we obtain a canonically defined collection $[\Phi]^{*}(g_{n}^{r})$
of weighted sets on $C_{1}$ of degree $n$ $(*)$. It is trivial to
see that $[\Phi^{-1}]^{*}\circ[\Phi]^{*}(g_{n}^{r})$ recovers the
original $g_{n}^{r}$ on $C_{2}$, by the fact the map $[\Phi]^{*}$
on branches is invertible, with inverse given by
$[\Phi^{-1}]^{*}$. Let $C^{ns}$ be a non-singular model of $C_{1}$
and $C_{2}$ with morphisms $\Phi_{1}:C^{ns}\rightarrow C_{1}$ and
$\Phi_{2}:C^{ns}\rightarrow C_{2}$ such that
$\Phi\circ\Phi_{1}=\Phi_{2}$ and $\Phi^{-1}\circ\Phi_{2}=\Phi_{1}$
as birational maps (see the proof of Lemma 5.7 in \cite{depiro6}).
We then have that
$[\Phi]^{*}(g_{n}^{r})=[\Phi_{1}^{-1}]^{*}\circ[\Phi_{2}]^{*}(g_{n}^{r})$.
It remains to prove that this collection given by $(*)$ defines a
$g_{n}^{r}$ on $C_{1}$. We will prove first that
$[\Phi_{2}]^{*}(g_{n}^{r})$ defines a $g_{n}^{r}$ on $C^{ns}$
$(\dag)$. Let the original $g_{n}^{r}$ on $C_{2}$ be defined by a
linear system $\Sigma$, having finite intersection with $C_{2}$,
such that $dimension(\Sigma)=r$ and $degree(\Sigma)=n'$, after
removing some fixed branch contribution of multiplicity $n'-n$. We
may assume that $n'=n$, as if the fixed branch contribution in
question is given by $W_{0}$ and $g_{n}^{r}\cup W_{0}=g_{n'}^{r}$,
then $[\Phi_{2}]^{*}(g_{n}^{r})\cup
[\Phi_{2}]^{*}(W_{0})=[\Phi_{2}]^{*}(g_{n'}^{r})$, hence it is
sufficient to prove that $[\Phi_{2}]^{*}(g_{n'}^{r})$ defines a
$g_{n'}^{r}$. Let $W_{1}$ be the fixed branch contribution of the
$g_{n}^{r}$ on $C_{2}$ and let $g_{n''}^{r}\subset g_{n}^{r}$ be
obtained by removing this fixed branch contribution. It will be
sufficient to prove that $[\Phi_{2}]^{*}(g_{n''}^{r})$ defines a
$g_{n''}^{r}$ on $C^{ns}$ as
$[\Phi_{2}]^{*}(g_{n}^{r})=[\Phi_{2}]^{*}(g_{n''}^{r})\cup
[\Phi_{2}]^{*}(W_{1})$ and we may then use Lemma 1.13. Let
$\Phi_{\Sigma_{1}}$ and $\Phi_{\Sigma_{2}}$ be presentations of
the morphisms $\Phi_{1}$ and $\Phi_{2}$. We may assume that $
Base(\Sigma_{1})$ and $Base(\Sigma_{2})$ are disjoint. Let
$\{\overline{\phi_{\lambda}}\}$ denote the lifted family of forms
on $C^{ns}$, defined by the linear system $\Sigma$ and the
presentation $\Phi_{\Sigma_{2}}$. We claim that
$[\Phi_{2}]^{*}(g_{n''}^{r})$ is defined by this system after
removing its fixed branch contribution. In order to see this, we first show
that for any branch $\gamma_{p}^{j}$ of $C$;\\

$I_{\gamma_{p}^{j}}^{\Sigma,mobile}(C,\phi_{\lambda})=I_{p_{j}}^{\Sigma,mobile}(C^{ns},\overline{\phi_{\lambda}})$ $(*)$ $(1)$\\

where $p_{j}$ corresponds to $\gamma_{p}^{j}$ in the fibre
$\Gamma_{[\Phi_{2}]}(x,p)$, see Section 5 of \cite{depiro6}.
By Definition 2.20 and Lemma 5.23 of \cite{depiro6}, we have that;\\

$I_{p_{j}}^{\Sigma,mobile}(C^{ns},\overline{\phi_{\lambda}})=Card(C^{ns}\cap
({\mathcal V}_{p_{j}}\setminus
p_{j})\cap\overline{\phi_{\lambda'}})$ for
$\lambda'\in {\mathcal V}_{\lambda}$, generic\\
\indent \ \ \ \ \ \ \ \ \ \ \ \ \ \ \ \ \ \ \ \ \ \ \ \ \ \ \ \ \ \ \ \ \ \ \ \ \ \ \ \ \ \ \ \ \ \ \ \ \ \ \ \ \ \ \ \ \ \ \ \ \ \ \ \ in $Par_{\Sigma}$\\

$I_{\gamma_{p}^{j}}^{\Sigma,mobile}(C,\phi_{\lambda})=Card(C\cap
(\gamma_{p}^{j}\setminus p)\cap\phi_{\lambda'})$ for $\lambda'\in
{\mathcal V}_{\lambda}$, generic in $Par_{\Sigma}$\\

As $(\gamma_{p}^{j}\setminus p)$ is in biunivocal correspondence
with ${({\mathcal V}_{p_{j}}\setminus p_{j})}$ under the morphism
$\Phi_{2}$, we obtain immediately the result $(*)$. Now, using
Lemma 5.27 of \cite{depiro6}, we have that, if $\gamma_{p}^{j}$
appears in a weighted set $W_{\lambda}$ of the $g_{n''}^{r}$ with
multiplicity $s$, then the corresponding branch $\gamma_{p_{j}}$
appears in the weighted set $[\Phi_{2}]^{*}(W_{\lambda})$ with
multiplicity equal to
$s=I_{p_{j}}^{mobile}(C^{ns},\overline{\phi_{\lambda}})$. Again,
using Lemma 5.27 of \cite{depiro6}, we obtain that
$[\Phi_{2}]^{*}(W_{\lambda})$ is given by
$C^{ns}\sqcap\overline{\phi_{\lambda}}$, after removing all fixed
point contributions of the linear system $\Sigma$. We, therefore,
obtain that $[\Phi_{2}]^{*}(g_{n''}^{r})$ is defined by $\Sigma$,
after removing all fixed point contributions, as required. This
proves $(\dag)$. We now claim that, for the given $g_{n}^{r}$ on
$C^{ns}$, $[\Phi^{-1}]^{*}(g_{n}^{r})$ defines a $g_{n}^{r}$ on
$C_{1}$, $(\dag\dag)$. Let $\Phi_{\Sigma_{3}}$ be a presentation
of the morphism $\Phi^{-1}$. If $\phi_{\lambda}$ is a form
belonging to the linear system $\Sigma$ defined on $C^{ns}$, using
the presentations $\Phi_{\Sigma_{1}}$ and $\Phi_{\Sigma_{3}}$ of
$\Phi$ and $\Phi^{-1}$, we obtain a lifted form
$\overline{\phi_{\lambda}}$ on $C_{1}$ and a lifted form
$\overline{\overline{\phi_{\lambda}}}$ on $C^{ns}$ again. We now
claim that, for $p\in C^{ns}$;\\

$I_{p}^{\Sigma,mobile}(C^{ns},\phi_{\lambda})=I_{p}^{\Sigma,mobile}(C^{ns},\overline{\overline{\phi_{\lambda}}})$ $(2)$\\

In order to see this, first observe that we can obtain the lifted
system of forms $\{\overline{\overline{\phi_{\lambda}}}\}$
directly from the linear system $\Sigma_{4}$, obtained by
composing bases of the linear systems $\Sigma_{1}$ and
$\Sigma_{3}$. The corresponding morphism $\Phi_{\Sigma_{4}}$
defines a birational map of $C^{ns}$ to itself, which is
equivalent to the identity map $Id$. Now the result follows
immediately from Definition 2.20 and Lemma 2.16 of \cite{depiro6},
both multiplicities are witnessed inside the canonical set $W$ of
$\Phi_{\Sigma_{4}}$, which, in this case, is just the domain of
definition of $\Phi_{\Sigma_{4}}$ on $C^{ns}$, see Definition 1.30
of \cite{depiro6}. Now, returning to the proof of $(\dag\dag)$, we
may suppose that the given $g_{n}^{r}$ on $C^{ns}$ is defined by
the linear system $\Sigma$, after removing all fixed point
contributions. Combining $(1)$ and $(2)$, we have that, for a branch
 $\gamma_{p}^{j}$ of $C_{1}$;\\

$I_{\gamma_{p}^{j}}^{\Sigma,mobile}(C_{1},\overline{\phi_{\lambda}})=I_{p_{j}}^{\Sigma,mobile}(C^{ns},\overline{\overline{\phi_{\lambda}}})=I_{p_{j}}^{\Sigma,mobile}(C^{ns},\phi_{\lambda})$\\

The result now follows from the same argument as above, using
Lemma 5.27 of \cite{depiro6}. This completes the theorem.

\begin{rmk}
Using the quoted Theorem 1.33 of \cite{depiro6}, one can use the
Theorem to reduce calculations involving $g_{n}^{r}$ on projective
algebraic curves to calculations on plane projective curves. We
will use this property extensively in the following sections.
\end{rmk}
\end{proof}

We finally note the following;\\

\begin{lemma}
For a given $g_{n}^{r}$, we always have that $r\leq n$.
\end{lemma}

\begin{proof}
The proof is almost identical to Lemma 2.24 of \cite{depiro6}. We
leave the details to the reader.

\end{proof}

\end{section}
\begin{section}{A Theory of Complete Linear Series on an Algebraic
Curve}

We now develop further the theory of $g_{n}^{r}$ on an algebraic
curve $C$, analogously to classical results for divisors on
non-singular algebraic curves. We will first assume that $C$ is a
plane projective algebraic curve, defined by some homogeneous
polynomial $F(X,Y,Z)$. Without loss of generality, we will use the
coordinates $x=X/Z$ and $y=Y/Z$ for local calculations on the
curve $C$, defined in this system by $f(x,y)=0$. Using Theorem
1.14, we will later derive general results for $g_{n}^{r}$
on an algebraic curve from the corresponding calculations for the plane case.\\

We consider first the case when $r=1$. By results of the previous
section, a $g_{n}^{1}$ is defined by a pencil $\Sigma$ of
algebraic curves $\{\phi(x,y)+\lambda\phi'(x,y)=0\}_{\lambda\in
P^{1}}$ (in affine coordinates), after removing some fixed point
contribution, where, by convention, we interpret the algebraic
curve $\phi(x,y)+\infty\phi'(x,y)=0$ to be $\phi'(x,y)=0$. We
assume that the $g_{n}^{1}$ is, in fact, cut out by this pencil.
Now suppose that $\gamma_{p}$ is a branch of $C$. We may assume
that $p$ corresponds to the origin $O$ of the affine coordinate
system $(x,y)$, (use a linear transformation and the result of
Lemma 2.1) By Theorem 6.1 of \cite{depiro6}, we can find algebraic
power series $\{x(t),y(t)\}$, with $x(t)=y(t)=0$, parametrising
$\gamma_{p}$. We can now substitute the power series in order to
obtain a formal expression of the form;\\

${\phi(x(t),y(t))\over\phi'(x(t),y(t))}={t^{i}u(t)\over
t^{j}v(t)}=t^{i-j}u(t)v(t)^{-1}$, where
$\{u(t),v(t),u(t)v(t)^{-1}\}$\\
\indent \ \ \ \ \ \ \ \ \ \ \ \ \ \ \ \ \ \ \ \ \ \ \ \ \ \ \ \ \
\ \ \ \ \ \ \ \ \ \ \ \ \ \ \ \ \ \ \ \  are units in $L[[t]]$.

We then define;\\

 $ord_{\gamma_{p}}({\phi\over\phi'})=i-j$,\\

$val_{\gamma_{p}}({\phi\over\phi'})=0$,\ \ \ \ \ \ \ \ \ \ \ \ \ \ \ \ \ \ \ \ \ \ if $i>j$, (${\phi\over\phi'}$ has a zero of order $i-j$)\\

 $ord_{\gamma_{p}}({\phi\over\phi'})=j-i$,\\

 $val_{\gamma_{p}}({\phi\over\phi'})=\infty$,\ \ \ \ \ \ \ \ \ \ \ \ \ \ \ \ \ \ \ \ \ if $i<j$, (${\phi\over\phi'}$ has a pole of order $j-i$)\\

 $ord_{\gamma_{p}}({\phi\over\phi'})=ord_{t}(h(t)-h(0))$,\\

 $val_{\gamma_{p}}({\phi\over\phi'})=h(0)$,\ \ \ \ \ \ \ \ \ \ \ \ \ \ \ \ \ \ \ if $i=j$ and $h(t)=u(t)v(t)^{-1}$\\

Observe that in all cases, $ord_{\gamma_{p}}$ gives a
\emph{positive} integer, while $val_{\gamma_{p}}$ determines an
element of $P^{1}$. In order to see that this construction does
not depend on the particular power series representation of
the branch, we require the following lemma;\\

\begin{lemma}
Let $\{C,\gamma_{p},\phi,\phi',g_{n}^{1},\Sigma\}$ be as defined above, then;\\

$ord_{\gamma_{p}}({\phi\over\phi'})=I_{\gamma_{p}}(C,\phi-\lambda\phi')$,
\ \ \ \ \ \ if $\gamma_{p}$ is not a base branch for the
$g_{n}^{1}$ \indent \ \ \ \ \ \ \ \ \ \ \ \ \ \ \ \ \ \ \ \ \ \ \
\ \ \ \ \ \
\ \ \ \ \ \ \ \ \ \ \ and ${\phi\over\phi'}(p)=val_{\gamma_{p}}({\phi\over\phi'})=\lambda$.\\

$ord_{\gamma_{p}}({\phi\over\phi'})=I_{\gamma_{p}}^{\Sigma,mobile}(C,\phi-\lambda\phi')$,
if $\gamma_{p}$ is a base branch for the $g_{n}^{1}$ and\\
\indent \ \ \ \ \ \ \ \ \ \ \ \ \ \ \ \ \ \ \  \ \ \ \ \ \ \  \ \
\  \ \ \ \ \ \ \ \ \ \ \ $\lambda=val_{\gamma_{p}}({\phi\over\phi'})$ is unique such that,\\
\indent\ \ \ \ \ \ \ \ \ \ \ \ \ \ \ \ \ \ \ \ \ \ \ \ \ \ \ \ \ \ \ \ \ \ \ \ \ \ \ \  for $\mu\neq\lambda$;\\
\indent \ \ \ \ \ \ \ \ \ \ \ \ \ \ \ \ \ \ \ \ \ \ \ \ \ \ \ \ \
 \ \ \ \ \ \ \ \ \ \ \
$I_{\gamma_{p}}(C,\phi-\lambda\phi')>I_{\gamma_{p}}(C,\phi-\mu\phi').$

\end{lemma}

\begin{proof}
Suppose that $\gamma_{p}$ is not a base branch for the
$g_{n}^{1}$, then ${\phi\over\phi'}(p)=\lambda$ is well defined,
if we interpret $(c/0)=\infty$ for $c\neq 0$, and
$\phi-\lambda\phi'$ is the unique curve in the pencil passing
through $p$. It is trivial to check, using the facts that
$\phi(p)=\phi(x(0),y(0))$ and $\phi'(p)=\phi'(x(0),y(0))$, that,
in all cases, $val_{\gamma_{p}}({\phi\over\phi'})=\lambda$ as
well. By Theorem 6.1 of \cite{depiro6}, we have that;\\

$I_{\gamma_{p}}(C,\phi-\lambda\phi')=ord_{t}[(\phi-\lambda\phi')(x(t),y(t))]$\\

If $\lambda=0$, then $\phi(p)=0$ and $\phi'(p)\neq 0$, hence, by a
straightforward algebraic calculation,
$\phi(x(t),y(t))=t^{i}u(t)$, for some $i\geq 1$, and
$\phi'(x(t),y(t))=v(t)$ for $\{u(t),v(t)\}$ units in $L[[t]]$.
Therefore,
$ord_{\gamma_{p}}({\phi\over\phi'})=ord_{t}\phi(x(t),y(t))$ and
the result follows.\\

If $\lambda=\infty$, then $\phi(p)\neq 0$ and $\phi(p)=0$, hence,
$\phi(x(t),y(t))=u(t)$ and $\phi'(x(t),y(t))=t^{j}v(t)$, for some
$j\geq 1$, and $\{u(t),v(t)\}$ units in $L[[t]]$. Therefore,
$ord_{\gamma_{p}}({\phi\over\phi'})=ord_{t}\phi'(x(t),y(t))$ and
the result follows.\\

If $\lambda\neq\{0,\infty\}$, then $\phi(x(t),y(t))=u(t)$ and
$\phi'(x(t),y(t))=v(t)$ with $\{u(t),v(t)\}$ units in $L[[t]]$. As
$v(t)$ is a unit in $L[[t]]$, we have that;\\

$ord_{t}({u(t)\over v(t)}-{u(0)\over
v(0)})=ord_{t}(v(t)({u(t)\over
v(t)}-{u(0)\over v(0)}))=ord_{t}(u(t)-{u(0)\over v(0)}v(t))$\\

Hence, by definition of $ord_{\gamma_{p}}$;\\

$ord_{\gamma_{p}}({\phi\over\phi'})=ord_{t}[(\phi-\lambda\phi')(x(t),y(t))]$\\

and the result follows.\\

Now suppose that $\gamma_{p}$ is a base branch for the
$g_{n}^{1}$, then $\phi(p)=\phi'(p)=0$ and we have that
$\phi(x(t),y(t))=t^{i}u(t)$ and $\phi'(x(t),y(t))=t^{j}v(t)$, for
some $i,j\geq 1$ and $\{u(t),v(t)\}$ units in $L[[t]]$. Again, we
divide the proof into the following cases;\\

$i>j$. In this case, by definition, $val_{\gamma_{p}}({\phi\over\phi'})=0$. We compute;\\

$ord_{t}(\phi(x(t),y(t))-\lambda\phi'(x(t),y(t)))=ord_{t}(t^{i}u(t)-\lambda
t^{j}v(t))$\\

When $\lambda=0$, we obtain, by Theorem 6.1 of \cite{depiro6},
that $I_{\gamma_{p}}(C,\phi)=i$ and, for $\lambda\neq 0$, that
$I_{\gamma_{p}}(C,\phi-\lambda\phi')=j$. Using Lemma 5.27 of
\cite{depiro6}, we obtain that
$I_{\gamma_{p}}^{\Sigma,mobile}(C,\phi)=i-j=ord_{\gamma_{p}}({\phi\over\phi'})$, as required.\\

$i<j$. In this case, by definition,
$val_{\gamma_{p}}({\phi\over\phi'})=\infty$. The computation for
$ord_{\gamma_{p}}$ is similar, with the critical value
being $\lambda=\infty$.\\

$i=j$. We compute;\\

$ord_{t}(\phi(x(t),y(t))-\lambda\phi'(x(t),y(t)))=ord_{t}[t^{i}(u(t)-\lambda
v(t))]$\\

Again, there exists a unique value of $\lambda={u(0)\over
v(0)}=val_{\gamma_{p}}({\phi\over\phi'})\neq\{0,\infty\}$ such
that $ord_{t}(u(t)-\lambda v(t))=k\geq 1$. By the same calculation
as above, we have that
$I_{\gamma_{p}}^{\Sigma,mobile}(C,\phi-\lambda\phi')=k$, for this
critical value of $\lambda$. By a similar algebraic calculation to
the above, using the fact that $v(t)$ is a unit, we also compute
$ord_{\gamma_{p}}({\phi\over\phi'})=k$, hence the result
follows.\\
\end{proof}

We now show the following;\\

\begin{lemma}
Given any algebraic curve $C\subset P^{w}$, with function field
$L(C)$, for a non-constant rational function $f\in L(C)$ and a
branch $\gamma_{p}$, we can unambiguously define
$ord_{\gamma_{p}}(f)$ and $val_{\gamma_{p}}(f)$.
\end{lemma}

\begin{proof}
The proof is similar to the above. We may, without loss of
generality, assume that $p$ corresponds to the origin of a
coordinate system $(x_{1},\ldots,x_{w})$. Using Theorem 6.1 of
\cite{depiro6}, we can find algebraic power series
$(x_{1}(t),\ldots,x_{w}(t))$ parametrising the branch
$\gamma_{p}$. By the assumption that $f$ is non-constant, we can
find a representation of $f$ as a rational function
$\phi(x_{1},\ldots,x_{w})\over\phi'(x_{1},\ldots,x_{w})$ in this
coordinate system, such that the pencil $\Sigma$ defined by
$\{\phi,\phi'\}$ has finite intersection with $C$, hence defines a
$g_{n}^{1}$. Using the method above, we can define
$ord_{\gamma_{p}}({\phi\over\phi'})$ and
$val_{\gamma_{p}}({\phi\over\phi'})$ for this representation. The
proof of Lemma 2.1 shows that these are defined independently of
the particular power series parametrising the branch. We need to
check that they are also defined independently of the particular
representation of $f$. Suppose that
$\{\phi_{1},\phi_{2},\phi_{3},\phi_{4}\}$ are algebraic forms with
the property that
${\phi_{1}\over\phi_{2}}={\phi_{3}\over\phi_{4}}$ as rational
functions on $C$. We claim that, for any branch $\gamma_{p}$ of
$C$,
$ord_{\gamma_{p}}({\phi_{1}\over\phi_{2}})=ord_{\gamma_{p}}({\phi_{3}\over\phi_{4}})$
and
$val_{\gamma_{p}}({\phi_{1}\over\phi_{2}})=val_{\gamma_{p}}({\phi_{3}\over\phi_{4}})$,
$(*)$. In order to see this, let $U\subset NonSing(C)$ be an open
subset of $C$, on which ${\phi_{1}\over\phi_{2}}$ and
${\phi_{3}\over\phi_{4}}$ are defined and equal. Let $g_{n}^{1}$
and $g_{m}^{1}$ on $C$ be defined by the pencils
$\Sigma_{1}=\{\phi_{1}-\lambda\phi_{2}\}_{\lambda\in P^{1}}$ and
$\Sigma_{2}=\{\phi_{3}-\lambda\phi_{4}\}_{\lambda\in P^{1}}$. Let
$V=U\setminus Base(\Sigma_{1})\cup Base(\Sigma_{2})$. Then
$V\subset U$ is also an open subset of $C$, which we will refer to
as the canonical set. Now, suppose that $\gamma_{p}\subset V$. We
will prove $(*)$ for this branch. As both
${\phi_{1}\over\phi_{2}}$ and ${\phi_{3}\over\phi_{4}}$ are
defined and equal at $p$, using the argument in Lemma 2.1, we have
that
$val_{\gamma_{p}}({\phi_{1}\over\phi_{2}})=val_{\gamma_{p}}({\phi_{3}\over\phi_{4}})$.
It is therefore sufficient, again by Lemma
2.1, to show that;\\

$I_{\gamma_{p}}(C,\phi_{1}-\lambda\phi_{2})=I_{\gamma_{p}}(C,\phi_{3}-\lambda\phi_{4})$,
for ${\phi_{1}\over\phi_{2}}(p)={\phi_{3}\over\phi_{4}}(p)=\lambda$ $(\dag)$\\

Suppose that $I_{\gamma_{p}}(\phi_{1}-\lambda\phi_{2})=m$, then,
by Lemma 5.25 of \cite{depiro6}, we can find $\lambda'\in{\mathcal
V}_{\lambda}\cap P^{1}$ and $\{p_{1},\ldots,p_{m}\}=V\cap
{\mathcal V}_{p}\cap (\phi_{1}-\lambda'\phi_{2})=0$ witnessing
this multiplicity. As $\{p,p_{1},\ldots,p_{m}\}$ lie inside $V$,
we also have that $\{p_{1},\ldots,p_{m}\}\subset V\cap{\mathcal
V}_{p}\cap(\phi_{3}-\lambda'\phi_{4})=0$, hence
$I_{\gamma_{p}}(C,\phi_{3}-\lambda\phi_{4})\geq m$. The result
$(\dag)$ then follows from the converse argument.\\

Now, suppose that $\gamma_{p}$ is one of the finitely many
branches of $C$, not lying inside $V$. We will just consider the
case when $\gamma_{p}$ is a base branch for both the $g_{n}^{1}$
and the $g_{m}^{1}$ defined above, the other cases being similar.
In order to prove $(*)$ for this branch, it is sufficient, by
Lemma 2.1, to show that;\\

$I_{\gamma_{p}}^{\Sigma_{1},mobile}(C,\phi_{1}-\lambda\phi_{2})=I_{\gamma_{p}}^{\Sigma_{2},mobile}(C,\phi_{3}-\mu\phi_{4})$,
for the critical values\\
\indent \ \ \ \ \ \ \ \ \ \ \ \ \ \ \ \ \ \ \ \ \ \ \ \ \ \ \ \ \ \ \ \ \ \ \ \ \ \ \ \ \ \ \ \ \ \ \ \ \ \ \ \ \ \ \ \ \ \ \ \ \ \ $\{\lambda,\mu\}$\\

and that the critical values $\{\lambda,\mu\}$ coincide,
$(\dag\dag)$.\\

Using the argument to prove $(\dag)$, witnessing the corresponding
multiplicities in the canonical set $V$, it follows that for
$\emph{any}$ $\nu\in P^{1}$;\\

$I_{\gamma_{p}}^{\Sigma_{1},mobile}(C,\phi_{1}-\nu\phi_{2})=I_{\gamma_{p}}^{\Sigma_{2},mobile}(C,\phi_{3}-\nu\phi_{4})$, $(\dag\dag\dag)$\\

If the critical values $\{\lambda,\mu\}$ were distinct, we would
have that;\\

$I_{\gamma_{p}}^{\Sigma_{1},mobile}(C,\phi_{1}-\lambda\phi_{2})>I_{\gamma_{p}}^{\Sigma_{1},mobile}(C,\phi_{1}-\mu\phi_{2})$\\
\indent\ \ \ \ \ \ \ \ \ \ \ \ \ \ \ \ \ $||$\ \ \ \ \ \ \ \ \ \ \
\ \ \ \ \ \ \ \ \ \ \ \ \ \ \ \ \ \ \ \ $||$\ \ \ \ \ \ \ \ \
\\
\indent $I_{\gamma_{p}}^{\Sigma_{2},mobile}(C,\phi_{3}-\lambda\phi_{4})<I_{\gamma_{p}}^{\Sigma_{2},mobile}(C,\phi_{3}-\mu\phi_{4})$\\

which is clearly a contradiction. Hence, $\lambda=\mu$ and the
result $(\dag\dag)$ follows from $(\dag\dag\dag)$. The lemma is
shown.
\end{proof}

\begin{lemma}{Birational Invariance of $ord_{\gamma_{p}}$ and $val_{\gamma_{p}}$}\\

Let $\Phi:C_{1}\leftrightsquigarrow C_{2}$ be a birational map
between projective algebraic curves with corresponding
isomorphisms $\Phi^{*}:L(C_{2})\rightarrow L(C_{1})$ and\\
$[\Phi]^{*}:\bigcup_{p\in C_{2}}\gamma_{p}\rightarrow\bigcup_{q\in
C_{1}}\gamma_{q}$ . Then, for non-constant $f\in L(C_{2})$ and
$\gamma_{p}$ a branch of $C_{2}$,
$ord_{\gamma_{p}}(f)=ord_{[\Phi]^{*}\gamma_{p}}(\Phi^{*}f)$ and
$val_{\gamma_{p}}(f)=val_{[\Phi]^{*}\gamma_{p}}(\Phi^{*}f)$.

\end{lemma}

\begin{proof}
Let $f$ be represented as a rational function by
${\phi_{1}\over\phi_{2}}$, as in Lemma 2.2, and consider the
$g^{1}_{n}$ on $C_{2}$, defined by the linear system
$\Sigma=\{\phi_{1}-\lambda\phi_{2}\}_{\lambda\in P^{1}}$. Let
$\Phi_{\Sigma_{1}}$ be a presentation of the birational map
$\Phi$. Using this presentation, we may lift the system $\Sigma$
to a corresponding linear system
$\{\overline{\phi_{1}}-\lambda\overline{\phi_{2}}\}_{\lambda\in
P^{1}}$. It is trivial to check that $\Phi^{*}f$ is represented by
the rational function
${\overline{\phi_{1}}\over\overline{\phi_{2}}}$. The proof of
Theorem 1.14 shows that, for a branch $\gamma_{p}$ of $C_{2}$;\\

$I_{\gamma_{p}}^{\Sigma,mobile}(C_{2},\phi_{1}-\lambda\phi_{2})=I_{[\Phi]^{*}\gamma_{p}}^{\Sigma,mobile}(C_{1},\overline{\phi_{1}}-\lambda\overline{\phi_{2}})$, $(*)$\\

We now need to consider the following cases;\\

Case 1. $\gamma_{p}$ and $[\Phi]^{*}\gamma_{p}$ are not base
branches for $\Sigma$ on $C_{2}$ and $C_{1}$.\\

Case 2. $\gamma_{p}$ is not a base branch, but
$[\Phi]^{*}\gamma_{p}$ is a base branch for $\Sigma$ on\\
\indent \ \ \ \ \ \ \ \ \ \ \  $C_{2}$ and $C_{1}$.\\

Case 3. $\gamma_{p}$ is a base branch and $[\Phi]^{*}\gamma_{p}$
is a base branch for $\Sigma$ on $C_{2}$\\
\indent\ \ \ \ \ \ \ \ \ \ \  and $C_{1}$.\\

For Case 1, we have, by Lemma 2.1 and $(*)$;\\

$ord_{\gamma_{p}}({\phi_{1}\over\phi_{2}})=I_{\gamma_{p}}(C_{2},\phi_{1}-\lambda\phi_{2})=I_{[\Phi]^{*}\gamma_{p}}(C_{1},\overline{\phi_{1}}-\lambda\overline{\phi_{2}})=ord_{[\Phi]^{*}\gamma_{p}}({\overline{\phi_{1}}\over\overline{\phi_{2}}})$\\

where
${\phi_{1}\over\phi_{2}}(p)={\overline{\phi_{1}}\over\overline{\phi_{2}}}(q)=val_{\gamma_{p}}({\phi_{1}\over\phi_{2}})=val_{\gamma_{q}}({\overline{\phi_{1}}\over\overline{\phi_{2}}})=\lambda$
and $[\Phi]^{*}\gamma_{p}=\gamma_{q}$.\\

For Case 3, we have, by Lemma 2.1, $(*)$ and a similar argument to
the previous lemma, to show that the critical value
$\lambda=val_{\gamma_{p}}({\phi_{1}\over\phi_{2}})$ is also the
critical value
$val_{\gamma_{q}}({\overline{\phi_{1}}\over\overline{\phi_{2}}})$
for the lifted system at the
corresponding branch $[\Phi]^{*}\gamma_{p}$, that;\\

$ord_{\gamma_{p}}({\phi_{1}\over\phi_{2}})=I_{\gamma_{p}}^{\Sigma,mobile}(C_{2},\phi_{1}-\lambda\phi_{2})=I_{[\Phi]^{*}\gamma_{p}}^{\Sigma,mobile}(C_{1},\overline{\phi_{1}}-\lambda\overline{\phi_{2}})=ord_{[\Phi]^{*}\gamma_{p}}({\overline{\phi_{1}}\over\overline{\phi_{2}}})$\\

Case 2 is similar, we leave the details to the reader.\\

The lemma now follows from the previous lemma, that the
definitions of $ord_{\gamma_{p}}(f)$,
$ord_{[\Phi]^{*}\gamma_{p}}(\Phi^{*}f)$,$val_{\gamma_{p}}(f)$ and
$val_{[\Phi]^{*}\gamma_{p}}(\Phi^{*}f)$ are independent of their
particular representations.

\end{proof}

We now show;\\

\begin{lemma}
Let $C$ be a projective algebraic curve, then, to any non-constant
rational function $f$ on $C$, we can associate a $g_{n}^{1}$ on
$C$, which we will denote by $(f)$, where $n=deg(f)$, (flatness?).
\end{lemma}

\begin{proof}
We define the weighted set $(f=\lambda)$ as follows;\\

$(f=\lambda):=\{n_{\gamma_{1}},\ldots,n_{\gamma_{r}}\}$\\

where
$\{\gamma_{1},\ldots,\gamma_{r}\}=\{\gamma:val_{\gamma}(f)=\lambda\}$ and $n_{\gamma}=ord_{\gamma}(f)$.\\

As $\lambda$ varies over $P^{1}$, we obtain a series of weighted
sets $W_{\lambda}$ on $C$. We claim that this series does in fact
define a $g_{n}^{1}$. In order to see this, let $f$ be represented
as a rational function by $\phi\over\phi'$. As before, we consider
the pencil $\Sigma$ of forms defined by
$(\phi-\lambda\phi')_{\lambda\in P^{1}}$. We claim that the series
is defined by this system $\Sigma$, after removing its fixed
branch contribution, $(*)$. In order to see this, we compare the
weighted sets $(f=\lambda)$ and $C\sqcap (\phi-\lambda\phi')$. For
a branch $\gamma_{p}$ which is not a fixed branch of the system
$\Sigma$,
we have, using Lemmas 2.1 and 2.2, that;\\

 $\gamma_{p}\in (f=\lambda)$ iff $val_{\gamma_{p}}(f)=\lambda$ iff
${\phi\over\phi'}(p)=\lambda$ iff $p\in
C\cap(\phi-\lambda\phi')$\\

In this case, by Lemmas 2.1 and 2.2, we have that;\\

$n_{\gamma_{p}}=ord_{\gamma_{p}}(f)=ord_{\gamma_{p}}({\phi\over\phi'})=I_{\gamma_{p}}(C,\phi-\lambda\phi')$\\

For a branch $\gamma_{p}$ which is a fixed branch of the system
$\Sigma$, we have, by Lemmas 2.1 and 2.2, that;\\

$\gamma_{p}\in (f=\lambda)$ iff
$val_{\gamma_{p}}({\phi\over\phi'})=\lambda$ iff $p\in
C\cap(\phi-\lambda\phi')$ and $\lambda$ is a critical value for
the system $\Sigma$ at $\gamma_{p}$.\\

In this case, by Lemmas 2.1 and 2.2, we have that;\\

$n_{\gamma_{p}}=ord_{\gamma_{p}}(f)=ord_{\gamma_{p}}({\phi\over\phi'})=I_{\gamma_{p}}^{\Sigma,mobile}(C,\phi-\lambda\phi')$ $(1)$\\

Let $I_{\gamma_{p}}=min_{\mu\in
P^{1}}I_{\gamma_{p}}(C,\phi-\mu\phi')$ be the fixed branch
contribution of $\Sigma$ at $\gamma_{p}$. Then, at the critical
value $\lambda$ for the system $\Sigma$;\\

 $I_{\gamma_{p}}^{\Sigma,mobile}(\phi-\lambda\phi')=I_{\gamma_{p}}(C,\phi-\lambda\phi')-I_{\gamma_{p}}$ $(2)$\\

Hence, the result $(*)$ follows from $(1),(2)$ and the definition
of\\
 $C\sqcap(\phi-\lambda\phi')$.\\

Finally, we show that $n=deg(f)$. Let $\Gamma_{f}$ be the
correspondence determined by the rational map $f:C\rightsquigarrow
P^{1}$. By classical arguments, $deg(f)$ is equal to the
cardinality of the generic fibre $\Gamma_{f}(\lambda)$, for
$\lambda\in P^{1}$. Fixing a presentation ${\phi\over\phi'}$ for
$f$, if $U\subset NonSing(C)$ is the canonical set for this
presentation, one may assume that the generic fibre
$\Gamma_{f}(\lambda)$ lies inside $U$. By Lemma 2.17 of
\cite{depiro6}, one may also assume that the corresponding
weighted set of the $g_{n}^{1}$ defined by $(f=\lambda)$ consists
of $n$ distinct branches, centred at the points of the generic
fibre $\Gamma_{f}(\lambda)$. Therefore, the result follows.

\end{proof}

\begin{rmk}
By convention, for a non-zero rational function $c\in
{L\setminus\{0\}}$, we define $(c=0)$ and $(c=\infty)$ to be the
empty weighted sets. The notion of a weighted set in a
$g_{n}^{1}$, generalises the classical notion of the divisor on a
non-singular curve. Using the above theorem, we can make sense of
the notion of linear equivalence of weighted sets.
\end{rmk}

We make the following definition;\\

\begin{defn}{Linear equivalence of weighted sets}\\

Let $C$ be an algebraic curve and let $A$ and $B$ be weighted sets
on $C$ of the same total multiplicity. We define $A\equiv B$ if
there exists a $g_{n}^{r}$ on $C$ such that $A$ and $B$ belong to
this $g_{n}^{r}$ as weighted sets.

\end{defn}

\begin{theorem}
Let hypotheses be as in the previous definition. If $A\equiv B$,
then there exists a rational function $g$ on $C$, such that $A$ is
defined by $(g=0)$ and $B$ is defined by $(g=\infty)$, possibly
after adding some fixed branch contribution.

\end{theorem}

\begin{proof}
If $r=0$ in the definition, then we must have that $A=B$. Hence,
we obtain the statement of the theorem by adding the fixed branch
contribution $A$ to the empty $g_{0}^{0}$, defined by
$(c=0)=(c=\infty)$, for a non-constant $c\in L^{*}$. Otherwise, by
the definition of a $g_{n}^{r}$, we may, without loss of
generality, find a pencil $\Sigma$ of algebraic forms,
$\{\phi-\lambda\phi'\}_{\lambda\in P^{1}}$, having finite intersection with $C$, such that;\\

$A=C\sqcap(\phi-\lambda_{1}\phi)$,\\

$B=C\sqcap(\phi-\lambda_{2}\phi')$\ \ \ \ \  $(\lambda_{1}\neq
\lambda_{2})$\\

Let $f$ be the rational function on $C$ defined by
${\phi\over\phi'}$. If $A$ and $B$ have no branches in common
(with multiplicity), $(\dag)$, then the pencil $\Sigma$ can have
no fixed
branches and, by Lemma 2.4, we have that;\\

$A=(f=\lambda_{1})$\\

$B=(f=\lambda_{2})$\ \ \ \ \ \ \ $(\lambda_{1}\neq\lambda_{2})$\\

Now we can find an algebraic automorphism $\alpha$ of $P^{1}$,
taking $\lambda_{1}$ to $0$ and $\lambda_{2}$ to $\infty$. We will
assume that $\{\lambda_{1},\lambda_{2}\}\neq\infty$, in which case
$\alpha$ can be given, for a coordinate $z$ on $P^{1}$, by the
Mobius transformation ${z-\lambda_{1}\over z-\lambda_{2}}$. The
other cases are left to the reader. Let $g$ be the rational
function on $C$ defined by $\alpha\circ f$. Now, suppose that
$\gamma$ is a branch of $C$, with $val_{\gamma}(f)=\lambda$ and
$ord_{\gamma}(f)=m$. Then, we claim that
$val_{\gamma}(g)=\alpha(\lambda)$ and $ord_{\gamma}(g)=m$, $(*)$.
If $\lambda\neq\{\lambda_{2},\infty\}$, using the method before
Lemma 2.1, we obtain
the following power series representation of $g$ at $\gamma$;\\

${(\lambda+\mu t^{m}+o(t^{m}))-\lambda_{1}\over(\lambda+\mu
t^{m}+o(t^{m}))-\lambda_{2}}=[(\lambda-\lambda_{1})+\mu t^{m}+o(t^{m})]\centerdot{1\over (\lambda-\lambda_{2})}[1-{\mu\over(\lambda-\lambda_{2})}t^{m}+o(t^{m})]$\\
\indent \ \ \ \ \ \ \ \ \ \ \ \ \ \ \ \ \ \ \ \ \
$={\lambda-\lambda_{1}\over
\lambda-\lambda_{2}}+t^{m}[{\mu(\lambda-\lambda_{2})-\mu(\lambda-\lambda_{1})\over
(\lambda-\lambda_{2})^{2}}]+o(t^{m})$\\
\indent \ \ \ \ \ \ \ \ \ \ \ \ \ \ \ \ \ \ \ \ \
$={\lambda-\lambda_{1}\over
\lambda-\lambda_{2}}+t^{m}[{\mu(\lambda_{1}-\lambda_{2})\over
(\lambda-\lambda_{2})^{2}}]+o(t^{m})$\\

and the claim $(*)$ follows from the assumption that
$\lambda_{1}\neq \lambda_{2}$. If $\lambda=\lambda_{2}$, we obtain
the following power series representation of $g$ at
$\gamma$;\\

${(\lambda+\mu t^{m}+o(t^{m}))-\lambda_{1}\over(\mu
t^{m}+o(t^{m}))}={1\over t^{m}}\centerdot
[(\lambda-\lambda_{1})+\mu t^{m}+o(t^{m})]\centerdot [\mu+o(1)]^{-1}$\\

which gives that $val_{\gamma}(g)=\infty=\alpha(\lambda_{2})$ and
$ord_{\gamma}(g)=m$, using the fact that $\lambda\neq\lambda_{1}$.
Finally, if $\lambda=\infty$, the Mobius transformation at
$\infty$ is given by ${{1\over z}-\lambda_{1}\over {1\over
z}-\lambda_{2}}={1-\lambda_{1}z\over 1-\lambda_{2}z}$ and $g$ may
be represented at $\gamma$ by
${\phi-\lambda_{1}\phi'\over\phi-\lambda_{2}\phi'}$. We then
obtain the power series representation of $g$ at $\gamma$;\\

${(t^{i}u(t)-\lambda_{1}t^{i+m}v(t))\over(t^{i}u(t)-\lambda_{2}t^{i+m}v(t))}={(u(t)-\lambda_{1}t^{m}v(t))\over(u(t)-\lambda_{2}t^{m}v(t))}={[1-\lambda_{1}t^{m}{v(t)\over
u(t)}]\over [1-\lambda_{2}t^{m}{v(t)\over u(t)}]}$\\
$\indent \ \ \ \ \ \ \ \ \ \ \ \ \ \ \ \ \ \ \ \ =1+(\lambda_{2}-\lambda_{1})t^{m}w(t)+o(t^{m})$, for $\{u(t),v(t),w(t)\}$\\
\indent \ \ \ \ \ \ \ \ \ \ \ \ \ \ \ \ \ \ \ \ \ \ \ \ \ \ \ \ \ \ \ \ \ \ \ \ \ \ \ \ \ \ \ \ \ \ \ \ \ \ \ \ \ \ \ \ \ \ \ \ \ \ \ \ \ \ \ \  units in $L[[t]]$\\

which gives that $val_{\gamma}(g)=1=\alpha(\infty)$ and
$ord_{\gamma}(g)=m$, using the fact that
$\lambda_{1}\neq\lambda_{2}$ again. This gives the claim $(*)$. It
follows that the weighted sets $(f=\lambda)$ correspond exactly to
the weighted sets $(g=\alpha(\lambda))$, in particularly the
$g_{n}^{1}$ defined by $(f)$ and $(g)$, as in Lemma 2.4, is the
same. With this new parametrisation of the $g_{n}^{1}$, we then
have that;\\

$A=(g=0)$\\

$B=(g=\infty)$\\

Hence, the result follows, with the assumption $(\dag)$. If $A$
and $B$ have branches in common, with multiplicity, we let $A\cap
B$ denote the weighted set consisting of these common branches
(with multiplicity). Then, the same argument holds, replacing $A$
by $A\setminus B=A-(A\cap B)$ and $B$ by $B\setminus A=B-(A\cap
B)$. After adding the fixed branch contribution $(A\cap B)$ to the
$g_{n}^{1}$ defined by $(g)$, we then obtain the result. Note
that, by Lemma 1.13, this addition defines a $g_{n+n'}^{1}$, where
$n'$ is the total multiplicity of $(A\cap B)$. \\

\end{proof}

\begin{rmk}
The definition we have given of linear equivalence of weighted
sets on a projective algebraic curve $C$ generalises the modern
definition of linear equivalence for effective divisors on a
smooth projective algebraic curve. More precisely we have;\\

Modern Definition; Let $A$ and $B$ be effective divisors on a
smooth projective algebraic curve $C$, then $A\equiv B$ iff
$A-B=div(g)$, for some $g\in L(C)^{*}$.\\

See, for example, p161 of \cite{Shaf} for relevant definitions and
notation. We now show that our definition is the same in this
case. First, observe that there exists a natural bijection between
the set of effective divisors on $C$, in the sense of \cite{Shaf},
and the collection of weighted sets on $C$, $(*)$. This follows
immediately from the fact, given in Lemma 5.29 of \cite{depiro6},
that, for each point $p\in C$, there exists a unique branch
$\gamma_{p}$, centred at $p$. Secondly, observe that the notion of
$div(g)$, for $g\in L(C)$, as given in \cite{Shaf}, is the same as
the notion of $div(g)$ which we give in Definition 2.9 below,
(taking into account the identification $(*)$), $(\dag)$. This
amounts to checking that, for a point $p\in C$, with corresponding
branch $\gamma_{p}$;\\

$v_{p}(g)=ord_{\gamma_{p}}(g)$ $(\dag\dag)$\\

where $v_{p}(g)$ is defined in p152 of \cite{Shaf}, and we temporarily adopt
the convention that $ord_{\gamma_{p}}(g)=0$ if $val_{\gamma}(g)\neq\{0,\infty\}$
and $ord_{\gamma_{p}}(g)$ is counted negatively if $val_{\gamma_{p}}(g)=\infty$.
First, one can use the fact, given in Lemma 4.9 of \cite{depiro6}, together with
remarks from the final section of this paper, that there exists a
birational map $\phi:C\leftrightsquigarrow C'$, such that $C'$ is
a plane projective algebraic curve, and $p$ corresponds to a
non-singular point $p'\in C'$ with $\{p,p'\}$ lying inside the
canonical sets associated to $\phi$. Using the calculation given
below, in Lemma 2.10, for $ord_{\gamma_{p}}$, and the definition
of $v_{p}$, one can assume that $v_{p}(g)\geq 0$ and $g\in
O_{p,C}$. Let $g'\in L(C')$ denote the corresponding rational
function to $g$ on $L(C)$. It is then a trivial algebraic
calculation, using the fact that the local rings $O_{p,C}$ and
$O_{p',C'}$ are isomorphic, to show that $v_{p}(g)=v_{p}(g')$. It
also follows from Lemma 2.3 that
$ord_{\gamma_{p}}(g)=ord_{\gamma_{p'}}(g')$. Hence, it is
sufficient to check $(\dag)$ for the plane projective curve $C'$.
We may, without loss of generality, assume that $v_{p}(g')\geq 1$
and that $g'$ is represented in some choice of affine coordinates
$\{x,y\}$ by the polynomial $q(x,y)$. If $Q(X,Y,Z)$ denotes the
projective equation of this polynomial and $p$ corresponds to the
origin of this coordinate system, then;\\

$v_{p}(g')=I_{p}(C,Q)=length({L[x,y]\over <h,q>})$\\

where $h$ is a defining equation for $C'$ in the coordinate system
$\{x,y\}$ and $I_{p}$ is the algebraic intersection multiplicity.
It also follows from Lemma 2.1, that;\\

$ord_{\gamma_{p}}(g')=I_{\gamma_{p}}(C,Q)$\\

Hence, it is sufficient to check that;\\

$I_{p}(C,Q)=I_{\gamma_{p}}(C,Q)$\\

This calculation was done in the paper \cite{depiro5}, hence
$(\dag\dag)$ and therefore $(\dag)$ is shown. Thirdly, it remains
to check that the definitions of linear equivalence are the same.
In order to see this, observe that we can write (for effective
divisors or weighted sets $A$ and $B$);\\

$A-B=(A\setminus B)+(A\cap B)]-[(B\setminus A)+(A\cap B)]=(A\setminus B)-(B\setminus A)$, $(\dag\dag\dag)$\\

If $A\equiv B$ in the sense of weighted sets (Definition 2.6),
then the calculation $(\dag\dag\dag)$ (which removes the fixed
branch contribution) and Theorem 2.7 shows that $A-B=div(g)$, for
some rational function $g\in L(C)$, where, here, $div(g)$ is as
defined in Definition 2.9. By $(\dag)$, it then follows that
$A\equiv B$ as effective divisors. Conversely, if $A\equiv B$ as
effective divisors, then there exists a rational function $g\in
L(C)$ such that $A-B=div(g)$, in the sense of the modern
definition given above. The above calculations $(\dag\dag\dag)$
and $(\dag)$ then show that $div(g)=(A\setminus B)-(B\setminus
A)$, in the sense of Definition 2.9 below. It follows, by Lemma
2.4, that there exists a $g_{n}^{1}$ to which $(A\setminus B)$ and
$(B\setminus A)$ belong as weighted sets. Adding the fixed branch
contribution $(A\cap B)$ to this $g_{n}^{1}$, we then obtain that
$A\equiv B$ in the sense of Definition 2.6, as required.
\end{rmk}

\begin{defn}
Let $C$ be a projective algebraic curve and let $f$ be a non-zero
rational function on $C$. Then we define $div(f)$ to
be the weighted set $A-B$ where;\\

$A=(f=0)$,\indent $B=(f=\infty)$\\

\end{defn}

We now require the following lemma;\\

\begin{lemma}
Let $C$ be a projective algebraic curve, and let $f$ and $g$ be
non-zero rational functions on $C$. Then;\\

$div({1\over f})=-div(f)$\\

$div(fg)=div(f)+div(g)$\\

$div({f\over g})=div(f)-div(g)$\\

\end{lemma}

\begin{proof}
In order to prove the first claim, it is sufficient to show that,
for a branch $\gamma$ of $C$;\\

$val_{\gamma}(f)=0$ iff $val_{\gamma}({1\over f})=\infty$\\

$val_{\gamma}(f)=\infty$ iff $val_{\gamma}({1\over f})=0$\\

and $ord_{\gamma}$ is preserved in both cases. This follows
trivially from the relevant power series calculation at a branch.
Namely, we can represent $f$ by ${\phi\over\phi'}$ and ${1\over
f}$ by ${\phi'\over\phi}$. Substituting the branch
parametrisation, we obtain that;\\

$val_{\gamma}(f)=0, ord_{\gamma}(f)=m$ iff $f\sim t^{m}u(t)$, \
$m\geq 1,u(t)\in L[[t]]$ a
unit.\\
\indent \ \ \ \ \ \ \ \ \ \ \ \ \ \ \ \ \ \ \ \ \ \ \ \ \ \ \ \ \
\ \ \ \ \ iff ${1\over f}\sim
t^{-m}u(t)^{-1}$\\
\indent \ \ \ \ \ \ \ \ \ \ \ \ \ \ \ \ \ \ \ \ \ \ \ \ \ \ \ \ \
\ \ \ \ \ iff $val_{\gamma}(f)=\infty, ord_{\gamma}(f)=m$\\

and the calculation for $val_{\gamma}(f)=\infty,
ord_{\gamma}(f)=m$ is similar.\\

In order to prove the second claim, we need to verify the following cases at a branch $\gamma$ of $C$;\\

Case 1. If $val_{\gamma}(f)=val_{\gamma}(g)\in\{0,\infty\}$,
$ord_{\gamma}(f)=m$ and
$ord_{\gamma}(g)=n$\\

\indent \ \ \ \ \ \ \ \ \ \ \  then $val_{\gamma}(fg)\in\{0,\infty\}$ and $ord_{\gamma}(fg)=m+n$\\

Case 2. If $val_{\gamma}(f)\neq val_{\gamma}(g)\in\{0,\infty\}$,
$ord_{\gamma}(f)=m$ and $ord_{\gamma}(g)=n$\\

\indent \ \ \ \ \ \ \ \ \ \ \ \ then
$val_{\gamma}(fg)\in\{0,\infty\}$ and $ord_{\gamma}(fg)=|m-n|$\\

Case 3. If exactly one of $val_{\gamma}(f)$ and $val_{\gamma}(g)$
is in $\{0,\infty\}$, with\\
\indent \ \ \ \ \ \ \  \ \ \ \ $ord_{\gamma}(f)$ or $ord_{\gamma}(g)=m$\\

\indent \ \ \ \ \ \ \ \ \ \ \ then $val_{\gamma}(fg)\in\{0,\infty\}$, with $ord_{\gamma}(fg)=m$.\\

Case 4. If neither of $val_{\gamma}(f)$ and $val_{\gamma}(g)$ are
in $\{0,\infty\}$\\

\indent \ \ \ \ \ \ \ \ \ \ \ \ then $val_{\gamma}(fg)$ is not in
$\{0,\infty\}$\\

If $f$ is represented by ${\phi\over\phi'}$ and $g$ is represented
by ${\psi\over\psi'}$, then we can represent $fg$ by
${\phi\psi\over\phi'\psi'}$. The proof of these cases then follow by elementary power series calculations at the branch $\gamma$.
For example, for Case 2, if $val_{\gamma}(f)=0$ and $ord_{\gamma}(f)=m$, $val_{\gamma}(g)=\infty$ and $ord_{\gamma}(g)=n$, then we have;\\

$f\sim t^{n}u(t)$, $g\sim t^{-m}v(t)$, $fg\sim
t^{n}t^{-m}u(t)v(t)=t^{n-m}w(t)$,\\
\indent \ \ \ \ \ \ \ \ \ \ \ \ \ \ \ \ \ \ \ \ \ \ \ \ \ \ \ \ \
\ \ \ \ \ \ \ \  for $\{u(t),v(t),w(t)\}$ units in
$L[[t]]$.\\

The third claim follows from the first two claims.

\end{proof}

We now claim the following;\\

\begin{theorem}{Transitivity of Linear Equivalence}\\

Let $C'$ be an algebraic curve. If $A,B,C$ are weighted sets on
$C'$ of the same total multiplicity, then, if $A\equiv B$ and
$B\equiv C$, we must have that $A\equiv C$.

\end{theorem}

\begin{proof}
By Theorem 2.7, we can find rational functions $f$ and $g$ on
$C'$, such that;\\

$(A\setminus B)-(B\setminus A)=div(f)$\\

$(B\setminus C)-(C\setminus B)=div(g)$\\

By Lemma 2.9, we have that;\\

$div(fg)=(A\setminus B)-(B\setminus A)+(B\setminus C)-(C\setminus
B)$\\

By drawing a Venn diagram, one easily checks that;\\

$(A\setminus B)-(B\setminus A)=(A\cap B^{c}\cap C^{c})+(A\cap
B^{c}\cap C)-(A^{c}\cap B\cap C^{c})-\\
\indent \ \ \ \ \ \ \ \ \ \ \ \ \ \ \ \ \ \ \ \ \ \ \ \ \ \ \ (A^{c}\cap B\cap C)$\\
\indent \ \ \ \ \ \ \ \ \ \ $+$\\
\indent $(B\setminus C)-(C\setminus B)=(A\cap B\cap
C^{c})+(A^{c}\cap
B\cap C^{c})-(A^{c}\cap B^{c}\cap C)-\\
\indent \ \ \ \ \ \ \ \ \ \ \ \ \ \ \ \ \ \ \ \ \ \ \ \ \ \ \
(A\cap B^{c}\cap C)$\\
\indent \ \ \ \ \ \ \ \ \ \ \ $||$\\
\indent $(A\setminus C)-(C\setminus A)=(A\cap B^{c}\cap
C^{c})+(A\cap
B\cap C^{c})-(A^{c}\cap B^{c}\cap C)-\\
\indent \ \ \ \ \ \ \ \ \ \ \ \ \ \ \ \ \ \ \ \ \ \ \ \ \ \ \ (A^{c}\cap B\cap C)$\\

Hence, $div(fg)=(A\setminus C)-(C\setminus A)$. Now, given the
$g_{n}^{1}$ defined by the rational function $fg$, as in Lemma
2.4, it follows that $(A\setminus C)$ and $(C\setminus A)$ belong
to this $g_{n}^{1}$ as weighted sets. We can now add the fixed
branch contribution $A\cap C$ to this $g_{n}^{1}$, giving a
$g_{n+n'}^{1}$, to which $A$ and $C$ belong as weighted sets.
Therefore, the result follows.

\end{proof}

As an immediate corollary, we have;\\

\begin{theorem}
Let $C$ be a projective algebraic curve, then $\equiv$ is an
equivalence relation on weighted sets for $C$ of a given
multiplicity.

\end{theorem}

We also have;\\

\begin{theorem}{Linear Equivalence preserved by Addition}\\

Let $C'$ be a projective algebraic curve and suppose that
$\{A,B,C,D\}$ are weighted sets on $C'$ with;\\

$A\equiv B$ and $C\equiv D$\\

then;\\

$A+C\equiv B+D$\\

\end{theorem}

\begin{proof}
By Definition 2.6, we can find a $g_{n}^{r}$ containing $C$ and
$D$ as weighted sets. If $s$ is the total multiplicity of $A$,
then, by Lemma 1.13, we can add the weighted set $A$ as a fixed
branch contribution to this $g_{n}^{r}$ and obtain a
$g_{n+s}^{r}$, containing $A+C$ and $A+D$ as weighted sets. Hence,
by Definition 2.6 again, we have that;\\

$A+C\equiv A+D$ $(1)$\\

Similarily, one shows, by adding $D$ as a fixed branch contribution to the $g_{n'}^{r'}$ containing
$A$ and $B$ as weighted sets, that;\\

$A+D\equiv B+D$ $(2)$\\

The result then follows immediately by combining $(1)$, $(2)$ and
using Theorem 2.11.

\end{proof}

We now develop further the theory of $g_{n}^{r}$ on a projective
algebraic curve $C$. We begin with the following definition;\\

\begin{defn}{Subordinate $g_{n}^{r}$}\\

Let  $\{g_{n}^{r},g_{n}^{t}\}$ be given on $C$ with the
\emph{same} order $n$. Then we say that;\\

$g_{n}^{r}\subseteq g_{n}^{t}$\\

if \emph{every} weighted set in $g_{n}^{r}$ is included in the
weighted sets of the $g_{n}^{t}$.

\end{defn}

We now claim the following;\\

\begin{theorem}{Amalgamation of $g_{n}^{r}$}\\

Let $\{g_{n}^{r},g_{n}^{s}\}$ be given on $C$, having a common
weighted set $G$, then there exists $t$ with $r\leq t, s\leq t$
and a $g_{n}^{t}$ such that $g_{n}^{r}\subseteq g_{n}^{t}$ and
$g_{n}^{s}\subseteq g_{n}^{t}$.
\end{theorem}

\begin{proof}
Assume first that $\{g_{n}^{r},g_{n}^{s}\}$ have no fixed branch
contribution and are defined exactly by linear systems. Then we
can find algebraic forms
$\{\phi_{0},\psi_{0}\}$ such that;\\

$G=(C\sqcap\phi_{0}=0)=(C\sqcap\psi_{0}=0)$\\

and;\\

$g_{n}^{r}$ is defined by
$C\sqcap(\epsilon_{0}\phi_{0}+\epsilon_{1}\phi_{1}+\ldots+\epsilon_{r}\phi_{r}=0)$\\

$g_{n}^{s}$ is defined by $C\sqcap(\eta_{0}\psi_{0}+\eta_{1}\psi_{1}+\ldots+\eta_{s}\psi_{s}=0)$\\

Now consider the linear system $\Sigma$ defined by;\\

$\epsilon\phi_{0}\psi_{0}+\psi_{0}(\epsilon_{1}\phi_{1}+\ldots+\epsilon_{r}\phi_{r})+\phi_{0}(\eta_{1}\psi_{1}+\ldots+\eta_{s}\psi_{s})=0$\\

and let $g_{m}^{t}$ be defined by $\Sigma$. As
$deg(\psi_{0}\phi_{0})=deg(\psi_{0})+deg(\phi_{0})$, we have that
$m=2n$. We claim that the fixed branch contribution of
$g_{2n}^{t}$ is exactly $G$, $(*)$. In order to see this, observe
that we can write an algebraic form in $\Sigma$ as;\\

$\psi_{0}\phi_{\bar\epsilon}+\phi_{0}\psi_{\bar\eta}$\\

If $\gamma$ is a branch counted $w$-times in $G$,
then, using the proof at the end of Lemma 1.13 and linearity of multiplicity at a branch,
see \cite{depiro6};\\

$I_{\gamma}(C,\psi_{0}\phi_{\bar\epsilon})=I_{\gamma}(C,\psi_{0})+I_{\gamma}(C,\phi_{\bar\epsilon})\geq
w$\\

$I_{\gamma}(C,\phi_{0}\psi_{\bar\eta})=I_{\gamma}(C,\phi_{0})+I_{\gamma}(C,\psi_{\bar\eta})\geq
w$\\

$I_{\gamma}(C,\psi_{0}\phi_{\bar\epsilon}+\phi_{0}\psi_{\bar\eta})=
min\{I_{\gamma}(C,\psi_{0}\phi_{\bar\epsilon}),I_{\gamma}(C,\phi_{0}\psi_{\bar\eta})\}\geq
w$ $(\dag)$\\

Hence, $\gamma$ is $w$-fold for the $g_{2n}^{t}$ and $G$ is
contained in the fixed branch contribution of the $g_{2n}^{t}$. In
order to obtain the exactness statement, $(*)$, first observe
that, if $\gamma$ is a fixed branch of the $g_{2n}^{t}$, then, in
particular, it belongs to $(C\sqcap \phi_{0}\psi_{0}=0)$. Hence,
it belongs either to $(C\sqcap \phi_{0}=0)$ or
$(C\sqcap\psi_{0}=0)$. Hence, it belongs to $G$. Now, using the
fact that the original $\{g_{n}^{r},g_{n}^{s}\}$ had no fixed
branch contribution, we can easily find $\phi_{\bar\epsilon_{0}}$
and $\psi_{\bar\eta_{0}}$ with $G$ disjoint from both $(C\sqcap
\phi_{\bar\epsilon_{0}}=0)$ and $(C\sqcap\psi_{\bar\eta_{0}}=0)$.
Then, by the same argument $(\dag)$, we obtain, for a branch $\gamma$ of $G$;\\

$I_{\gamma}(C,\psi_{0}\phi_{\bar\epsilon_{0}}+\phi_{0}\psi_{\bar\eta_{0}})=w$\\

hence, $\gamma$ is counted $w$-times in
$C\sqcap(\psi_{0}\phi_{\bar\epsilon_{0}}+\phi_{0}\psi_{\bar\eta_{0}}=0)$
and, therefore, $(*)$ holds, as required. Now, as $G$ had total
multiplicity $n$, removing this fixed branch contribution from the
$g_{2n}^{t}$, we obtain a $g_{n}^{t}$. We then claim that
$g_{n}^{r}\subseteq g_{n}^{t}$ and $g_{n}^{s}\subseteq g_{n}^{t}$,
$(**)$. By Definition 2.14, it is sufficient to check that, if
$\{W_{1},W_{2}\}$ are weighted sets appearing in
$\{g_{n}^{r},g_{n}^{s}\}$, defined by $(C\sqcap
\phi_{\bar\epsilon}=0)$ and $(C\sqcap\psi_{\bar\eta}=0)$, then
they appear in the $g_{n}^{t}$. We clearly have that both
$\psi_{0}\phi_{\bar\epsilon}$ and $\phi_{0}\psi_{\bar\eta}$ belong
to $\Sigma$ and the calculation $(\dag)$ shows that;\\

$C\sqcap(\psi_{0}\phi_{\bar\epsilon}=0)=W_{1}+G$\\

$C\sqcap(\phi_{0}\psi_{\bar\eta}=0)=W_{2}+G$\\

Hence, the result $(**)$ follows after removing the fixing branch
contribution $G$. The fact that $r\leq t$ and $s\leq t$ then
follows easily from the definition of the dimension of a $g_{n}^{r}$
and Theorem 1.3.\\

Now consider the case when the $\{g_{n}^{r},g_{n}^{s}\}$ are
defined exactly by linear systems and \emph{have} a fixed branch
contribution. Let $G_{1}\subseteq G$ and $G_{2}\subseteq G$ be
these fixed branch contributions and let $G_{3}=G_{1}\cap G_{2}$.
We claim that the fixed branch contribution of the $g_{2n}^{t}$
defined by $\Sigma$, as given above, in this case is exactly
$G_{3}+G$. The proof is similar to the above and left to the
reader. Now, removing the fixed branch contribution $G$, we obtain
a series $g_{n}^{t}$ with fixed branch contribution $G_{3}$. A
similar proof to the above, left to the reader, shows that this
$g_{n}^{t}$ contains the original series
$\{g_{n}^{r},g_{n}^{s}\}$. Finally, we need to consider the case
when the $\{g_{n}^{r},g_{n}^{s}\}$ are defined, after removing
some fixed branch contribution from linear series. Let $G_{1}$ and
$G_{2}$, with total multiplicity $r_{1}$ and $r_{2}$, be these
fixed branch contributions and let
$\{g_{n+r_{1}}^{r},g_{n+r_{2}}^{s}\}$ be the series obtained from
adding these fixed branch contributions to
$\{g_{n}^{r},g_{n}^{s}\}$. In this case, the linear system
$\Sigma$, as given above, defines a $g_{2n+r_{1}+r_{2}}^{t}$. We
claim that the weighted set $G\cup G_{1}\cup G_{2}$, of total
multiplicity $(n+r_{1}+r_{2})$, is contained in the fixed branch
contribution of this series. This follows from a similar
calculation, using the method above, the details are left to the
reader. Removing this weighted set from the
$g_{2n+r_{1}+r_{2}}^{t}$, we obtain a $g_{n}^{t}$ and a similar
calculation shows that this contains the original
$\{g_{n}^{r},g_{n}^{s}\}$, again the details are left to the
reader.
\end{proof}

As a corollary, we have;\\

\begin{theorem}
Let a $g_{n}^{r}$ be given on $C$, then there exists a
\emph{unique} $g_{n}^{t}$ on $C$, with $r\leq t\leq n$, such that;\\

$g_{n}^{r}\subseteq g_{n}^{t}$\\

and, for \emph{any} $g_{n}^{s}$ such that $g_{n}^{r}\subseteq
g_{n}^{s}$, we have that;\\

$g_{n}^{s}\subseteq g_{n}^{t}$\\
\end{theorem}

\begin{proof}
By Lemma 1.16, we can find $r\leq t\leq n$ and a $g_{n}^{t}$ on
$C$, with $g_{n}^{r}\subseteq g_{n}^{t}$ and $t$ maximal with this
property. If $g_{n}^{r}\subseteq g_{n}^{s}$, then
$\{g_{n}^{s},g_{n}^{t}\}$ would contain a common weighted set. By
Theorem 2.15, we could then find $t'\leq n$ such that $s\leq t'$,
$t\leq t'$ and $g_{n}^{s}\subseteq g_{n}^{t'}$,
$g_{n}^{t}\subseteq g_{n}^{t'}$. If $g_{n}^{s}\varsubsetneq
g_{n}^{t}$, then, by elementary dimension considerations, we would
have that $t<t'\leq n$ and $g_{n}^{r}\subset g_{n}^{t'}$,
contradicting maximality of $t$. Hence, $g_{n}^{s}\subseteq
g_{n}^{t}$. The uniqueness statement also follows from a similar
amalgamation argument, using Theorem 2.15.

\end{proof}

We can then make the following definition;\\

\begin{defn}
We call a $g_{n}^{r}$ on $C$ complete if it cannot be strictly
contained in a $g_{n}^{t}$ of greater dimension. If $G$ is any
weighted set on $C$ of total multiplicity $n$, then we define
$|G|$ to be the unique complete $g_{n}^{t}$ to which $G$ belongs.
\end{defn}

We then have that;\\

\begin{theorem}
Let $G$ be a weighted set on $C$, then, $G\equiv G'$ if and only
if $G'$ belongs to $|G|$. In particular, $G\equiv G'$ if and only
if $|G|=|G'|$.

\end{theorem}

\begin{proof}
The proof of the first part of the theorem is quite
straightforward. By definition, if $G'$ belongs to $|G|$, then
$G\equiv G'$. Conversely, if $G'\equiv G$, then, by Definition
2.6, we can find a $g_{n}^{1}$, containing the given weighted sets
$G$ and $G'$. By Theorem 2.16, we can find a unique complete
$g_{n}^{t}$ on $C$, with $1\leq t\leq n$, such that
$g_{n}^{1}\subseteq g_{n}^{t}$. As $G$ belongs to this $g_{n}^{t}$
as a weighted set, it follows by Definition 2.17 that
$|G|=g_{n}^{t}$. Hence, $G'$ belongs to $|G|$ as required. For the
second part, if $G\equiv G'$, then, by the first part, $G'$
belongs to $|G|$. It follows immediately from Definition 2.17 and
Theorem 2.16, that $|G|\subseteq |G'|$. Reversing this argument,
we have that $|G'|\subseteq |G|$, hence $|G|=|G'|$ as required.
Conversely, if $|G|=|G'|$, then clearly $G\equiv G'$ by Definition
2.6.
\end{proof}

We now make the following definition;\\

\begin{defn}{Linear System of a Weighted Set}\\

Let $G$ be a weighted set on a projective algebraic curve $C$,
then we define the Riemann-Roch space ${\mathcal L}(C,G)$ or
${\mathcal L}(G)$ to be the vector space defined as;\\

$\{g\in L(C)^{*}:div(g)+G\geq 0\}\cup\{0\}$\\

where $div(g)$ was defined in Definition 2.9.

\end{defn}

\begin{rmk}
That ${\mathcal L}(G)$ defines a vector space follows easily from
Lemma 2.10, the fact that, for non-constant rational functions
$\{f,g,f+g\}\subset
L(C)$ and a branch $\gamma$ of $C$, we have that;\\

$ord_{\gamma}(f+g)\geq min\{ord_{\gamma}(f),ord_{\gamma}(g)\}$, $(*)$\\

where, for this remark only, $ord_{\gamma}$ is counted
\emph{negatively} if $val_{\gamma}$ is infinite, and an argument
on constants, $(**)$.
 We now give a brief proof of $(*)$;\\

We just consider the following $2$ cases;\\

Case 1. $val_{\gamma}(f)<\infty$ and $val_{\gamma}(g)<\infty$\\

We then have, substituting the relative parametrisations, that;\\

$f\sim c+c_{1}t^{m}+\ldots$ and $g\sim d+d_{1}t^{n}+\ldots$, where
$ord_{\gamma}(f)=m\geq 1$, $ord_{\gamma}(g)=n\geq 1$ and
$\{c_{1},d_{1}\}\subset L$ are non-zero. Then;\\

$f+g\sim (c+d)+c_{1}t^{m}+d_{1}t^{n}+\ldots$\\

If $(f+g)-(c+d)\equiv 0$, as an algebraic power series in
$L[[t]]$, then $(f+g)=(c+d)$ as a rational function on $C$,
contradicting the assumption. Hence, we obtain that
$ord_{\gamma}(f+g)=min\{ord_{\gamma}(f),ord_{\gamma}(g)\}$, if
$m\neq n$ or $m=n$ and $c_{1}+d_{1}\neq 0$, and
$ord_{\gamma}(f+g)>min\{ord_{\gamma}(f),ord_{\gamma}(g)\}$
otherwise. Hence, $(*)$ is shown in this case.\\

Case 2. $val_{\gamma}(f)=val_{\gamma}(g)=\infty$\\

We then have that;\\

$f\sim c_{1}t^{-m}+\ldots$ and $g\sim d_{1}t^{-n}+\ldots$, where
$ord_{\gamma}(f)=-m\leq -1$, $ord_{\gamma}(g)=-n\leq -1$ and
$\{c_{1},d_{1}\}\subset L$ are non-zero. Then;\\

$f+g\sim c_{1}t^{-m}+d_{1}t^{-n}+\ldots$\\

By the assumption that $f+g$ is not a constant, if $m=n$ and
$c_{1}+d_{1}=0$, we must have higher order terms in $t$ in the
Cauchy series for $(f+g)$, hence
$ord_{\gamma}(f+g)>min\{ord_{\gamma}(f),ord_{\gamma}(g)\}$.
Otherwise, we have that
$ord_{\gamma}(f+g)=min\{ord_{\gamma}(f),ord_{\gamma}(g)\}$, hence
$(*)$ is shown in this case as well.\\

The remaining cases are left to the reader. One should also
consider the case of constants, $(**)$. Technically, one cannot
define $ord_{\gamma}$ for a constant in $L$. However, we did, by
convention, define $div(c)=0$, for $c\in L^{*}$, in Remarks 2.5.\\
\end{rmk}

We now show the following;\\

\begin{lemma}
For a weighted set $G$, $dim({\mathcal L}(G))=t+1$, where $t$ is
given in Definition 2.17. In particular, ${\mathcal L}(G)$ is
finite dimensional.

\end{lemma}

\begin{proof}
Let $t$ be given by Definition 2.17. If $t=0$, then $G=(0)$ and
${\mathcal L}(G)=L$. This follows easily from the well known fact
that the only regular functions on a projective algebraic curve
are the constants (see, for example, \cite{Shaf}, p59). In this
case, we then have that $dim({\mathcal L}(G))=1$, as required.
Otherwise, let $t\geq 1$ be given as in Definition 2.17, with the
unique complete $g_{n}^{t}$ containing $G$. After adding some
fixed branch contribution $W$, we can find a linear system
$\Sigma$, having finite intersection with $C$, with basis
$\{\phi_{0},\ldots,\phi_{j},\ldots,\phi_{t}\}$ defining this
$g_{n}^{t}$. Moreover, we may assume that $C\sqcap \phi_{0}=G\cup
W$, $(*)$. Let $\{f_{1},\ldots,f_{j},\ldots,f_{t}\}$ be the
sequence of rational functions on $C$ defined by
$f_{j}={\phi_{j}\over\phi_{0}}$. We claim that;\\

$div(f_{j})+G\geq 0$, for $1\leq j\leq t$ $(**)$\\

In order to show $(**)$, it is sufficient to prove that, for a
branch $\gamma$ with $val_{\gamma}(f_{j})=\infty$, we have that
$\gamma$ belong to $G$ and, moreover, that $\gamma$ is counted at
least $ord_{\gamma}(f_{j})$ times in $G$. Let $\Sigma_{j}$ be the
pencil of forms defined by $(\phi_{j}-\lambda\phi_{0})_{\lambda\in
P^{1}}$. By the proof of Lemma 2.4, we have that $(f_{j}=\infty)$
is defined by $(C\sqcap\phi_{0})$, after removing the fixed branch
contribution of this pencil. By $(*)$ and the fact that the fixed
branch contribution of $\Sigma_{j}$ includes $W$, we have that
$(f_{j}=\infty)\subseteq G$. Hence, $(**)$ is shown as required.
By Definition 2.19, we then have that $f_{j}$ belongs to
${\mathcal L}(G)$. We now claim that there do \emph{not} exist
constants $\{c_{0},\ldots,c_{j},\ldots,c_{t}\}\subset L$ such that;\\

$c_{0}+c_{1}f_{1}+\ldots+c_{j}f_{j}+\ldots+c_{t}f_{t}=0$ $(***)$\\

as rational functions on $C$. If so, we would have that;\\

$c_{0}\phi_{0}+c_{1}\phi_{1}+\ldots+c_{j}\phi_{j}+\ldots+c_{t}\phi_{t}$\\

vanished identically on $C$, contradicting the fact that $\Sigma$
has finite intersection with $C$. Hence, by $(***)$,
$\{1,f_{1},\ldots,f_{t}\}\subset {\mathcal L}(G)$ are linearly
independent and $dim({\mathcal L}(G))\geq t+1$. Conversely,
suppose that $dim({\mathcal L}(G))\geq k+1$, then we can find
$\{1,f_{1},\ldots,f_{j},\ldots,f_{k}\}\subset {\mathcal L}(G)$
which are linearly independent, $(\dag)$. By the usual method of
equating denominators, we can find algebraic forms
$\{\phi_{0},\ldots,\phi_{k}\}$ of the same degree, such that
$f_{j}$ is represented by $\phi_{j}\over\phi_{0}$, for $1\leq
j\leq k$. Let $\Sigma$ be the linear system defined by this
sequence of forms. By $(\dag)$, $\Sigma$ has finite intersection
with $C$. Let $W$, having total multiplicity $n'$, be the fixed
branch contribution of this system and let
$(C\sqcap\phi_{0})=G_{0}\cup W$. We claim that $G_{0}\subseteq G$,
$(\dag\dag)$. Suppose not, then there exists a branch $\gamma$
with $I_{\gamma}^{\Sigma,mobile}(C,\phi_{0})=s$, where $\gamma$ is
counted strictly less than $s$-times in $G$. By the definition of
$I_{\gamma}^{\Sigma,mobile}$, we can find a form $\phi_{\lambda}$
belonging to $\Sigma$, distinct from $\phi_{0}$, witnessing this
multiplicity. Consider the pencil $\Sigma_{\lambda}$ defined by
$(\phi_{\lambda}-\mu\phi_{0})_{\mu\in P^{1}}$. We then clearly
have that $I_{\gamma}^{\Sigma_{\lambda},mobile}(C,\phi_{0})=s$ as
well, $(\dag\dag\dag)$. Let
$f_{\lambda}={\phi_{\lambda}\over\phi_{0}}$. By the proof of Lemma
2.4, we have that $(f_{\lambda}=\infty)$ is defined by
$(C\sqcap\phi_{0})$, after removing the fixed branch contribution
of $\Sigma_{\lambda}$. By $(\dag\dag\dag)$, it follows that the
branch $\gamma$ is counted $s$-times in $(f_{\lambda}=\infty)$ and
therefore $div(f_{\lambda})+G\ngeq 0$. However, $f_{\lambda}$ is a
linear combination of $\{1,\ldots,f_{k}\}$, hence
$f_{\lambda}\in{\mathcal L}(G)$, which is a contradiction. Hence,
$(\dag\dag)$ is shown. Now, consider the $g_{n}^{k}$ defined by
$\Sigma$. Let $W'$ be the weighted set $G\setminus G_{0}$ of total
multiplicity $n''$. By Lemma 1.13, we can add the weighted set
$W'$ to the $g_{n}^{k}$ and obtain a $g_{n+n''}^{k}$ with fixed
branch contribution $W'\cup W$. Now, removing the fixed branch
contribution $W$ from this $g_{n+n'}^{k}$, we obtain a
$g_{n+n''-n'}^{k}$ containing $G$ exactly as a weighted set. It
follows, from Definition 2.17, that $k\leq t$. Hence, in
particular, $dim({\mathcal L}(G))$ is finite and $dim({\mathcal
L}(G)\leq t+1$. Therefore, the lemma is proved.

\end{proof}

We now extend the notion of linear equivalence to include virtual,
or non-effective, weighted sets.\\

\begin{defn}
We define a generalised weighted set $G$ on $C$ to be a linear
combination of branches;\\

$n_{1}\gamma_{p_{1}}^{j_{1}}+\ldots+n_{r}\gamma_{p_{r}}^{j_{r}}$\\

where $\{n_{1},\ldots,n_{r}\}$ belong to ${\mathcal Z}$. If
$\{n_{1},\ldots,n_{r}\}$ belong to ${\mathcal Z}_{\geq 0}$, we
call the weighted set effective. Otherwise, we call the weighted
set virtual. We define $n=n_{1}+\ldots+n_{r}$ to be the total
multiplicity or degree of $G$.
\end{defn}

\begin{rmk}
It is an easy exercise to see that there exist well defined
operations of addition and subtraction on generalised weighted
sets. It is also easy to check that any generalised weighted set
$G$ may be written uniquely as $G_{1}-G_{2}$, where
$\{G_{1},G_{2}\}$ are \emph{disjoint effective} weighted sets.
\end{rmk}

\begin{defn}
Let $A$ and $B$ be generalised weighted sets on $C$ of the same
total multiplicity. Let $\{A_{1},A_{2}\}$ and $\{B_{1},B_{2}\}$ be
the unique effective weighted sets, as given by the previous
remark. Then we define;\\

$(A_{1}-A_{2})\equiv (B_{1}-B_{2})$ iff $(A_{1}+B_{2})\equiv
(B_{1}+A_{2})$\\

and;\\

$A\equiv B$ iff $(A_{1}-A_{2})\equiv (B_{1}-B_{2})$\\

\end{defn}

\begin{rmk}
Note that if $\{A_{1}',A_{2}'\}$ and $\{B_{1}',B_{2}'\}$ are
\emph{any} effective weighted sets such that;\\

$A=A_{1}'-A_{2}'$ and $B=B_{1}'-B_{2}'$\\

then $A\equiv B$ iff $A_{1}'+B_{2}'\equiv B_{1}'+A_{2}'$\\

The proof is just manipulation of effective weighted sets. We
clearly have that;\\

$A_{1}+A_{2}'=A_{1}'+A_{2}$ and $B_{1}+B_{2}'=B_{1}'+B_{2}$ $(*)$\\

We then have;\\

\indent $A\equiv B$\indent iff \ \ \ \ \ \ \ $A_{1}+B_{2}\equiv B_{1}+A_{2}$(Definition 2.20)\\
\indent \ \ \ \ \ \ \ \ \ \ \ iff $A_{1}+A_{2}'+B_{2}\equiv
B_{1}+A_{2}+A_{2}'$
(Theorem 2.12)\\
\indent \ \ \ \ \ \ \ \ \ \ \ iff $A_{1}'+A_{2}+B_{2}\equiv
B_{1}+A_{2}+A_{2}'$ (by (*))\\
\indent \ \ \ \ \ \ \ \ \ \ \ iff $\ \ \ \ \ \ \
A_{1}'+B_{2}\equiv B_{1}+A_{2}'$
(Theorem 2.12)\\
\indent \ \ \ \ \ \ \ \ \ \ \ iff $A_{1}'+B_{2}+B_{1}'\equiv
B_{1}+B_{1}'+A_{2}'$ (Theorem 2.12)\\
\indent \ \ \ \ \ \ \ \ \ \ \ iff $A_{1}'+B_{1}+B_{2}'\equiv
B_{1}+B_{1}'+A_{2}'$ (by (*))\\
\indent \ \ \ \ \ \ \ \ \ \ \ iff $\ \ \ \ \ \ \
A_{1}'+B_{2}'\equiv
B_{1}'+A_{2}'$ (Theorem 2.12)\\
\end{rmk}

We then have;\\

\begin{theorem}{Transitivity of Linear Equivalence}\\

Let $C'$ be an algebraic curve. If $A,B,C$ are generalised
weighted sets on $C'$ of the same total multiplicity, then, if
$A\equiv B$ and $B\equiv C$, we must have that $A\equiv C$.

\end{theorem}

\begin{proof}
Let $\{A_{1},A_{2}\}$, $\{B_{1},B_{2}\}$ and $\{C_{1},C_{2}\}$ be
the effective weighted sets as given by Remarks 2.23. Then, by
Definition 2.24, we have that;\\

$(A_{1}+B_{2})\equiv (B_{1}+A_{2})$ and $(B_{1}+C_{2})\equiv
(C_{1}+B_{2})$\\

By Theorem 2.12, we have that;\\

$(A_{1}+B_{1}+B_{2}+C_{2})\equiv (C_{1}+B_{1}+B_{2}+A_{2})$\\

It then follows, by Definition 2.6, that there exists a
$g_{n}^{1}$, containing $(A_{1}+B_{1}+B_{2}+C_{2})$ and
$(C_{1}+B_{1}+B_{2}+A_{2})$ as weighted sets. Clearly
$(B_{1}+B_{2})$ is contained in the fixed branch contribution of
this $g_{n}^{1}$. Removing this fixed branch contribution, we
obtain;\\

$A_{1}+C_{2}\equiv C_{1}+A_{2}$\\

By Definition 2.20, we then have that $A\equiv C$ as required.

\end{proof}

It follows immediately from Theorem 2.11 and Theorem 2.22 that;\\

\begin{theorem}
Let $C$ be a projective algebraic curve, then $\equiv$ is an
equivalence relation on generalised weighted sets for $C$ of a
given total multiplicity.
\end{theorem}

\begin{rmk}
Again, the definition of linear equivalence that we have given for
generalised weighted sets on a smooth projective algebraic curve
$C$ is equivalent to the modern definition for divisors. More
precisely, we have;\\

Modern Definition; Let $A$ and $B$ be divisors on a smooth
projective algebraic curve $C$, then $A\equiv B$ iff $A-B=div(g)$,
for some $g\in L(C)^{*}.$\\

See, for example, p161 of \cite{Shaf} for relevant definitions and
notation. In order to show that our definition is the same, use
Remarks 2.8 and the following simple argument;\\

$A\equiv B$ as generalised weighted sets iff $A_{1}+B_{2}\equiv
B_{1}+A_{2}$\\

where $\{A_{1},A_{2},B_{1},B_{2}\}$ are the effective weighted
sets given by Definition 2.24. Then;\\

$A_{1}+B_{2}\equiv B_{1}+A_{2}$ iff
$(A_{1}+B_{2})-(B_{1}+A_{2})=div(g)$ $(g\in L(C)^{*})$\\

by Remarks 2.8, where $div(g)$ is the modern definition. By a
straightforward calculation, we have that;\\

$(A_{1}+B_{2})-(B_{1}+A_{2})=A-B$ as divisors or generalised
weighted\\
\indent \ \ \ \ \ \ \ \ \ \ \ \ \ \ \ \ \ \ \ \ \ \ \ \ \ \ \ \ \
\ \ \ \ \ \ \ \ \ \ \  sets.

Hence, the notions of equivalence coincide.

\end{rmk}

We also have;\\

\begin{theorem}{Linear Equivalence Preserved by Addition}\\

Let $C'$ be a projective algebraic curve and suppose that
$\{A,B,C,D\}$ are generalised weighted sets on $C'$ with;\\

$A\equiv B$ and $C\equiv D$\\

then;\\

$A+C\equiv B+D$\\

\end{theorem}

\begin{proof}
Let $\{A_{1},A_{2}\}$, $\{B_{1},B_{2}\}$, $\{C_{1},C_{2}\}$ and
$\{D_{1},D_{2}\}$ be effective weighted sets as given by Remarks
2.19. Then, by Definition 2.20, we have that;\\

$A_{1}+B_{2}\equiv B_{1}+A_{2}$ and $C_{1}+D_{2}\equiv
D_{1}+C_{2}$\\

Hence, by Theorem 2.20;\\

$A_{1}+B_{2}+C_{1}+D_{2}\equiv B_{1}+A_{2}+D_{1}+C_{2}$ $(*)$\\

We clearly have that;\\

$A+C=(A_{1}+C_{1})-(A_{2}+C_{2})$ and $B+D=(B_{1}+D_{1})-(B_{2}+D_{2})$\\

as an identity of generalised weighted sets. Moreover, as\\
$(A_{1}+C_{1}),(A_{2}+C_{2}),(B_{1}+D_{1})$ and $(B_{2}+D_{2})$
are all effective, we can apply Remarks 2.23 and $(*)$ to obtain
the result.
\end{proof}

We now make the following definition;\\

\begin{defn}
Let $G$ be a generalised weighted set on a projective algebraic
curve $C$, then we define $|G|$ to be the collection of
generalised weighted sets $G'$ with $G'\equiv G$. We define
$order(|G|)$ to be the total multiplicity (possibly negative) of
any generalised weighted set in $|G|$.
\end{defn}

\begin{rmk}
If $G$ is an \emph{effective} weighted set, the collection defined
by Definition 2.30 is \emph{not} the same as the collection given
by Definition 2.17, as it includes virtual weighted sets. Unless
otherwise stated, we will use Definition 2.17 for \emph{effective}
weighted sets. This convention is in accordance with the Italian
terminology, we hope that this will not cause too much confusion
for the reader.

\end{rmk}

We now show that the notions of linear equivalence introduced
in this section are birationally invariant;\\

\begin{theorem}
Let $\Phi:C_{1}\leftrightsquigarrow C_{2}$ be a birational map.
Let $A$ and $B$ be generalised weighted sets on $C_{2}$, with
corresponding generalised weighted sets $[\Phi]^{*}A$ and
$[\Phi]^{*}B$ on $C_{1}$. Then $A\equiv B$, in the sense of either
Definition 2.6 or 2.24, iff $[\Phi]^{*}A\equiv [\Phi]^{*}B$.

\end{theorem}

\begin{proof}
Suppose that $A\equiv B$ in the sense of Definition 2.6. Then,
there exists a $g_{n}^{r}$ on $C_{2}$ containing $A$ and $B$ as
weighted sets. By Theorem 1.14, there exists a corresponding
$g_{n}^{r}$ on $C_{1}$, containing $[\Phi]^{*}A$ and $[\Phi]^{*}B$
as weighted sets. Hence, again by Definition 2.6,
$[\Phi]^{*}A\equiv [\Phi]^{*}B$. The converse is similar, using
$[\Phi^{-1}]^{*}$. If $A\equiv B$ in the sense of Definition 2.24,
then the same argument works.

\end{proof}

As a result of this theorem, we introduce the following
definition;\\

\begin{defn}
Let $\Phi:C_{1}\leftrightsquigarrow C_{2}$ be a birational map.
Then, given a generalised weighted set $A$ on $C_{2}$, we
define;\\

$[\Phi]^{*}|A|=|[\Phi]^{*}A|$\\

where, in the case that $A$ is effective, $|A|$ can be taken
either in the sense of Definition 2.17 or Definition 2.30.

\end{defn}

\begin{rmk}
The definition depends only on the complete series $|A|$, rather
than its particular representative $A$. This follows immediately
from Definition 2.17, Definition 2.30 and Theorem 2.32.

\end{rmk}

We finally introduce the following definition;\\

\begin{defn}{Summation of Complete Series}\\

Let $A$ and $B$ be generalised weighted sets, defining complete
series $|A|$ and $|B|$, in the sense of Definition 2.30. Then, we
define the sum;\\

$|A|+|B|$\\

to be the complete series, in the sense of Definition 2.30,
containing all generalised weighted sets of the form $A'+B'$ with
$A'\in |A|$ and $B'\in |B|$. If $A$ and $B$ are \emph{effective}
weighted sets with $|A|$, $|B|$ taken in the sense of Definition
2.17, then we make the same definition for the sum in the sense of
Definition 2.17.\\

\end{defn}

\begin{rmk}
This is a good definition by Theorem 2.13 and Theorem 2.29.
\end{rmk}

\begin{defn}{Difference of Complete Series}\\

Let $A$ and $B$ be generalised weighted sets, defining complete
series $|A|$ and $|B|$, in the sense of Definition 2.30. Then, we
define the difference;\\

$|A|-|B|$\\

to be the complete series, in the sense of Definition 2.30,
containing all generalised weighted sets of the form $A'-B'$ with
$A'\in |A|$ and $B'\in |B|$. If $A$ and $B$ are \emph{effective}
weighted sets with $|A|$, $|B|$ taken in the sense of Definition
2.17, then we can in certain cases define a difference in the
sense of Definition 2.17. (This is called the residual series, the
reader can look at \cite{Sev} for more details)

\end{defn}

\begin{rmk}
This is again a good definition, for generalised weighted sets
$\{A,B\}$, it follows trivially from the previous definition and
the fact that $\{A,-B\}$ are also generalised weighted sets.
\end{rmk}

\end{section}
\begin{section}{A geometrical definition of the genus of an
algebraic curve}
\end{section}

The purpose of this section is to give a geometrical definition of
the genus of an algebraic curve. In the case of a singular curve,
this cannot be achieved using purely algebraic methods. Our treatment follows the
presentation of Severi in \cite{Sev}.\\

We begin with the following lemma;\\

\begin{lemma}
Let $C$ be a projective algebraic curve, and suppose that a
$g_{n}^{1}$ is given on $C$, with no fixed branch contribution.
Then there exist a finite number of weighted sets $W_{\lambda}$ in
the $g_{n}^{1}$, possessing multiple branches.

\end{lemma}

\begin{proof}
We may assume that the $g_{n}^{1}$ is defined by a pencil
$\Sigma$, having finite intersection with $C$. Let
$\theta(\lambda)$ be the statement;\\

$\theta(\lambda)\equiv\forall y[(y\in \phi_{\lambda}\cap
C)\rightarrow y\in NonSing(C)\wedge
R.Mult_{y}(C,\phi_{\lambda})=1]$\\

See also Lemma 2.17 of \cite{depiro6}. Then $\theta$ defines a
constructible condition on $Par_{\Sigma}$ and moreover, using
Lemma 2.17 of \cite{depiro6}, we have that $\theta$ holds on an
open subset $U\subset Par_{\Sigma}$. For $\lambda\in U$, we have
that the intersection $(C\cap\phi_{\lambda})$ is transverse, in
the sense of Lemma 2.4 of \cite{depiro6}, and is contained in $W$.
Now, using Lemma 5.29 of \cite{depiro6} and the definition of
$(C\sqcap\phi_{\lambda})$, it follows that each branch of the
corresponding weighted set $W_{\lambda}$ is counted once and lies
inside $W$. Hence, if $W_{\lambda}$ is a weighted set, possessing
multiple branches, we must have that $\lambda\in
({Par_{\Sigma}\setminus U})$. As $Par_{\Sigma}$ has dimension $1$,
this is a finite set, hence the result follows.
\end{proof}

We now make the following definition;\\

\begin{defn}
Let $C$ be a projective algebraic curve and suppose that a
$g_{n}^{1}$ is given on $C$, with no fixed branch contribution.
Then we define the Jacobian of this $g_{n}^{1}$ to be the weighted set;\\

$Jac(g_{n}^{1})=\{\alpha_{\gamma_{1}},\ldots,\alpha_{\gamma_{j}},\ldots,\alpha_{\gamma_{r}}\}$\\

where $\{\gamma_{1},\ldots,\gamma_{j},\ldots,\gamma_{r}\}$
consists of the finitely many branches which are multiple for some
weighted set $W_{\lambda_{j}}$ of the $g_{n}^{1}$ and $\gamma_{j}$
appears with multiplicity $\alpha_{\gamma_{j}}+1$ in
$W_{\lambda_{j}}$.

\end{defn}

\begin{rmk}
This is a good definition by Lemma 3.1 and the fact that any branch
$\gamma$ can only appear in one weighted set $W_{\lambda}$, using
the hypothesis that the $g_{n}^{1}$ has no fixed branches.

\end{rmk}

We now analyse further the Jacobian of a $g_{n}^{1}$, with
\emph{no} fixed branch contribution. Using Theorem 1.14 of this
paper and Theorem 4.16 of \cite{depiro6}, we will derive general
results for projective algebraic curves from consideration of the
case where $C$ is a plane projective curve, having at most nodes
as singularities, $(\dag)$. Until the end of
Theorem 3.21, this assumption will be in force.\\

\begin{lemma}
Let the $g_{n}^{1}$, without fixed branch contribution, be given
on $C$. Then there exists a rational function $h$ on $C$,
defining this $g_{n}^{1}$, such that the weighted set;\\

$G=(h=\infty)$\\

\noindent consists of $n$ distinct branches, each counted once,
lying
inside\\
$NonSing(C)$.

\end{lemma}

\begin{proof}

By Theorem 2.7, we may assume that the given $g_{n}^{1}$ is
defined by $(g)$ for some rational function $g$ on $C$. Using the
proof of Lemma 3.1, we may assume, that, for generic $\lambda\in
P^{1}$, the weighted set $(g=\lambda)$ consists of $n$ distinct
branches, lying inside $NonSing(C)$, each counted once. Using the
proof of Theorem 2.7, we can find a Mobius transformation $\alpha$
of $P^{1}$, taking $\lambda$ to $\infty$, such that $h=\alpha\circ
g$ also defines the given $g_{n}^{1}$ and such that
$G=(h=\infty)=(g=\lambda)$. The result follows.

\end{proof}

We now show the following, the reader should refer to \cite{depiro6} for the relevant notation;\\

\begin{lemma}
Suppose that $deg(C)=m$, then there exists a homographic change of
variables of $P^{2}$, such that, in this new coordinate system $(x',y')$;\\

(i). The line at $\infty$ cuts $C$ transversely in $m$ distinct
non-singular points.\\

(ii). The tangent lines to $C$ parallel to the $y'$-axis all have
$2$-fold contact (contatto), and are based at non-singular points of $C$.\\

(iii). The branches of $J=Jac(g_{n}^{1})$ and $G=(h=\infty)$ are
all in finite position, with base points distinct from the points
of contact in $(ii)$.\\

\end{lemma}

\begin{proof}
We use the fact that a generic point of $C$ has character $(1,1)$.
The proof of this result requires duality arguments, which may be
found later in the paper. It follows, by Remark 6.6 of
\cite{depiro6}, that there exist only finitely many non-ordinary
branches. Hence, there exist finitely many tangent lines
$\{l_{\gamma_{1}},\ldots,l_{\gamma_{r}}\}$, based at
$\{p_{1},\ldots,p_{r}\}$, (possibly with repetitions), such
that;\\

$I_{\gamma_{j}}(C,p_{j},l_{\gamma_{j}})\geq 3$, $(for\ 1\leq j\leq
r)$\\

By assumption, $C$ has at most finitely many nodes as
singularities. Let $\{q_{1},\ldots,q_{s}\}$ be the base points of
these nodes and suppose that
$\{l_{\gamma_{q_{1}^{1}}},l_{\gamma_{q_{1}^{2}}},\ldots,l_{\gamma_{q_{s}^{1}}},l_{\gamma_{q_{s}}^{2}}\}$
are the $2s$ tangent lines (possibly with repetitions)
corresponding to these nodes. Let
$\{l_{\gamma_{1}'},\ldots,l_{\gamma_{t}'}\}$ define the tangent
lines to each of the branches appearing in $(Jac(g_{n}^{1})\cup
G)$. Now choose a point $P$ not lying on $C$ or any of the above
defined tangent lines. Let $\Sigma=\{l_{\lambda}^{P}\}_{\lambda\in
P^{1}}$ be the pencil defined by all lines passing through the
point $P$. Then $\Sigma$ defines a $g_{m}^{1}$ on $C$ without
fixed branches. By the proof of Lemma 3.1, for generic $\lambda$,
$l_{\lambda}^{P}$ intersects $C$ transversely in $m$ distinct
branches, based at non-singular points of $C$. Moreover, we may
assume, using the fact that the pencil has no base branches, that
$(C\sqcap l_{\lambda}^{P})$ is disjoint from $Jac(g_{n}^{1})$ and
$G$ $(*)$. By construction, we also have that, if $l_{Px}$ belongs
to $\Sigma$ and defines the tangent line $l_{\gamma_{x}}$ to a
branch $\gamma_{x}$ based at $x$, in the sense of Definition 6.3
of \cite{depiro6}, then $x\in NonSing(C)$, $l_{Px}$ has $2$-fold
contact (contatto) with the branch $\gamma_{x}$, $(**)$, and the
branch $\gamma_{x}$ does not appear in $(Jac(g_{n}^1)\cup
G)$,(***). Now choose a homography, sending the point $[0:1:0]$
and the line $Z=0$, in the original coordinates $[X:Y:Z]$, to $P$
and $l_{\lambda}^{P}$. Let $[X':Y':Z']$ be the new coordinate
system defined by this homography. For the affine coordinate
system $(x',y')$, defined by $x'={X'\over Z'}$ and $y'={Y'\over
Z'}$, we have that the line at $\infty$ has the property $(i)$,
and, by $(*)$, the branches of $Jac(g_{n}^{1})$ and $G$ are all in
finite position. The lines parallel to the $y'$-axis correspond to
the lines, exluding $l_{\lambda}^{P}$, in the pencil defined by
$\Sigma$, in this new coordinate system. Hence, $(ii)$ follows
immediately from $(**)$ and $(iii)$ then follows from $(***)$. The
lemma is proved.

\end{proof}

We now claim the following;\\

\begin{lemma}
Let $C$ be given in the coordinate system defined by Lemma 3.5,
henceforth denoted by $(x,y)$. Then the $g_{m}^{1}$ on $C$, given
by the lines parallel to the $y$-axis, and the line at $\infty$,
is defined by $(x)$, as in Lemma 2.4. Moreover,
$J'=Jac(g_{m}^{1})$ consists exactly of the branches of contact
between $C$ and the lines parallel to the $y$-axis, while
$G'=(x=\infty)$ consists of $m$ distinct branches, centred at
non-singular points of $C$.

\end{lemma}

\begin{proof}
In order to prove the first claim, first observe that, by its
construction, the given $g_{m}^{1}$ has no fixed branch
contribution. Moreover, it is generated by the lines $X=0$ and
$Z=0$, that is defined by the pencil of lines $(X-\lambda
Z)_{\lambda\in P^{1}}$. By Lemma 2.4, because $x$ is represented
by $({X\over Z})$, as a rational function on $C$, we have that the
series $(x)$ is defined by $(X-\lambda Z)_{\lambda\in P^{1}}$,
after removing its fixed branch contribution. As this series has
no fixed branch contribution, the result follows. In order to
compute $J'=Jac(g_{m}^{1})$, we need to determine;\\

$\{\lambda\in P^{1}:C\sqcap (X-\lambda Z)$\ contains\ a\ multiple\
branch$\}$\\

This corresponds to the set;\\

$\{\lambda\in P^{1}:I_{\gamma_{p}}(C,p,(X-\lambda Z))\geq 2\}$,
for some $p\in C\cap (X-\lambda Z)$, $\gamma_{p}$\\
\indent \ \ \ \ \ \ \ \ \ \ \ \ \ \ \ \ \ \ \ \ \ \ \ \ \ \ \ \ \ \ \ \ \ \ \ \ \ \ \ \ \ \ \ \ \ \ \  a\ branch\ at\ $p$.\\

As $C$ has at most nodes as singularities, the order of each
branch $\gamma$ on $C$ is $1$, see Definition 6.3 of
\cite{depiro6}. Using Theorem 6.2 of \cite{depiro6}, we then have
that $I_{\gamma_{p}}(C,p,(X-\lambda Z))\geq 2$ iff $(X-\lambda Z)$
is the tangent line $l_{\gamma_{p}}$ of $\gamma_{p}$. By
construction of the $g_{m}^{1}$, this can only occur if
$(X-\lambda Z)$ is the tangent line to an ordinary branch, see
Definition 6.3 of \cite{depiro6} again, centred at a non-singular
point of $C$, $(*)$. When $\lambda=\infty$, the corresponding line
in the pencil is given by the line at $\infty$, which cuts $C$
transversely, hence this possibility is excluded. Therefore, the
only possible values of $\lambda$ occur for lines parallel to the
$y$-axis. By $(*)$, for such a branch $\gamma_{p}$ of contact, we
have;\\

$I_{\gamma_{p}}(C,p,(X-\lambda Z))=2$\\

Now, by definition, $mult_{\gamma_{p}}(Jac(g_{m}^{1}))=2-1=1$.
Hence, the result follows. For the final part of the lemma, it is
easy to verify that $\{\gamma\in C:val_{\gamma}(x)=\infty\}$
correspond to the branches of intersection between $C$ and the
line at $\infty$. The result then follows by $(i)$ of Lemma 3.5.\\

\end{proof}

We now claim;\\

\begin{lemma}

Let $\gamma$ be a branch of $C$, in finite position, with
coordinates $(a,b)$, which does \emph{not} belong to
$J'=Jac(g_{m}^{1})$. Then one can find a
power series representation of $\gamma$ of the form;\\

$y(x)=b+c_{1}(x-a)+c_{2}(x-a)^{2}+\ldots$ $(*)$\\

\end{lemma}

\begin{rmk}
In the representation $(*)$, given in Lemma 3.7, we have slightly
abused the terminology of Theorem 6.2 in \cite{depiro6}. We mean,
here, that $(x,y(x))$ should parametrise the branch $\gamma$, in
the sense that, for any algebraic function $F(x,y)$,
$F(x,y(x))\equiv 0$ iff $F$ vanishes on $C$, otherwise $F$ has
finite intersection with $C$ and;\\

$ord_{(x-a)}F(x,y(x))=ord_{(x-a)}F(a+(x-a),b+c_{1}(x-a)+\ldots)$\\
$\indent \ \ \ \ \ \ \ \ \ \ \ \ \ \ \ \ \ \ \ \ \ \
=I_{\gamma_{(a,b)}}(C,F)$

\end{rmk}

\begin{proof}{(Lemma 3.7)}\\

First make the linear change of coordinates $x'=x-a$ and $y'=y-b$,
so that, in this new coordinate system, the branch $\gamma$ is
centred at $(0,0)$. Using Theorem 6.2 of \cite{depiro6}, we can
find algebraic power series $\{x'(t),y'(t)\}$, parametrising the
branch $\gamma$ in the coordinate system $(x',y')$, of the form;\\

$x'(t)=a_{1}t+a_{2}t^{2}+\ldots$\\

$y'(t)=b_{1}t+b_{2}t^{2}+\ldots$\\

It is then a trivial calculation to check that;\\

$x(t)=a+a_{1}t+a_{2}t^{2}+\ldots$\\

$y(t)=b+b_{1}t+b_{2}t^{2}+\ldots$\\

parametrises the branch $\gamma$ in the coordinate system $(x,y)$,
with the terminology similar to the previous remarks.

Using the fact that the branch has order $1$, which is preserved
by the homographic change of coordinates, the vector
$(a_{1},b_{1})\neq 0$. If $a_{1}=0$, then the tangent line
$l_{\gamma}$ to the branch, in the coordinate system $(x',y')$,
would be parallel to the $y'$-axis, hence the translation of
$l_{\gamma}$ by $(a,b)$, which is the tangent line to $\gamma$ in
the coordinate system $(x,y)$ would be parallel to the $y$-axis.
Therefore, by Lemma 3.6, $\gamma$ would belong to
$J'=Jac(g_{m}^{1})$, contradicting the assumption of the lemma.
Hence, we can assume that $a_{1}\neq 0$. As ${dx'\over
dt}|_{t=0}\neq 0$, we can apply the inverse function theorem to
the power series $x'(t)$ and find an algebraic power series
$t(x')$, with ${dt\over dx'}|_{x'=0}\neq 0$, such that $x'(t(x'))=x'$. Then we can write;\\

$y'(t(x'))=b_{1}t(x')+b_{2}t(x')^{2}+\ldots$\\

We claim that the sequence $(x',y'(t(x')))$ parametrises the
branch $\gamma$ in the coordinate system $(x',y')$, $(*)$. We
clearly
have that, for any algebraic function $F(x',y')$;\\

$F(x'(t),y'(t))\equiv 0$ iff $F(x'(t(x')),y'(t(x')))\equiv 0$ iff
$F(x',y'(t(x')))\equiv 0$\\

If $ord_{t}F(x'(t),y'(t))=m<\infty$, then we have;\\

$F(x'(t),y'(t))=t^{m}u(t)$, for a unit $u(t)\in L[[t]]$.\\

We then have that;\\

$F(x',y'(t(x')))=t(x')^{m}u(t(x'))$, $(1)$\\

As ${dt\over dx'}|_{x'=0}\neq 0$, we have that;\\

$t(x')=x'v(x')$, for a unit $v(x')\in L[[x']]$, $(2)$\\

Combining $(1)$ and $(2)$ gives that;\\

$F(x',y'(t(x')))=(x'v(x'))^{m}u(x'v(x'))=(x')^{m}v(x')^{m}u(x'v(x')),$\\

\indent \ \ \ \ \ \ \ \ \ \ \ \ \ \ \ \ \ \ \ \ \ \ \ \ \ \ \
where
$v(x')^{m}u(x'v(x'))$ is a unit in $L[[x']]$\\

\noindent This implies that $ord_{x'}F(x',y'(t(x')))=m$ as well,
hence $(*)$ is shown. It follows easily that the sequence
$(a+x',b+y'(t(x')))$ parametrises the branch $\gamma$ in the
coordinate system $(x,y)$, with the above extension of
terminology. Hence, if we let $y(x)=b+y'(t(x-a))$, then so does
the sequence $(x,y(x))$, with the convention
of Remarks 3.8. The lemma is then shown.\\

\end{proof}

Now let $C$ be defined in the coordinate system $(x,y)$ by $f=0$
and let $h$ be the rational function, given by Lemma 3.4, in this
coordinate system. With hypotheses as in Lemma 3.7, any rational
function $\theta$ formally determines a Cauchy series
$\theta(x,y(x))$, in $(x-a)$, see the explanation at the beginning
of Section 2, which we will also denote by $\theta$.\\

We have that;\\

$f(x,y(x))\equiv 0$\\

and, hence;\\

$f_{x}+f_{y}{dy\over dx}=0$, ${dy\over dx}=-{f_{x}\over f_{y}}$\\

${dh\over dx}=h_{x}+h_{y}{dy\over dx}=h_{x}+h_{y}\centerdot
-{f_{x}\over
f_{y}}={h_{x}f_{y}-h_{y}f_{x}\over f_{y}}$ $(*)$\\

\noindent The calculation $(*)$ should be justified carefully at
the level of Cauchy series. The first part and the case when $h$
belongs to the polynomial ring $L[x,y]$ was considered in Lemma
2.10 of \cite{depiro6}, $(**)$. In general, we can find
$\{h_{1},h_{2}\}$ in $L[x,y]$ such that $h={h_{1}\over h_{2}}$.
Using the result $(**)$ and the quotient rule for differentiating
rational functions, it is sufficient to check that for the formal
(algebraic) power series in $(x-a)$, determined by $h_{2}$;\\

${d(1/h_{2})\over dx}={dh_{2}\over dx}\centerdot {-1\over
(h_{2})^{2}}$ $(***)$\\

This can be shown by a similar calculation to that in Lemma 2.10
of \cite{depiro6}. Namely, we can find a sequence of polynomials
$\{h_{2}^{m}\}_{m\geq 0}$ in $(x-a)$, converging to $h_{2}$ in the
power series ring $L[[x-a]]$. The result $(***)$ holds for each
$\{h_{2}^{m}\}$, hence, by general continuity results for
multiplication in $L[[x-a]]$, it is sufficient to show that;\\

${d(1/h_{2}^{m})\over dx}\rightarrow {d(1/h_{2})\over dx}$ and ${dh_{2}^{m}\over dx}\rightarrow {dh_{2}\over dx}$\\

The second part of this calculation was done in Lemma 2.10 of
\cite{depiro6}, (even in non-zero characteristic). The first part
follows from the second part by representing $(1/h_{2})$ as
$(x-a)^{-n}u_{2}(x-a)$, for some $n\geq 0$ and a unit
$u_{2}(x-a)\in L[[x-a]]$, and finding a sequence of units
$\{u_{2}^{m}(x-a)\}_{m\geq 1}$ in $L[[x-a]]$ such that
$u_{2}^{m}\rightarrow u_{2}$ and
$(1/h_{2}^{m})=(x-a)^{-n}u_{2}^{m}$. One can then use the product
rule for an algebraic power series and the function $(x-a)^{-n}$,
along with continuity of addition in $L[[x-a]]$.\\

Using $(*)$, we can consider ${dh\over dx}$ as a rational function
on the curve $C$ and, for a branch $\gamma$, define
$val_{\gamma}({dh\over dx})$ and $ord_{\gamma}({dh\over dx})$. For
convenience, we will use the notation $ord_{\gamma}({dh\over
dx})=-k$, for a positive integer $k$, to mean that
$val_{\gamma}({dh\over dx})=\infty$ and $ord_{\gamma}({dh\over
dx})=k$. We now claim the following;\\

\begin{lemma}
Let $\gamma$ be a branch of $C$, distinct from $G\cup G'\cup J'$,
then, for the $g_{n}^{1}$ defined by Lemma 3.1, $\gamma$ is
counted (contato) $s$-times in some weighted set iff
$ord_{\gamma}({dh\over dx})=s-1$. If $\gamma$ belongs to $G$, then
$ord_{\gamma}({dh\over dx})=-2$.

\end{lemma}

\begin{proof}
For the first part of the lemma, suppose that $\gamma$ is counted
$s$-times in some weighted set. Then, by Lemma 3.4,
$val_{\gamma}(h)<\infty$ and $ord_{\gamma}(h)=s$. By Lemma 3.7 and
the construction before Lemma 2.1, $h$ determines an algebraic
power series, at the branch $\gamma$,
of the form;\\

$h=\lambda+(x-a)^{s}\psi(x-a)$, with $\psi(0)\neq 0,\lambda<\infty$\\

We then have that;\\

${dh\over dx}=(x-a)^{s-1}[s\psi(x-a)+(x-a)\psi'(x-a)]$\\

At $x=a$, the expression in brackets reduces to $s\psi(0)\neq 0$,
hence $ord_{\gamma}({dh\over dx})=s-1$ as required. The converse
statement is also clear by this calculation. If $\gamma$ belongs
to $G$, then, by Lemma 3.4, $val_{\gamma}(h)=\infty$ and
$ord_{\gamma}(h)=-1$. Then $h$ determines an algebraic power
series at $\gamma$ of the form;\\

$h=(x-a)^{-1}\psi(x-a)$, with $\psi(0)\neq 0$\\

We then have that;\\

${dh\over dx}=-(x-a)^{-2}[\psi(x-a)-(x-a)\psi'(x-a)]$\\

At $x=a$, the expression in brackets reduces to $\psi(0)\neq 0$,
hence $ord_{\gamma}({dh\over dx})=-2$ as required.
\end{proof}
It remains to consider the branches $G'\cup J'$. We achieve this
by the following geometric constructions;\\

$(i)$. Construction for $G'$;\\

We use the change of variables $x={1\over x'}$ and $y={y'\over
x'}$. This is a homography as the map;\\

$\Theta:(x',y')\mapsto ({1\over x'},{y'\over x'})$\\

is the restriction to affine coordinates of the map;\\

$\Theta:[X':Y':Z']\mapsto [Z':Y':X']$\\

As $\Theta$ is a homography, the character of corresponding
branches is preserved.  Let $F(x',y')$ define the equation of $C$
in this new coordinate system. The point $P$, given by $[0:1:0]$,
relative to the coordinate system $(x,y)$, is fixed by this
homography, hence the weighted set $G'_{1}$, corresponding to
$G'$, consists of branches in finite position relative to the
coordinate system $(x',y')$, and is defined by $C\sqcap (x'=0)$.
As the branches of $G'_{1}$ are simple, they cannot coincide with
any of the branches of contact of tangents to $C$ parallel to the
$y'$-axis. Hence, we can apply the power series method given by
Lemma 3.7. Let $H(x',y')$ be the rational function corresponding
to $h(x,y)$ from Lemma 3.4. We
claim the following;\\

\begin{lemma}
Let hypotheses be as in $(i)$. Then, for every branch $\gamma$ of
$G_{1}'$, we have that $ord_{\gamma}({dH\over dx'})=0$.

\end{lemma}

\begin{proof}
By $(i)$ and $(iii)$ of Lemma 3.5, we have that $G=(h=\infty)$ and
$G'$ are disjoint. Hence, for every branch $\gamma$ of $G'$,
$val_{\gamma}(h)<\infty$. For a given branch $\gamma$ of $G'$, let
$val_{\gamma}(h)=c$ and consider the rational function $h-c$. We
clearly have that $val_{\gamma}(h-c)=0$. If $ord_{\gamma}(h-c)\geq
2$, then $\gamma$ would be multiple for $(h=c)$, hence, by Lemma
3.4, would belong to $J=Jac(g_{n}^{1})$. This contradicts the
fact, from (i) of Lemma 3.5, that $J$ and $G'$ are disjoint.
Therefore, we must have that $ord_{\gamma}(h-c)=1$. It follows
that, for any branch $\gamma$ of $G_{1}'$, we can find a constant
$c$ such that $ord_{\gamma}(H-c)=1$ as well. Now, if $\gamma$ is
centred at $(a,b)$ in the coordinate system $(x',y')$, then, using
the method of Lemma 3.7, $H$ determines an algebraic power series
at $\gamma$
of the form;\\

$H(x',y'(x'))=c+c_{1}(x'-a)+\ldots$ $(c_{1}\neq 0)$\\

It follows, from differentiating this expression, that
$ord_{\gamma}({dH\over dx'})=0$ as required.

\end{proof}

We now claim;\\

\begin{lemma}
Let hypotheses be as in $(i)$. Then;\\

${dH\over dx'}={dh\over dx}.{dx\over dx'}$\\

as an identity of Cauchy series, for corresponding branches
satisfying the requirements that they are in finite position and
are not branches of contact for tangent lines parallel to the
$y'$-axis or $y$-axis respectively.
\end{lemma}

\begin{proof}
Let $\gamma'$ and $\gamma$ be such corresponding branches, centred
at $p'=(a,b)$ and $p=({1\over a},{b\over a})$, of $F=0$ and $f=0$
respectively. Let $\gamma'$ be parametrised, as in Lemma 3.7, by
the sequence $(x',y'(x'))$. Then, we claim that the sequence;\\

$(x(x',y'(x')),y(x',y'(x')))=(x(x'),y(x',y'(x')))=({1\over x'},{y'(x')\over x'})$\\

parametrises the corresponding branch $\gamma$, in the sense of
Remarks 3.8, $(*)$. This follows easily from the fact that the
morphism $\Theta$, given in $(i)$, is a homography of $P^{2}$,
hence, in particular, one has that, for any algebraic function $\theta$;\\

$I_{\gamma'}(p',C,\Theta^{*}(\theta))=I_{\gamma}(p,C,\theta)$\\

We now claim that;\\

$(x(x'),y(x',y'(x')))=(x(x'),y(x(x')))$, $(**)$\\

as an identity of algebraic power series in $(x'-a)$. By $(*)$ and
Lemma 3.7, both sequences define valid parametrisations,
$(\theta_{1}(x-a),\theta_{2}(x-a))$ and
$(\theta_{1}(x-a),\theta_{3}(x-a))$ in the sense of Remarks 3.8,
of $\gamma$. Moreover, we clearly have that the initial terms of
both sequences are identically equal. Now, observe that
$ord_{(x'-a)}(x(x')-{1\over a})=1$, hence we may apply the method
of Lemma 3.7 (Inverse Function Theorem), in order to find an
algebraic power series $t(z)\in zL[[z]]$ such that;\\

$\theta_{1}(t(x'-a))={1\over a}+(x'-a)$\\

and\\

$({1\over a}+(x'-a),\theta_{2}(t(x'-a)))$, $({1\over a}+(x'-a),\theta_{3}(t(x'-a)))$\\

 both define valid parametrisations of $\gamma$ in the sense of Remarks 3.8, $(***)$. Now, supposing that
$\theta_{2}(t(x'-a))=\theta_{3}(t(x'-a))$, then\\
$(\theta_{2}-\theta_{3})(t(x'-a))\equiv 0$, hence, it follows
straightforwardly that\\
 $(\theta_{2}-\theta_{3})(x-a)\equiv 0$
and $\theta_{2}(x-a)=\theta_{3}(x-a)$. Then $(**)$ is shown. We
may, therefore, assume that;\\

$\theta_{2}(t(x'-a))=\sum_{n=0}^{\infty}a_{n}(x'-a)^{n}$, $\theta_{3}(t(x'-a))=\sum_{n=0}^{\infty}b_{n}(x'-a)^{n}$\\

and $a_{n}\neq b_{n}$ for some $n\geq 1$. Consider the rational
function\\
 $\psi_{n}={1\over n!}{d^{n}y\over dx^{n}}$ on $C$,
given by the explanation after the proof of Lemma 3.7. By $(***)$
and Lemma 2.1, we would then have that both
$val_{\gamma}(\psi_{n})=a_{n}$ and $val_{\gamma}(\psi_{n})=b_{n}$,
which is clearly a contradiction. Hence, the claim $(**)$ is
shown.\\

Now, we have, by $(**)$, for the corresponding branches $\gamma'$ and $\gamma$;\\

$H(x',y'(x'))=h(x(x'),y(x',y'(x')))=h(x(x'),y(x(x')))$, $(\dag)$\\

Applying the chain rule for differentiating Cauchy series, $(\dag\dag)$, to the composition;\\

$H:x'\mapsto x(x')\mapsto h(x(x'),y(x(x')))$\\

and using $(\dag)$, the lemma follows. However, we should still
justify the use of $(\dag\dag)$ in the following form;\\

A Chain Rule for Cauchy Series.\\

Let $q(x')\in L(x')$ be a rational function, such that $q(a)$ is
defined as an element of $L$, and let $h(x)$ define a Cauchy
series in $Frac(L[[x]])$. Then $H(x')=h(q(x')-q(a))$
defines a Cauchy series in $Frac(L[[x'-a]])$ and;\\

${dH\over dx'}={dh\over dx}|_{(q(x')-q(a))}.{dq\over dx'}$, $(*)$\\

as an identity of Cauchy series in $Frac(L[[x'-a]])$.\\

In order to see the first part of the claim, observe that
$q(x')-q(a)$ can be expanded as an algebraic power series in
$(x'-a)L[[x'-a]]$, hence the formal substitution of $q(x')-q(a)$
in $h(x)$ determines a Cauchy series in $Frac(L[[x'-a]])$. For the
second part of the claim, choose a sequence $\{h_{m}\}_{m\geq 0}$
of \emph{polynomials} in $L[x,1/x]$ such that
$\{h_{m}\}\rightarrow h$ in the non-archimidean topology induced
by the standard valuation on the field $Frac(L[[x]])$. Let
$H_{m}(x')=h_{m}(q(x')-q(a))$ be the corresponding rational
function, which also determines a Cauchy series in the field
$Frac(L[[x'-a]])$. By the chain rule for rational functions, $(*)$
holds, replacing $H$ by $H_{m}$ and $h$ by $h_{m}$. Hence, the
identity $(*)$ also holds at the level of Cauchy series in
$Frac(L[[x'-a]])$. We have that the sequence $\{H_{m}\}_{m\geq
0}\rightarrow H$, in the non-archimidean topology induced by the
standard valuation on the field $Frac(L[[x'-a]])$. This follows
from the fact that, for sufficently large $m$;\\

$ord_{(x'-a)}(h-h_{m})(q(x')-q(a))\geq ord_{x}(h-h_{m})$\\

By calculations shown above (even in non-zero characteristic), we
have that $({dh_{m}\over dx})_{m\geq 0}\rightarrow {dh\over dx}$
and $({dH_{m}\over dx'})_{m\geq 0}\rightarrow {dH\over dx'}$. The
result $(*)$ then follows by continuity of multiplication in
$L[[x'-a]]$ and uniqueness of
limits.\\

\end{proof}

We now combine Lemmas 3.10 and Lemmas 3.11, in order to obtain;\\

\begin{lemma}
Let $C$ be given in the original coordinate system $(x,y)$, then,
for every branch $\gamma$ of $G'$, we have that
$ord_{\gamma}({dh\over dx})=2$.

\end{lemma}

\begin{proof}
By Lemma 3.11 and the fact that $x(x')={1\over x'}$, we have;\\

${dH\over dx'}=-{1\over x'^{2}}{dh\over dx}=-x^{2}{dh\over dx}$, $(*)$\\

This identity holds at the level of Cauchy series in $L[[x'-a]]$
for corresponding branches satisfying the conditions of Lemma
3.10. However, using Lemma 3.10, we can also consider $G={dH\over
dx'}+{1\over x'^{2}}{dh\over dx}$ as a rational function on the
curve $C$. By $(*)$, for an appropriate branch $\gamma$,
satisfying the conditions of Lemma 3.10, we have that the Cauchy
series expansion of $G$ at $\gamma$ is identically zero. By Lemma
2.1, this can only occur if $G$ vanishes identically on $C$.
Hence, we can assume that $(*)$ holds at the level of rational
functions on $C$ as well. By Lemma 3.10, we found that, for a
branch $\gamma$ of $G_{1}'$, $ord_{\gamma}({dH\over dx'})=0$.
Using the fact that the branches of $G_{1}'$ are centred at
$x'=0$, we have, straightforwardly, that, for such a branch
$\gamma$, $ord_{\gamma}(-x'^{2}{dH\over dx'})=2$. Using Lemma 2.3,
applied to the homography $\Theta$ given in $(i)$, we have that
$ord_{\gamma}({dh\over dx})=2$ for any given branch of $G'$, as
required.

\end{proof}

(ii). Construction for $J'$;\\

We use the change of variables $x=y'$ and $y=x'$. This is a
homography as the map;\\

$\Theta:(x',y')\rightarrow (y',x')$\\

is the restriction to affine coordinates of the map;\\

$\Theta:[X':Y':Z']\mapsto [Y':X':Z']$\\

Again, as $\Theta$ is a homography, the character of corresponding
branches is preserved. Let $F(x',y')$ denote the equation of $C$
in this new coordinate system. The point $P$, given by $[0:1:0]$,
relative to the coordinate system $(x,y)$, corresponds to the
point $[1:0:0]$ in the coordinate system $(x',y')$. Hence, the
weighted set $J_{1}'$, corresponding to $J'$, consists of branches
in finite position relative to the coordinate system $(x',y')$,
defined by the branches of contact between $C$ and tangents
parallel to the $x'$-axis. In particular, they \emph{cannot} be
branches of contact between $C$ and tangents parallel to the
$y'$-axis. Hence, we can again apply the power series method given
by Lemma 3.7. Let $H(x',y')$ be the rational function
corresponding to $h(x,y)$ from Lemma 3.4. We claim the
following;\\

\begin{lemma}
Let hypotheses be as in $(ii)$. Then, for every branch $\gamma$ of
$J_{1}'$, we have that $ord_{\gamma}({dH\over dx'})=0$.
\end{lemma}

\begin{proof}
The proof is similar to Lemma 3.10. By (iii) of Lemma 3.5, we have
that $G=(h=\infty)$ and $J'$ are disjoint. Hence, for every branch
$\gamma$ of $J'$, $val_{\gamma}(h)<\infty$. For a given branch
$\gamma$ of $J'$, let $val_{\gamma}(h)=c$ and consider the
rational function $h-c$. We then have that $val_{\gamma}(h-c)=0$.
If $ord_{\gamma}(h-c)\geq 2$, then $\gamma$ would be multiple for
$(h=c)$, hence, by Lemma 3.4, would belong to $J=Jac(g_{n}^{1})$.
This contradicts the fact, from $(iii)$ of Lemma 3.5, that $J$ and
$J'$ are disjoint. Therefore, we must have that
$ord_{\gamma}(h-c)=1$. It follows that, for any branch $\gamma$ of
$J_{1}'$, we can find a constant $c$ such that
$ord_{\gamma}(H-c)=1$ as well. Now, if $\gamma$ is centred at
$(a,b)$ in the coordinate system $(x',y')$, then, using the method
of Lemma 3.7, $H$ determines an algebraic power series at $\gamma$
of the form;\\

$H(x',y'(x'))=c+c_{1}(x'-a)+\ldots$ $(c_{1}\neq 0)$\\

By differentiating this expression, we have that
$ord_{\gamma}({dH\over dx'})=0$ as required.

\end{proof}

We now make an obvious extension to Lemma 3.7;\\

\begin{lemma}
Let $\gamma$ be a branch of $C$, in finite position, with
coordinates $(a,b)$, such that the tangent line of $\gamma$ is not
parallel to the $x$-axis. Then one can find a power series
representation of $\gamma$ of the form;\\

$x(y)=a+c_{1}(y-b)+c_{2}(y-b)^{2}+\ldots$ $(*)$\\

\end{lemma}

\begin{proof}
The proof is the same as Lemma 3.7.
\end{proof}

We can also make a similar extension to the remarks made between
the proof of Lemma 3.7 and Lemma 3.9. Namely, we can
define ${dx\over dy}$, ${dh\over dy}$ (as rational functions on $C$), and we have that;\\

${dx\over dy}=-{f_{y}\over f_{x}}$,  ${dh\over dy}={h_{y}f_{x}-h_{x}f_{y}\over f_{x}}$\\

We now claim the following;\\

\begin{lemma}
Let $C$ be given in the original coordinate system $(x,y)$ and let
$\gamma$ be a branch of $C$, centred at $(a,b)$, with the property
that its tangent line $l_{\gamma}$ is neither parallel to the
$x$-axis nor to the $y$-axis. Then the sequences $(x(y),y(x(y))$
and $(x(y),y)$ both parametrise the branch $\gamma$ and moreover;\\

$(x(y),y(x(y)))=(x(y),y)$\\

\noindent as an identity of sequences of algebraic power series in
$L[[y-b]]$.

\end{lemma}

\begin{proof}

By Lemmas 3.7 and 3.14, the sequences $(x,y(x))$ and $(x(y),y)$
both parametrise the branch $\gamma$, as algebraic power series,
in $L[[x-a]]$ and $L[[y-b]]$ respectively. Moreover, we claim
that;\\

$ord_{(y-b)}(x(y)-a)=1$, $(*)$\\

This follows from a close inspection of the proof of Lemma 3.7,
using the fact that $\{a_{1},b_{1}\}$ given there are both
non-zero, by the hypotheses on the branch $\gamma$. Now, by $(*)$
and the fact that $y(x)$ defines an algebraic power series in
$L[[x-a]]$, the substitution of $x(y)$ for $x$ in this power
series defines  an algebraic power series in $L[[y-b]]$. We need
to show that the sequence $(x(y),y(x(y)))$ still parametrises
$\gamma$, $(**)$. The proof is very similar to Lemma 3.7. Suppose
that $F(x,y)$ is an algebraic function. Then we have that;\\

$F(x,y(x))\equiv 0$ iff $F(x(y),y(x(y)))\equiv 0$ iff $F$ vanishes
on $C$.\\

If $ord_{(x-a)}F(x,y(x))=m<\infty$, then we have;\\

$F(x,y(x))=(x-a)^{m}u(x,a)$, where $u(x,a)$ is a unit in
$L[[x-a]]$.\\

We then have that;\\

$F(x(y),y(x(y)))=(x(y)-a)^{m}u(x(y),a)$, $(1)$\\

By $(*)$, we have that;\\

$(x(y)-a)=(y-b)v(y,b)$, where $v(y,b)$ is a unit in $L[[y-b]]$, $(2)$\\

Combining $(1)$ and $(2)$ gives that;\\

$F(x(y),y(x(y)))=(y-b)^{m}v(y,b)^{m}u((y-b)v(y,b))$, where $u(x)$
is a\\
\indent \ \ \ \ \ \ \ \ \ \ \ \ \ \ \ \ \ \ \ \ \ \ \ \ \ \ \ \ \ \ \ \ \ \ \ \ \ \ \ \ \ \ \ \ \ \ \ \ \ \ \ \ \ \ \ \ \ \ \ \ \ \ \ \ \ \ \ \ \ \  unit in $L[[x]]$\\

It is easily checked that $v(y,b)^{m}u((y-b)v(y,b))$ defines a
unit in $L[[y-b]]$. Hence;\\

$ord_{(y-b)}F(x(y),y(x(y)))=m$\\

 as well. This proves $(**)$ as required. In order to show the last part of the lemma,
 we use $(*)$ above and the method of Lemma 3.11, (uniqueness of parametrisation
 of a branch with given first term in the sequence, this was the proof of $(**)$ in that Lemma).
 The details are left to the reader. \\

\end{proof}

\begin{lemma}
Let $C$ be given in the original coordinate system $(x,y)$.
Then;\\

${dh\over dy}={dh\over dx}.{dx\over dy}$ $(*)$\\

where the identity can be taken either at the level of Cauchy
series, for branches $\gamma$ of $C$ with the property that they
are in finite position and their tangent line is neither parallel
to the $x$-axis nor to the $y$-axis, or at the level of rational
functions on $C$.

\end{lemma}

\begin{proof}
For a branch $\gamma$, centred at $(a,b)$, satisfying the
requirements of the lemma, we
consider the composition;\\

$h:y\mapsto x(y)\mapsto h(x,y(x))|_{x(y)}=h(x(y),y(x(y)))=h(x(y),y)$ $(\dag)$\\

where the final identity comes from Lemma 3.15. Applying a chain
rule for Cauchy series, $(\dag\dag)$, to $(\dag)$, we obtain the
result $(*)$ at the level of Cauchy series. Again, we justify
$(\dag\dag)$ in the following form, one should compare the
result with the version given in Lemma 3.11;\\

Another Chain Rule for Cauchy Series.\\

Let $x(y)$ define an algebraic power series in $L[[y-b]]$, with
constant term $x(b)=a$, and let $h(x)$ define a Cauchy series in
$Frac(L[[x-a]])$. Then $g(y)=h(x(y))$ defines a Cauchy series in
$Frac(L[[y-b]])$ and;\\

${dg\over dy}={dh\over dx}|_{x(y)}.{dx\over dy}$ $(**)$\\

as an identity of Cauchy series in $Frac(L[[y-b]])$.\\

The first part of the claim follows easily from the obvious fact that\\
$x(y)-a$ belongs to $(y-b)L[[y-b]]$, hence, as $h(x)$ is a Cauchy
series in $Frac(L[[x-a]])$, the formal substitution of $x(y)$ in
$h(x)$ (and ${dh\over dx}$) defines a Cauchy series in
$Frac(L[[y-b]])$. In order to show $(**)$, let
$\{x_{n}(y)\}_{n\geq 1}$ define a sequence of \emph{polynomials}
in $L[y-b]$ with the property that $\{x_{n}(y)\}_{n\geq
1}\rightarrow x(y)$ in the non-archimidean topology on $L[[y-b]]$
and $x_{n}(b)=a$. Let $g_{n}(y)=h(x_{n}(y))$ be the corresponding
Cauchy series in
$Frac(L[[y-b]])$. Then we have that;\\

${dg_{n}\over dy}={dh\over dx}|_{x_{n}(y)}.{dx_{n}\over dy}$\\

as an identity of Cauchy series in $Frac(L[[y-b]])$, by the proof
of the previous version of the Chain Rule in Lemma 3.11. Now, by
proofs given in the paper \cite{depiro6}, we have that
$({dx_{n}\over dy})_{n\geq 1}\rightarrow {dx\over dy}$ as
algebraic power series in $L[[y-b]]$. We now claim that ${dh\over
dx}|_{x_{n}(y)}\rightarrow {dh\over dx}|_{x(y)}$ as Cauchy series
in $Frac(L[[y-b]])$, $(***)$. In order to see this, observe that
we can write ${dh\over dx}=h_{0}+h_{1}$, where $h_{1}$ belongs to
$(x-a)L[[x-a]]$, and $h_{0}={dh\over dx}-h_{1}$ is a finite sum of
terms of order at most $0$. We then have, by explicit calculation, that;\\

$h_{1}=\sum_{m\geq 1}a_{m}(x-a)^{m}$\\

$h_{1}(x(y))-h_{1}(x_{n}(y))=\sum_{m\geq
1}a_{m}[(x(y)-a)^{m}-(x_{n}(y)-a)^{m}]$\\

\indent \ \ \ \ \ \ $=\sum_{m\geq
1}a_{m}[[(x_{n}(y)-a)+(x-x_{n})(y)]^{m}-(x_{n}(y)-a)^{m}]$\\

\indent \ \ \ \ \ \  $=\sum_{m\geq 1}a_{m}(x-x_{n})(y)r_{m}(y)$, with $ord_{(y-b)}r_{m}(y)\geq m-1$\\

Using the definition of convergence, given any $s\geq 1$, we can
find $n(s)$ such that $ord_{(y-b)}(x-x_{n(s)})(y)\geq s$. The
above calculation then shows that
$ord_{(y-b)}[h_{1}(x(y))-h_{1}(x_{n(s)}(y))]\geq s$ as well.
Hence, we must have that $\{h_{1}(x_{n}(y))\}_{n\geq 1}\rightarrow
h_{1}(x(y))$ in $L[[y-b]]$. As $h_{0}$ is a finite sum of terms,
it follows easily, by continuity of the basic operations
$\{+,.\}$, that $\{h_{0}(x_{n}(y))\}_{n\geq 1}\rightarrow
h_{0}(x(y))$ in $Frac(L[[y-b]])$. Hence, the result $(***)$
follows. A similar argument shows that $\{g_{n}(y)\}_{n\geq
1}\rightarrow g(y)$ in $Frac(L[[y-b]])$. Hence, using arguments in
\cite{depiro6}, we have that $({dg_{n}\over dy})_{n\geq
1}\rightarrow g$ in $Frac(L[[y-b]])$, as well. Now, $(**)$ follows
from uniqueness of limits and continuity of multiplication for the
non-archimidean topology on $Frac(L[[y-b]])$.

In order to complete the proof of Lemma 3.16, it remains to prove
that the identity $(*)$ may be taken at the level of rational
functions on $C$. This follows from the same argument given at the
beginning of Lemma 3.12.

\end{proof}

We now show the following;\\

\begin{lemma}
Let hypotheses be as in $(ii)$, then;\\

${dH\over dx'}={dh\over dy}={dh\over dx}{dx\over dy}=-{f_{y}\over f_{x}}{dh\over dx}$\\

where the identities may be taken either at the level of Cauchy
series, for corresponding branches satisfying the requirement that
they are in finite position and are not branches of contact for
tangent lines parallel to the $x'$-axis or $y'$-axis and tangent
lines parallel to the $x$-axis or $y$-axis respectively, or at the
level of rational functions on $C$.
\end{lemma}

\begin{proof}
Let $\gamma$ and $\gamma'$ be such corresponding branches, centred
at $(a,b)$ and $(b,a)$ respectively, in the coordinate systems
$(x',y')$ and $(x,y)$. By Lemma 3.7, we can find a
parametrisation $(x',y'(x'))$ of $\gamma$ such that;\\

${dH\over dx'}={d\over dx'}H(x',y'(x'))$, $(1)$\\

 as a Cauchy series in
$Frac(L[[x'-a]])$. As the morphism $\Theta$ given in $(ii)$ is a
homography, we have that the sequence;\\

$(x(x',y'(x')),y(x',y'(x')))=(y'(x'),x')$\\

also parametrises $\gamma'$ in the sense of Remarks 3.8, see the
corresponding proof of Lemma 3.11. Moreover, we have that;\\

$H(x',y'(x'))=h(y'(x'),x')$, $(2)$\\

as a Cauchy series in $Frac(L[[x'-a]])$, by definition of $H$ and
$h$. By the hypotheses on $\gamma'$, we can find a parametrisation
of $\gamma'$, given by Lemma 3.14, of the form $(x(y),y)$. Making
the substitution $y=x'$, gives an identical parametrisation
$(x(x'),x')$ in the variable $x'$. Using the uniqueness result for
such parametrisations, see the proof of Lemma 3.11, we have
that;\\

$(y'(x'),x')=(x(x'),x')$, $(3)$\\

as sequences of algebraic power series in $Frac(L[[x'-a]])$. We,
therefore, have, combining $(1),(2)$ and $(3)$, that;\\

${dH\over dx'}={d\over dx'}h(y'(x'),x')={d\over dx'}h(x(x'),x')={dh\over dy}|_{y=x'}$\\

This gives the first identity at the level of Cauchy series in\\
$Frac(L[[x'-a]])$. The second identity, at the branch $\gamma'$,
for Cauchy series, comes from Lemma 3.16. The third identity, for
Cauchy series, follows from the remarks made between Lemmas 3.14
and 3.15. In order to obtain the identities at the level of
rational functions on $C$, it is sufficient to consider the case
${dH\over dx'}={dh\over dy}$, the other identities having already
been considered previously. Let $G$ be the rational function
corresponding to ${dh\over dy}$ in the coordinate system
$(x',y')$. We may consider ${dH\over dx'}-G$ as a rational
function on the curve $C$, in this coordinate system. For a
parametrisation $(x',y'(x'))$ of the branch $\gamma$, we then have that;\\

${dH\over dx'}(x',y'(x'))={dH\over dx'}={dh\over
dy}|_{y=x'}={dh\over dy}(y'(x'),x')=G(x',y'(x'))$\\

Hence, $({dH\over dx'}-G)(x',y(x'))\equiv 0$. This implies, by
arguments already shown in this paper, that ${dH\over dx'}-G$
vanishes identically on $C$, hence ${dH\over dx'}={dh\over dy}$ as
rational functions on $C$, as required.

\end{proof}

We can now show the following;\\

\begin{lemma}
Let $C$ be given in the original coordinate system $(x,y)$, then,
for each branch $\gamma$ of $J'$, $ord_{\gamma}({dh\over dx})=-1$.
\end{lemma}

\begin{proof}
By Lemma 3.17, we have that;\\

${dH\over dx'}=-{f_{y}\over f_{x}}{dh\over dx}$, $(*)$\\

as an identity of rational functions on $C$. We now claim that,
for a branch $\gamma$ of $J'$, $ord_{\gamma}(-{f_{y}\over
f_{x}})=1$ $(**)$. In order to see this, observe that, by $(ii)$
of Lemma 3.5, such a branch is based at a non-singular point $q$
of $C$. Hence, the vector $(f_{x},f_{y})$, evaluated at $q$, is
not identically zero. By hypotheses on $\gamma$, that the tangent
line is parallel to the $y$-axis, we have that $f_{y}(q)=0$, hence
$f_{x}(q)\neq 0$. In particular, this implies that
$0<val_{\gamma}(f_{x})<\infty$ and, therefore, that
$ord_{\gamma}(-f_{x})=0$. It also implies that
$val_{\gamma}(f_{y})=0$. In order to show $(**)$, it is therefore
sufficient to prove that $ord_{\gamma}(f_{y})=1$. By the result of
Lemma 2.1, this is equivalent to showing that;\\

$I_{\gamma}(C,f_{y}=0)=1$, $(\dag)$\\

We achieve this by the following method;\\

Let $q=(a,b)$ and let $l_{\gamma}$ be the tangent line to the
branch $\gamma$. By the assumption that $\gamma$ belongs to $J'$
,$(ii)$ of Lemma 3.5 and the fact that $\gamma$ is based at the
non-singular point $q$, we have;\\

$I_{\gamma}(C,l_{\gamma})=I_{q}(C,l_{\gamma})=I_{q}(f=0,l_{\gamma})=2$\\

As was shown in the paper \cite{depiro5}, $I_{q}$ is symmetric for
plane algebraic curves. Hence, we must have that;\\

$I_{q}(l_{\gamma},C)=I_{q}(l_{\gamma},f=0)=2$, $(\dag\dag)$\\

as well. By the assumption that $l_{\gamma}$ is parallel to the
$y$-axis, we clearly have that the sequence $(a,b+t)$ parametrises
this tangent line. Hence, using Lemma 2.1 and $(\dag\dag)$, we have that;\\

$ord_{t}f(a,b+t)=2$ and $f(a,b+t)=t^{2}u(t)$ for a unit $u(t)\in L[[t]]$\\

Applying the Rule for differentiating algebraic power series,
given in Lemma 3.7, we have that;\\

$f_{y}(a,b+t)=2tu(t)+t^{2}u'(t)$, (\footnote{An alternative argument is required when $char(L)=2$, we leave the details to the reader})\\

Hence;\\

$ord_{t}f_{y}(a,b+t)=I_{q}(l_{\gamma},f_{y}=0)=I_{q}(f_{y}=0,l_{\gamma})=1$\\

This clearly implies that $f_{y}=0$ is non-singular at $q$ and
that the tangent line $l_{(f_{y}=0)}$ is distinct from
$l_{\gamma}$. Hence, the intersection between $C$ and $f_{y}=0$ is
algebraically transverse at $q$, therefore, as was shown in
\cite{depiro5}, we obtain $(\dag)$ and then $(**)$ follows as
well.\\

In order to finish the proof of the lemma, we use the fact from
Lemma 3.13, that, for a branch $\gamma$ of $J'$,
$ord_{\gamma}({dH\over dx'})=0$. Combined with $(*)$ and $(**)$,
it follows easily that $ord_{\gamma}({dh\over dx})=-1$ as
required.

\end{proof}

We can now summarise what we have shown in the following theorem;\\

\begin{theorem}{Theorem on Differentials}

Let hypotheses and notation be as in the remarks immediately before Lemma 3.9, then;\\

$({dh\over dx}=0)=J\cup 2G'$ and $({dh\over dx}=\infty)=J'\cup 2G$\\

In particular, one has that $J+2G'\equiv J'+2G$.\\

\end{theorem}

\begin{proof}
By Lemma 3.9, a branch $\gamma$, distinct from $G\cup G'\cup J'$,
is counted in $({dh\over dx}=0)$ exactly if it appears in
$J=Jac(g_{m}^{1})$, and is, moreover, also counted with the same
multiplicity. By Lemma 3.12, each branch $\gamma$ of $G'$ is
counted twice in $({dh\over dx}=0)$. Again, by Lemma 3.9, each
branch $\gamma$ of $G$ is counted twice in $({dh\over
dx}=\infty)$. Finally, by Lemma 3.18, each branch $\gamma$ of
$J'=Jac(g_{n}^{1})$ is counted once in $({dh\over dx}=\infty)$.
Hence, the first part of the lemma is shown. As ${dh\over dx}$
defines a rational function on $C$, if $s=deg({dh\over dx})$,
then, by Lemma 2.4, we can associate a $g_{s}^{1}$, defined by
$({dh\over dx})$, to which $J+2G'$ and $J'+2G$ belong as weighted
sets. Hence, by the definition of linear equivalence given in
Definition 2.6, we have that $J+2G'\equiv J'+2G$ as required.

\end{proof}

\begin{rmk}
In the above theorem, one may assume that $G$ and $G'$ denote
\emph{any} $2$ weighted sets of the given $g_{n}^{1}$ and
$g_{m}^{1}$ respectively. In order to see this, suppose that
$G_{1}$ and $G_{1}'$ are any two such weighted sets. Then we claim
that;\\

$J'+2G_{1}\equiv J'+2G$ and $J+2G_{1}'\equiv J+2G'$ $(*)$\\

In order to show $(*)$, by Theorem 2.7 we may assume that
$G-G_{1}=div(f_{1})$ and $G'-G_{1}'=div(f_{2})$ for rational
functions $\{f_{1},f_{2}\}$ on $C$. By Lemma 2.9, we then have
that;\\

 $2G-2G_{1}=div(f_{1}^{2})$ and $2G'-2G_{1}'=div(f_{2}^{2})$\\

  Hence, in particular, by Lemma 2.4 and Definition 2.6;\\

  $2G\equiv 2G_{1}$ and $2G'\equiv 2G_{1}'$ $(**)$\\

  The claim $(*)$ then follows by adding the fixed branch
  contributions $J'$ and $J$ to the series defining the
  equivalences in $(**)$ and using Lemma 1.13. From $(*)$, we
  deduce the equivalence;\\

  $J'+2G_{1}\equiv J+2G_{1}'$\\

  immediately from Theorem 3.19 and Theorem 2.10, as required.\\

\end{rmk}

We now improve Theorem 3.19 as follows;\\

\begin{theorem}
Let a $g_{n_{1}}^{1}$ and $g_{n_{2}}^{1}$ be given on $C$, with no
fixed branch branch contribution. Then, if
$J_{1}=Jac(g_{n_{1}}^{1})$, $J_{2}=Jac(g_{n_{2}}^{1})$ and
$\{G_{1},G_{2}\}$ denote any $2$ weighted sets appearing in
$\{g_{n_{1}}^{1},g_{n_{2}}^{1}\}$, then;\\

$J_{1}+2G_{2}\equiv J_{2}+2G_{1}$\\

\end{theorem}

\begin{proof}
By Lemma 3.4, we can find rational functions $\{h_{1},h_{2}\}$,
defining the $\{g_{n_{1}}^{1},g_{n_{2}}^{1}\}$ on $C$, satisfying
the conclusion of the lemma with respect to
$G_{1}^{\infty}=(h_{1}=\infty)$ and
$G_{2}^{\infty}=(h_{2}=\infty)$. An inspection of Lemma 3.5 shows
that we can also obtain the conclusion there with $J_{1}$
replacing $J$ and $G_{1}^{\infty}$ replacing $G$. Now one can
follow through the proof up to Theorem 3.19, in order to obtain;\\

$J_{1}+2G'\equiv J'+2G_{1}^{\infty}$ and $J_{2}+2G'\equiv
J'+2G_{2}^{\infty}$\\

By Remarks 3.20, we can replace $G_{1}^{\infty}$ by $G_{1}$ and
$G_{2}^{\infty}$ by $G_{2}$, in order to obtain;\\

$J_{1}+2G'\equiv J'+2G_{1}$ and $J_{2}+2G'\equiv J'+2G_{2}$\\

We then obtain, by Theorem 2.24;\\

$J_{1}-2G_{1}\equiv J'-2G'\equiv J_{2}-2G_{2}$\\

Hence, by Theorem 2.23;\\

$J_{1}-2G_{1}\equiv J_{2}-2G_{2}$\\

and, therefore, the lemma follows from Theorem 2.24 again.

\end{proof}

We now improve Definition 3.2.\\

\begin{defn}
Let $C$ be a projective algebraic curve and let an arbitrary
$g_{n}^{1}$ be given on $C$. Let $K$ be the fixed branch
contribution of this $g_{n}^{1}$, with total multiplicity $m$ and
let $g_{n-m}^{1}$ be obtained by removing this fixed branch
contribution. Then we define;\\

$Jac(g_{n}^{1})=2K+Jac(g_{n-m}^{1})$\\

\end{defn}

\begin{rmk}
Note that, as a consequence of the definition, if a $g_{n}^{1}$ is
given on $C$, then a branch $\gamma$, which is counted $\beta$
times for the $g_{n}^{1}$ and $\beta+\alpha$ times in a particular
weighted set of the $g_{n}^{1}$, is counted $2\beta+\alpha-1$
times in $Jac(g_{n}^{1})$.

\end{rmk}

We then have an improved version of Theorem 3.21;\\

\begin{theorem}
Let $C$ be a plane projective algebraic curve and let a
$g_{n}^{1}$ be given on $C$. Then, if $G$ is any weighted set in
this $g_{n}^{1}$ and $J=Jac(g_{n}^{1})$, then the series $|J-2G|$
is \emph{independent} of the particular $g_{n}^{1}$.

\end{theorem}
\begin{proof}
Suppose that a $g_{n_{1}}^{1}$ and a $g_{n_{2}}^{1}$ are given on
$C$ and let $\{K_{1},K_{2}\}$ be their fixed branch contributions
of multiplicity $\{m_{1},m_{2}\}$ respectively. Removing these
contributions, we obtain a $g_{n_{1}-m_{1}}^{1}$ and a
$g_{n_{2}-m_{2}}^{1}$ on $C$ with no fixed branch contributions.
Let $\{G_{1},G_{2}\}$ be weighted sets of
$\{g_{n_{1}}^{1},g_{n_{2}}^{1}\}$, then we can find weighted sets $\{G_{1}',G_{2}'\}$
of $\{g_{n_{1}-m_{1}}^{1},g_{n_{2}-m_{2}}^{1}\}$ such that;\\

$G_{1}=K_{1}+G_{1}'$ and $G_{2}=K_{2}+G_{2}'$ $(*)$\\

Let $J_{1}'=Jac(g_{n_{1}-m_{1}}^{1})$ and
$J_{2}'=Jac(g_{n_{2}-m_{2}}^{1})$, then we have, by Theorem
3.21;\\

$J_{1}'+2G_{2}'\equiv J_{2}'+2G_{1}'$\\

Therefore, by Theorem 2.12;\\

$(J_{1}'+2K_{1})+2(K_{2}+G_{2}')\equiv
(J_{2}'+2K_{2})+2(K_{1}+G_{1}')$\\

If $J_{1}=Jac(g_{n_{1}}^{1})$ and $J_{2}=Jac(g_{n_{2}}^{1})$, we
then have, by Definition 3.22 and $(*)$, that;\\

$J_{1}+2G_{2}\equiv J_{2}+2G_{1}$\\

Hence, the result follows by Theorem 2.24 and Definition 2.25.
\end{proof}

We now transfer this result to an arbitrary projective algebraic
curve $C$, using birationality. We first require;

\begin{lemma}
Let $[\Phi]:C_{1}\leftrightsquigarrow C_{2}$ be a birational map.
Let a $g_{n}^{1}$ be given on $C_{2}$, with corresponding
$[\Phi]^{*}(g_{n}^{1})$ on $C_{1}$. Then we have that;\\

$[\Phi]^{*}(Jac(g_{n}^{1}))=Jac([\Phi]^{*}(g_{n}^{1}))$\\

\end{lemma}

\begin{proof}

The result follows trivially from the fact that $[\Phi]^{*}$
defines a bijection on branches, see Lemma 5.7 of \cite{depiro6},
the definition of $Jac(g_{n}^{1})$ given in Definition 3.22 and
the definition of $[\Phi]^{*}(g_{n}^{1})$ given in Theorem 1.14.

\end{proof}

We can now extend Theorem 3.24;\\

\begin{theorem}
Let $C$ be an arbitrary projective algebraic curve and let a
$g_{n}^{1}$ be given on $C$. Then, if $G$ is any weighted set in
this $g_{n}^{1}$ and $J=Jac(g_{n}^{1})$, then the series $|J-2G|$
is \emph{independent} of the particular $g_{n}^{1}$.

\end{theorem}

\begin{proof}
By Theorem 1.33 of \cite{depiro6}, we find a birational
$\Phi:C\leftrightsquigarrow C'$, with $C'$ a plane projective
curve. Suppose that a $g_{n_{1}}^{1}$ and a $g_{n_{2}}^{1}$ are
given on $C$. Let $\{G_{1},G_{2}\}$ be weighted sets of
$\{g_{n_{1}}^{1},g_{n_{2}}^{1}\}$ and let $\{J_{1},J_{2}\}$ be the
Jacobians of $\{g_{n_{1}}^{1},g_{n_{2}}^{1}\}$. Using Lemma 3.25,
we, therefore, obtain that\\
$\{[\Phi^{-1}]^{*}(G_{1}),[\Phi^{-1}]^{*}(G_{2})\}$ are weighted
sets of
$\{[\Phi^{-1}]^{*}(g_{n_{1}}^{1}),[\Phi^{-1}]^{*}(g_{n_{2}}^{1})\}$
and $\{[\Phi^{-1}]^{*}(J_{1}),[\Phi^{-1}]^{*}(J_{2})\}$ are
Jacobians of
$\{[\Phi^{-1}]^{*}(g_{n_{1}}^{1}),[\Phi^{-1}]^{*}(g_{n_{2}}^{1})\}$.
By the result of Theorem 3.24, we obtain;\\

$[\Phi^{-1}]^{*}(J_{1})-2[\Phi^{-1}]^{*}(G_{1})\equiv
[\Phi^{-1}]^{*}(J_{2})-2[\Phi^{-1}]^{*}(G_{2})$\\

Applying $[\Phi]^{*}$ to this equivalence, using Theorem 2.32 and
the fact that this map is linear
on branches, we obtain that;\\

$J_{1}-2G_{1}\equiv J_{2}-2G_{2}$\\

Hence, by Definition 2.30, $|J_{1}-2G_{1}|=|J_{2}-2G_{2}|$ as
required.
\end{proof}

Using this result, we are able to show the following;\\

\begin{theorem}
Let $C$ be a projective algebraic curve and let a $g_{n}^{r}$ be
given on $C$. Then, if $\{g_{n}^{1},g_{n}^{1'}\}$ are any
\emph{two} series with;\\

$g_{n}^{1}\subseteq g_{n}^{r}$ and $g_{n}^{1'}\subseteq
g_{n}^{r}$\\

then;\\

$Jac(g_{n}^{1})\equiv Jac(g_{n}^{1'})$\\
\end{theorem}

\begin{proof}
Let $\{G,G'\}$ be weighted sets of $\{g_{n}^{1},g_{n}^{1'}\}$ and
let $\{J,J'\}$ be their Jacobians. By Theorem 3.26, we have that;\\

$J-2G\equiv J'-2G'$\\

As $\{G,G'\}$ both belong to the same $g_{n}^{r}$, we also have
that $G\equiv G'$. By Theorem 2.29, we then obtain that $J\equiv
J'$ as required.

\end{proof}

As a result of this theorem, we can make the following
definition;\\

\begin{defn}{Jacobian Series}\\

Let $C$ be a projective algebraic curve and let a $g_{n}^{r}$ be
given on $C$, with $r\geq 1$, then we define $Jacob(g_{n}^{r})$ to
be the \emph{complete} series containing the Jacobians
$Jac(g_{n}^{1})$ for \emph{any} $g_{n}^{1}\subseteq g_{n}^{r}$.
\end{defn}

\begin{rmk}
If we choose any two subordinate $g_{n}^{1}\subseteq g_{n}^{r}$,
$g_{n}^{1'}\subseteq g_{n}^{r}$ then, we have, by Theorem 2.18 and
Theorem 3.27, that $|Jac(g_{n}^{1})|=|Jac(g_{n}^{1'})|$. Hence,
there \emph{does} exist a complete linear series with the property
required of the definition.
\end{rmk}

We now make the further;\\

\begin{defn}
Let $C$ be a projective algebraic curve and let $G$ be an
effective weighted set, defining a complete series $|G|$, with
$dim(|G|)\geq 1$. Then, we define;\\

$|G_{j}|=Jacob(|G|)$\\

\end{defn}

We then have;\\

\begin{theorem}{Fundamental Theorem of Jacobian Series}\\

Let $C$ be a projective algebraic curve and let $\{A,B\}$ be
effective weighted sets on $C$, with multiplicity $\{m,n\}$
respectively, defining complete series $|A|$ and $|B|$, with
 $dim(|A|)\geq 1$ and $dim(|B|)\geq 1$. Then;\\

$|(A+B)_{j}|=|A_{j}|+|2B|=|B_{j}|+|2A|$\\

\end{theorem}

\begin{proof}
Consider first the complete series $|A|+|B|=|A+B|$. By the
hypotheses, we can find a $g_{m}^{1}\subseteq |A|$, $(*)$. Adding
$B$ as a fixed branch contribution to this $g_{m}^{1}$, we obtain
a $g_{m+n}^{1}\subseteq |A+B|$. We then have, by Definition 3.22;\\

$Jac(g_{m+n}^{1})=Jac(g_{m}^{1})+2B$\\

It then follows immediately from Remarks 3.29 and Definition 3.30 that;\\

$|(A+B)_{j}|=|Jac(g_{m}^{1})+2B|=|Jac(g_{m}^{1})|+|2B|$\\

Again, by Remarks 3.29, $(*)$ and Definition 3.30, we have that;\\

$|Jac(g_{m}^{1})|=|A_{j}|$\\

Hence,\\

$|(A+B)_{j}|=|A_{j}|+|2B|$\\

as required. The remaining equality is similar.\\

\end{proof}

We now have;\\

\begin{theorem}{Canonical Series}\\

Let $C$ be a projective algebraic curve and let $A$ be an
effective weighted set, such that $dim(|A|)\geq 1$, $(*)$.
Then the series;\\

$|A_{j}|-|2A|$\\

depends only on $C$ and \emph{not} on the choice of $A$ with the
requirement $(*)$.
\end{theorem}

\begin{proof}
Suppose that $\{A,B\}$ are effective weighted sets satisfying
$(*)$. By Theorem 3.31, we have that;\\

$|A_{j}|+|2B|=|B_{j}|+|2A|$\\

Subtracting the series $|2A|+|2B|$ from both sides, we obtain
that;\\

$|A_{j}|-|2A|=|B_{j}|-|2B|$\\

as required.\\
\end{proof}

We now make the following;\\

\begin{defn}{Genus of a Projective Algebraic Curve}\\

Let $C$ be a projective algebraic curve, then we define the
genus $\rho$ of $C$ by the formula;\\

$2\rho-2=r-2n$\\

where $n=order(|A|)$ and $r=order(|A_{j}|)$, for any effective
weighted set with the property that $dim(|A|)\geq 1$.

\end{defn}

\begin{rmk}
This is a good definition by Theorem 3.32.
\end{rmk}

It is clear from the definition that the genus $\rho$ is rational. We also have;\\

\begin{theorem}
The genus $\rho$ is a birational invariant.

\end{theorem}

\begin{proof}
Let $\Phi:C\leftrightsquigarrow C'$ be a birational map. If $A$ is an effective weighted set on $C'$ with $dim(|A|)\geq 1$, then $[\Phi]^{*}A$ is an effective weighted set on $C$ with $dim(|[\Phi]^{*}A|)\geq 1$. We clearly have that $order(|A|)=order(|[\Phi]^{*}A|)$. We also have that $Jacob(|[\Phi^{*}]A|)=[\Phi]^{*}(Jacob(|A|))$, using Lemma 3.25. In particular, we then have that $order(|A_{j}|)=order(|([\Phi^{*}]A)_{j}|)$. The result then follows immediately from Definition 3.33.
\end{proof}

\begin{theorem}
Let $C$ be a plane projective algebraic curve of order $n$, having $d$ nodes as singularities, then;\\

$\rho={(n-1)(n-2)\over 2}-d$\\

In particular, the genus $\rho$ of any projective algebraic curve is an integer.

\end{theorem}

\begin{proof}
Fix a generic point $P$ in the plane, and consider the pencil of lines passing through $P$. This defines a $g_{n}^{1}$ on $C$, where $n=order(C)$. Using Lemma 3.5 and the corresponding coordinate system $(x,y)$, we may assume that $Jac(g_{n}^{1})$ consists exactly of the non-singular branches of $C$ (each counted once) in finite position, whose tangent lines are parallel to the $y$-axis. (Here, the branch formulation of a $g_{n}^{1}$ is critical, the lines in the given pencil passing through the nodes of $C$ make no contribution to $Jac(g_{n}^{1})$ as they pass through each \emph{branch} of the node transversely.) It follows that $r$ in Definition 3.33 is exactly the number of these branches. By Remarks 4.2 below, this is exactly the class of $C$, which gives the following formula relating the class and the genus;\\

$2\rho-2=r-2n$, $(\dag)$\\

Now,  applying the Plucker formula, given in Section 4, we have;\\

$r=n(n-1)-2d$, $(*)$\\

Combining $(\dag)$ and $(*)$, we obtain;\\

$\rho={n(n-1)-2d\over 2}-(n-1)={(n-1)(n-2)\over 2}-d$\\

as required. The remaining part of the theorem then follows immediately from Theorem 3.35 and the fact that any projective algebraic curve is birational to a plane projective curve, having at most nodes as singularities. (Note when $C$ has no nodes, the genus gives the number of alcoves of $n$ lines in general position).

\end{proof}

\begin{theorem}

Let $C$ be a plane projective curve of order $n$, having $d$ nodes as singularities, then;\\

$d\leq {(n-1)(n-2)\over 2}$\\

In particular, $\rho\geq 0$, for any projective algebraic curve.

\end{theorem}

\begin{proof}
By $(*)$ of Theorem 3.36, we have;\\

$d\leq {n(n-1)\over 2}\leq {(n-1)(n+2)\over 2}$, $(**)$\\

As the parameter space $Par_{n-1}$ of plane curves, having dimension $n-1$, has dimension ${(n-1)(n+2)\over 2}$, by $(**)$, we can find a plane curve $C'$ of order $n-1$, passing through the $d$ nodes of $C$ and through ${(n-1)(n+2)\over 2}-d$ further non-singular points of $C$. As $C$ is irreducible, the curve $C'$ cannot contain $C$ as a component, therefore, by Bezout's Theorem, must intersect $C$ in exactly $n(n-1)$ points, counted with multiplicity. Using Lemma 1.11 of \cite{depiro8}, we have the total intersection multiplicity contributed by the nodes is at least $2d$, and, clearly, the total intersection multiplicity contributed by the further non-singular points is at least ${(n-1)(n+2)\over 2}-d$. Therefore,\\

$n(n-1)\geq (2d+{(n-1)(n+2)\over 2}-d)$\\

which gives that;\\

${(n-1)(n-2)\over 2}\geq d$\\

as required. Again, using the fact that any projective algebraic curve is birational to a plane projective curve, having at most nodes as singularities, and combining the first part of this theorem with Theorem 3.36, we obtain that $\rho\geq 0$.

\end{proof}

We now require the following definition;\\

\begin{defn}

Let $C$ be a projective algebraic curve. We say that $C$ is rational if it is birational to a line.

\end{defn}

We now show the following result;\\

\begin{theorem}

Let $C$ be a projective algebraic curve, then $C$ is rational if and only if its genus $\rho=0$.

\end{theorem}

\begin{proof}
By Theorem 3.35, it is sufficient to show that the genus of a line is zero. As a line in the plane has no nodes and is of order $1$, this follows immediately from Theorem 3.36. Conversely, suppose that the genus $\rho$ of $C$ is zero. We can assume that $C$ is a plane projective curve, having just $d$ nodes as singularities. By Theorem 3.36, we then have that;\\

$d={(n-1)(n-2)\over 2}$ $(1)$\\

Now, consider the linear system $\Sigma$ of curves of degree $(n-1)$, passing through the $d$ nodes of $C$. As $Par_{n-1}$ has dimension ${(n-1)(n+2)\over 2}$, using $(1)$, this system has dimension;\\

$h\geq {(n-1)(n+2)\over 2}-{(n-1)(n-2)\over 2}=2(n-1)$ $(2)$\\

In particular, as we can assume that that $n\geq 2$, otherwise the result is proved, we have that $h\geq 2$. Now, choose a further $(2n-3)$ non-singular points on $C$. By $(2)$, there exists a linear system $\Sigma_{1}$ of plane curves of degree $(n-1)$, passing through the $d$ nodes of $C$ and these further non-singular points, and this system has dimension;\\

$h_{1}\geq (2n-2)-(2n-3)=1$ $(3)$\\

Now, as $C$ is irreducible, any form $F_{\lambda}$ in the linear system $\Sigma_{1}$ has finite intersection with $C$, hence, using the Branched Version of Bezout's Theorem, given in Theorem 5.13 of \cite{depiro6}, given a curve $F_{\lambda}$ belonging to $\Sigma_{1}$, $C\sqcap F_{\lambda}$ defines an effective weighted set $W_{\lambda}$ of total multiplicity $n(n-1)$. We, therefore, obtain a $g_{n(n-1)}^{h_{1}}$ on $C$. By construction, the number of base branches for this $g_{n(n-1)}^{h_{1}}$ is at least;\\

$2d+(2n-3)=(n-1)(n-2)+(2n-3)=n(n-1)-1$\\

We consider the following cases;\\

(i). One of these $n(n-1)-1$ base branches is $2$-fold for the $g_{n(n-1)}^{h_{1}}$.\\

(ii). There exists a further base branch for the $g_{n(n-1)}^{h_{1}}$. $(\dag)$\\

Using the Branched Version of Bezout's theorem again, we would then have that the total intersection multiplicity of intersection is contributed by the base branches. By Theorem 1.3, this implies that some form contains all of $C$, contradicting the fact that $C$ is irreducible. Hence, $(\dag)$ cannot occur. It follows that each base branch is $1$-fold, for the $g_{n(n-1)}^{h_{1}}$ and there are exactly $n(n-1)-1$ base branches. By Definition 1.4, we can remove this base branch contribution, in order to obtain a $g_{1}^{h_{1}}$ with no base branches. By Lemma 1.16 and $(3)$, we then have that $h_{1}=1$. The $g_{1}^{1}$ then obtained is simple, by Definition 2.29 of \cite{depiro6}, hence, by Lemma 2.30 of \cite{depiro6}, defines a birational map to $P^{1}$. This proves the theorem.

\end{proof}

\begin{rmk}
The characterisation of rational curves in terms of the vanishing of their genus may lead to applications in the field of Zariski structures, see the remarks immediately after Theorem 1.3 of \cite{depiro6}. As the technique we have used to define the genus of an algebraic curve is geometric rather than algebraic, one would hope to extend the definition to Zariski curves, and therefore, to find examples of non-algebraic Zariski curves which exhibit properties of lines.

\end{rmk}

\begin{section}{A Plucker Formula for Plane Algebraic Curves}

The purpose of this section is to give a geometric proof of an elementary
Plucker formula for plane projective algebraic curves. The proof
that we give depends on the method of the Italian school, however the underlying geometrical idea is due to Plucker. One can find other modern \emph{algebraic} proofs, see, for example \cite{GH}. In order to state the theorem, we first require the following definition;\\

\begin{defn}{Class of a Plane Algebraic Curve}\\

Let $C\subset P^{2}$ be a plane projective algebraic curve. Then
we define the class of $C$ to be the number of its tangent lines
passing through a generic point $Q\in P^{2}$.

\end{defn}

\begin{rmk}
In order to see that this is a good definition, let $F(X,Y,Z)$ be
a defining equation for $C$. For a non-singular point $p$ of $C$,
the equation of the tangent line $l_{p}$ is given by;\\

$F_{X}(p)U+F_{Y}(p)V+F_{Z}(p)W=0$\\

For the finitely many singular points $\{p_{1},\ldots
p_{j},\ldots,p_{r}\}$, which are the origins of branches
$\{\gamma_{1}^{1},\ldots,\gamma_{1}^{t(1)},\ldots,\gamma_{r}^{1},\ldots,\gamma_{r}^{t(r)}\}$,
we obtain finitely many tangent lines
$\{l_{\gamma_{1}^{1}},\ldots,l_{\gamma_{1}^{t(1)}},\ldots,l_{\gamma_{r}^{1}},\ldots,l_{\gamma_{r}^{t(r)}}\}$.
It is easily checked that the union $V$ of these finitely many
tangent lines is defined over the field of definition of $C$.
Hence, we may assume that $Q$ does not lie on any of these
tangents. It also follows from duality arguments, see Section 5,
that there exist finitely many bitangent lines
$\{l_{1},\ldots,l_{s}\}$ (Check this). Again, it is easily checked
that the union $W$ of these finitely many tangent lines is defined
over the field of definition of $C$. Hence, we can assume that $Q$
does not lie on any of these bitangents. Now, the condition
$Cl_{\geq m}(Q)$ that there exist at least $m$ tangent lines,
centred at non-singular points of $C$,
passing through $Q$, is given by;\\

$\exists_{x_{1}\neq \ldots\neq x_{m}}[\bigwedge_{1\leq j\leq
m}(NS(x_{j})\wedge
F_{X}(x_{j})Q_{0}+F_{Y}(x_{j})Q_{1}+F_{Z}(x_{j})Q_{2}=0)]$\\

The condition $Cl_{=m}(Q)$ that there exist exactly $m$ tangent
lines, centred at non-singular points of $C$, passing
through $Q$, is given by;\\

$Cl_{\geq m}(Q)\wedge\neg Cl_{\geq m+1}(Q)$\\

It is clear that each predicate $Cl_{m}$ is defined over the field
of definition of $C$, hence, if it holds for \emph{some} generic
$Q\in P^{2}$, it holds for \emph{any} generic $Q\in P^{2}$.
Finally, observe that for \emph{generic} $Q$, there can only exist
finitely many tangent lines passing through $Q$. (We leave the
reader to check this result.) Hence, the class does define a
non-negative integer $m\geq 0$. Observe that it is possible for
$m=0$, for example if $C$ is a line. Also, observe that it is
possible for $\neg Cl_{m}(P)$ $(m\geq 0)$ to hold, if $P$ is
\emph{not} generic, for example if $C$ is strange. Excluding
exceptional cases, one can show, using duality arguments, that a
generic point of a projective algebraic curve is an ordinary
simple point. In this case, it follows that the union $W$ of the
tangent lines to non-ordinary simple points is defined over the
field of definition of $C$, hence that the class is witnessed by
ordinary simple points.

\end{rmk}

We now give a statement of the elementary Plucker formula;\\

\begin{theorem}
Let $C$ be a plane projective algebraic curve of order $n$ and
class $m$, with $d$ nodes. Then;\\

$n+m+2d=n^{2}$\\

\end{theorem}

\begin{rmk}
We refer the reader to the papers \cite{depiro6} and \cite{depiro8} for some relevant
terminology. The reader should observe that we are using the
\emph{weaker} definition of a node as the origin of exactly $2$
non-singular branches with distinct tangent directions. The
original formula is usually stated with cusps, however we will not
require this stronger version.
\end{rmk}

\begin{proof}{Theorem 4.3}\\

Let the class of $m$ of $C$ be witnessed by a generic point $Q$ of
$P^{2}$. Let $\{\tau_{1},\ldots,\tau_{m}\}$ be the finitely many
tangents of $C$ passing through $Q$. We can assume that these tangent lines are all based at non-singular points of $C$ and do not coincide with the tangent directions of any of the finitely many nodes of $C$. Choose a line $l$, passing
though $Q$, which does not coincide with any of these tangent
lines and does not intersect $C$ in any of the finitely many
nodes. Let $\{X,Y,Z\}$ be a choice of coordinates such that the
line $l$ corresponds to $Z=0$, hence defines the line at infinity
in the affine coordinate system $\{x={X\over Z},y={Y\over Z}\}$.
Let $F(x,y)=0$ define $C$ in this coordinate system. By
construction, we have that the tangent lines
$\{\tau_{1},\ldots,\tau_{m}\}$ are all parallel in the coordinate
system $\{x,y\}$. Let $(\alpha,\beta)$ be the gradient vector of
each of these lines. We define the translation $C_{t}$ of $C$
as follows;\\

$C_{t}=\{(x,y):(x-t\alpha,y-t\beta)\in
C\}=\{(x,y):F(x-t\alpha,y-t\beta)=0\}$\\

Now observe that we can find polynomials $\lambda_{(i,j)}(t)$, for
$0\leq (i+j)\leq n$, such that;\\

$F(x-t\alpha,y-t\beta)=F(x,y;t)=\sum_{0\leq (i+j)\leq
n}\lambda_{(i,j)}(t)x^{i}y^{j}$ $(*)$\\

Let $\{T_{0},T_{1}\}$ be coordinates for $P^{1}$ such that
$t={T_{0}\over T_{1}}$. Making the substitutions $\{t={T_{0}\over
T_{1}},x={X\over Z},y={Y\over Z}\}$ in $(*)$, and projectivising
the resulting equation, by multiplying through by a suitable
denominator, we obtain an \emph{algebraic} (not necessarily
linear) family of curves, parametrised by $P^{1}$. Let
$Par_{t}\subset P^{{n(n+3)\over 2}}=Par_{n}$ be the image of the map;\\

$\Phi:P^{1}\rightarrow P^{{n(n+3)\over 2}}$, ${T_{0}\over
T_{1}}\mapsto
(\lambda_{(0,0)}({T_{0}\over T_{1}}):\ldots:\lambda_{(i,j)}({T_{0}\over T_{1}}):\ldots)$\\

Then $Par_{t}$ defines a projective algebraic curve, parametrised
by $P^{1}$, whose points determine each curve in the algebraic
family $\{C_{t}\}_{t\in P^{1}}$. For a given $t_{0}\in Par_{t}$,
we can define a tangent line $l_{t_{0}}$ to $Par_{t}$ by;\\

$(\lambda_{(0,0)}(t_{0})+\mu\lambda'_{(0,0)}(t_{0}):\ldots:\lambda_{(i,j)}(t_{0})+\mu\lambda'_{(i,j)}(t_{0}),\ldots)$\\

\begin{rmk}

Observe that, in the case when $t_{0}$ is a smooth point of
$Par_{t}$, $l_{t_{0}}$ defines \emph{the} tangent line of
$Par_{t}$ at $t_{0}$ in the sense of Section 1 of \cite{depiro6},
$(*)$. This follows easily from the fact that, for any homogeneous
polynomial $G(X_{(0,0)},\ldots,X_{(i,j)},\ldots)$ vanishing on
$Par_{t}$, we have from;\\

$G(\lambda_{(0,0)}(t),\ldots,\lambda_{(i,j)}(t),\ldots)=0$\\

that;\\

${\partial G\over \partial
X_{(0,0)}}|_{\lambda_{(0,0)}(t)}\lambda_{(0,0)}'(t)+\ldots+{\partial
G\over
\partial
X_{(i,j)}}|_{\lambda_{(i,j)}(t)}\lambda_{(i,j)}'(t)+\ldots=0$ $(1)$\\

and from;\\

$G(s\lambda_{(0,0)}(t_{0}),\ldots,s\lambda_{(i,j)}(t_{0}),\ldots)=0$\\

that;\\

${\partial G\over \partial
X_{(0,0)}}|_{\lambda_{(0,0)}(t_{0})}\lambda_{(0,0)}(t_{0})+\ldots+{\partial
G\over
\partial
X_{(i,j)}}|_{\lambda_{(i,j)}(t_{0})}\lambda_{(i,j)}(t_{0})+\ldots=0$ $(2)$\\

In order to extend these considerations to singular points, let
$U=\Phi^{-1}(NonSing(Par_{t}))$ and consider the covers $V\subset
U\times P^{{n(n+3)\over 2}}$ and $V^{*}\subset U\times P^{{n(n+3)\over 2}}$ defined by;\\

$V=\{(t,y):y\in l_{\Phi(t)}\}$\\

$V^{*}=\{(t,\lambda):H_{\lambda}\supset l_{\Phi(t)}\}$\\

Let $\bar V$ and $\bar V^{*}$ define the Zariski closure of these
two covers inside $P^{1}\times P^{{(n+1)(n+2)\over 2}}$. Let
$Par_{t}^{ns}$ be a nonsingular model of $Par_{t}$ with birational
map $\Phi':Par_{t}^{ns}\leftrightsquigarrow Par_{t}$ and let
$\Phi''=\Phi'^{-1}\circ\Phi$. Using the method of Lemma 6.13 in
\cite{depiro6}, one can show that, for $t_{0}\in P^{1}\setminus
U$, the fibre $\bar V^{*}(t_{0})$ consists of parameters for forms
$H_{\lambda}$ containing the tangent line $l_{\gamma^{j}_{p}}$,
where $\Phi(t_{0})=p$ and the branch $\gamma_{p}^{j}$ corresponds
to $p_{j}\in Par_{t}^{ns}$, with $\Phi''(t_{0})=p_{j}$. Using the
duality argument in Lemma 5.3 below, one then deduces that the
fibre $\bar V(t_{0})$ defines the tangent line
$l_{\gamma^{j}_{p}}$, $(**)$.\\

It then follows, by combining the arguments $(*)$ and $(**)$,
that, even for a singular point $t_{0}$ of $Par_{t}$, $l_{t_{0}}$
defines a tangent line to the branch corresponding to
$\Phi''(t_{0})$.\\

Following the Italian terminology, we refer to the \emph{pencil}
of curves defined by $l_{t_{0}}$, for given $t_{0}\in P^{1}$, as
the curves \emph{infinitamente vicine} to $C_{t_{0}}$, and we
refer to \emph{any} curve in the pencil, distinct from
$C_{t_{0}}$, by;\\

$F(x,y;t_{0}+dt_{0})$\\

This notation is motivated by the following fact;\\

Let $\Phi:P^{1}\rightarrow Par_{t}$ be the parametrisation of
$Par_{t}$, then, for \emph{any} polynomial representation of
$\Phi$ of the form;\\

$\Phi:t\mapsto (\lambda_{0}(t):\lambda_{1}(t):\ldots)$\\

$l_{t_{0}}$ is generated by
$(\lambda_{0}(t_{0}):\lambda_{1}(t_{0}):\ldots)$ and
$(\lambda'_{0}(t_{0}):\lambda'_{1}(t_{0}):\ldots)$\\

The proof of this fact follows easily from the above remarks or
from the observation that, if $h(t)$ is a polynomial with;\\

$(\lambda_{0}(t):\lambda_{1}(t):\ldots)=h(t)(\mu_{0}(t):\mu_{1}(t):\ldots)$\\

then $(\lambda_{0}'(t_{0}):\lambda_{1}'(t_{0}):\ldots)$
corresponds to;\\

$h(t_{0})(\mu_{0}'(t_{0}):\mu_{1}'(t_{0}):\ldots)+h'(t_{0})(\mu_{0}(t_{0}):\mu_{1}(t_{0}):\ldots)$\\

and hence belongs to the tangent line generated by;\\

$\{(\mu_{0}(t_{0}):\mu_{1}(t_{0}):\ldots),(\mu_{0}'(t_{0}):\mu_{1}'(t_{0}):\ldots)\}$\\

\end{rmk}

We now continue the proof of Theorem 4.3;\\

Claim 1: Let $t_{0}=0$, then the algebraic curve $C_{dt_{0}}$,
infinitamente vicine to $C_{t_{0}}=C$, passes through the finitely
many points $\{p_{1},\ldots,p_{m}\}$ which witness the class of
$C$ and, moreover;\\

$I_{p_{j}}(C_{t_{0}},C_{dt_{0}})=1$, for $1\leq j\leq m$\\

In order to prove Claim 1, let $p_{j}=(a_{j},b_{j})$ be such a
point, then we obtain a parametrisation of the tangent line
$l_{p_{j}}$ of the form;\\

$(x(s),y(s))=(a_{j}+\alpha s,b_{j}+\beta s)$\\

We may assume, see Remarks 4.2, that $p_{j}$ is an ordinary simple
point, hence that
$I_{p_{j}}(C_{t_{0}},l_{p_{j}})=2$. This implies that;\\

$F(x(s),y(s);t)=F(x(s)-\alpha t,y(s)-\beta t)=(s-t)^{2}u(s-t),$\\
\indent \ \ \ \ \ \ \ \ \ \ \ \ \ \ \ \ \ \ \ \ \ \ \ \ \ \ \ \ \ \ \ \ \ \ \ \ \ \ \ \ \ \ \ \ \ \ \ \ \ \ \ \ \ \ \ \ \  $u(s-t)$ a unit in $L[[s-t]]$ $(*)$\\

Now, differentiating both sides of the equation $(*)$ with respect
to $t$, we obtain that;\\

${\partial F\over\partial
t}|_{(x(s),y(s),t)}=-2(s-t)u(s-t)-(s-t)^{2}u'(s-t)\\
\indent \ \ \ \ \ \ \ \ \ \ \ \ \ \ \ \
=-(s-t)[2u(s-t)+(s-t)u'(s-t)]$\\

Setting $t=0$, we obtain that;\\

$F_{dt_{0}}(x(s),y(s))=-s[2u(s)+su'(s)]$\\

Hence, $I_{p_{j}}(C_{dt_{0}},l_{p_{j}})=1$. This implies, by
arguments given in \cite{depiro5}, that
$I_{p_{j}}(C_{t_{0}},C_{dt_{0}})=1$ as well.\\

Claim 2: Let $p$ be a non-singular point of $C$, in finite
position, then $C_{dt_{0}}$ passes through $p$ iff $p$ is one of
the finitely many points $\{p_{1},\ldots,p_{m}\}$ witnessing the
class of $C$.\\

One direction of the claim follows immediately from Claim 1.
Conversely, suppose that $p=(p_{1},p_{2})$ lies in finite position
and $C_{dt_{0}}(p_{1},p_{2})$, $(\dag)$. We have a parametrisation
of the line $L_{p}$ of the form;\\

$(x(t),y(t))=(p_{1}+\alpha t,p_{2}+\beta t)$\\

where $L_{p}$ denotes any of the tangent lines witnessing the
class of $C$, translated to $p$. We then have that;\\

$F(x(t),y(t);t)=F(p_{1},p_{2})=0$ $(**)$\\

Differentiating $(**)$ with respect to $t$, we obtain that;\\

$F_{x}|_{(x(t),y(t);t)}x'(t)+F_{y}|_{(x(t),y(t);t)}y'(t)+F_{t}|_{(x(t),y(t);t)}=0$\\

Setting $t=0$, we obtain that;\\

$F_{x}|_{p}\alpha+F_{y}|_{p}\beta+F_{dt_{0}}|_{p}=0$, $(***)$\\

By $(\dag)$, $(**)$ and the fact that $p$ is non-singular, we
obtain that $L_{p}$ is the tangent line to $C$ at $p$. Hence, $p$
must witness the class of $C$.\\

Claim 3. Let $\{q_{1},\ldots,q_{n}\}$ be the finitely many points
lying at infinity. Then $C_{dt_{0}}$ passes through $q_{j}$, for
$1\leq j\leq n$, and, moreover;\\

$I_{q_{j}}(C_{t_{0}},C_{dt_{0}})=1$\\

The points at infinity are given by the intersections of $C$ with
the line $l$ passing through $Q$. By the choice of $l$, given at
the beginning of the proof, these intersections all define
non-singular points of $C$. If $q\in (C\cap l)$ and
$I_{q}(C,l)\geq 2$, then $l$ would define the tangent line of $C$
at $q$. This contradicts the fact that $l$ was chosen to avoid the
finitely many tangent lines $\{\tau_{1},\ldots,\tau_{m}\}$
witnessing the class of $C$. Hence, $I_{q}(C,l)=1$, $(\dag)$, and
the fact that there exist $n$ distinct points at infinity then
follows by an application of Bezout's theorem, using the
assumption that $deg(C)=n$. We now claim that, if $q\in (C\cap
l)$, then $q$ belongs to $\{C_{t}\}_{t\in P^{1}}$, $(*)$. In order
to see this, first observe that the line $l$ at infinity (defined
by $Z=0$ in the choice of coordinates $\{X,Y,Z\}$) is \emph{fixed}
by the
translation $\theta_{t}$ along the tangent lines witnessing the class of $C$;\\

$\theta_{t}:(x,y)\mapsto (x-t\alpha,y-t\beta):(X:Y:Z)\mapsto
(X-t\alpha Z:Y-t\alpha Z:Z)$\\

The claim $(*)$ then follows from the fact that $C_{t}$ is defined
as $Zero(F\circ \theta_{t})$ and the defining equation $F$ of $C$
vanishes on $q$. We further claim that $q$ belongs to
$C_{dt_{0}}$, for $t_{0}=0$, $(**)$. In order to see this, choose
a system of coordinates $\{x',y'\}$ such that $q=(q_{1},q_{2})$ is
in finite position. It is a simple algebraic calculation to show
that we can find a polynomial $G(x',y';\bar z)$ such that the
family of curves $\{C_{t}\}_{t\in P^{1}}$ is represented by
$G(x',y';\Phi^{new}(t))$, for a choice of morphism
$\Phi^{new}:P^{1}\rightarrow P^{{n(n+3)\over 2}}$. As before,
let $Par_{t}^{new}$ be the image of $P^{1}$ under the morphism
$\Phi^{new}$. It is a straightforward algebraic calculation to
show that, if $\theta:P^{2}\rightarrow P^{2}$ is chosen to be a
\emph{homographic} change of variables, then the corresponding
induced morphism $\Theta:P^{{n(n+3)\over 2}}\rightarrow
P^{{n(n+3)\over 2}}$, on the parameter space for projective
algebraic curves of degree $n$, is also a homography. By
construction, we have that $\Theta\circ\Phi=\Phi^{new}$, $(***)$,
where $\Phi$ denoted the old parametrisation of $Par_{t}$. As
$\Theta$ is a homography, using the identity $(***)$ and the chain
rule, we obtain that, for corresponding points
$\{\Phi(t),\Phi^{new}(t)\}$ of $\{Par_{t},Par_{t}^{new}\}$, the
tangent line $l_{\Phi(t)}$ of $Par_{t}$ is mapped by $\Theta$ to
the tangent line $l_{\Phi^{new}(t)}$ of $Par_{t}^{new}$. It
therefore follows that, for given $t_{0}\in P^{1}$, the curves
infinitamente vicine to $C_{t_{0}}$ (see Remarks 4.5), can be
computed by differentiating with respect to the parameter $t$ in
\emph{either} of the coordinate systems $\{\{x,y\},\{x',y'\}\}$.
By $(*)$, we have that;\\

$G(q_{1},q_{2};t)=0$, for $t\in P^{1}$\\

Therefore, differentiating with respect to $t$, we have that;\\

${\partial G\over \partial t}(q_{1},q_{2};t)=0$\\

In particular, setting $t=t_{0}=0$, we obtain the claim $(**)$. We
now consider the pencil of curves defined by
$\{C_{t_{0}},C_{dt_{0}}\}$, which clearly has finite intersection
with $l$, hence defines a $g_{n}^{1}$. By $(**)$, we have that the
set of intersections $(C\cap l)$ are base points (branches) for
this $g_{n}^{1}$. By $(\dag)$ and results of \cite{depiro6}, we
have that the base branch contribution of any of these
intersections is $1$. In particular, it follows that, for a
\emph{generic} choice of $C_{dt_{0}}$, that
$I_{q_{j}}(l,C_{dt_{0}})=1$, for $1\leq j\leq n$. Now applying the
usual argument on tangent lines, we obtain that
$I_{q}(C_{t_{0}},C_{dt_{0}})=1$, for $1\leq j\leq n$ as well. The
result of Claim 3 then follows.\\

Claim 4. Let $\{p_{1},\ldots,p_{d}\}$ be the $d$ nodes of $C$.
Then $C_{dt_{0}}$ passes through $p_{j}$, for $1\leq j\leq d$,
and, moreover;\\

$I_{p_{j}}(C_{t_{0}},C_{dt_{0}})=2$\\

Let $p_{j}=(c_{j},d_{j})$  be such a point, and let;\\

$(x(t),y(t))=(c_{j}+\alpha t,d_{j}+\beta t)$\\

be a parametrisation of any of the tangent lines $\{\tau_{1},\ldots,\tau_{m}\}$, translated to $p_{j}$. We have that;\\

$F(x(t),y(t);t)=F(c_{j},d_{j})=0$ $(*)$\\

Moreover, as $(x(t),y(t))$ defines a node of the translated curve $C_{t}$, we have that $F_{x}|_{(x(t),y(t);t)}=0$ and $F_{y}|_{(x(t),y(t);t)}=0$. Hence, differentiating $(*)$ with respect to $t$, we obtain that;\\

$F_{t}|_{(x(t),y(t);t)}=0$\\

Setting $t=0$, we obtain that $F_{dt_{0}}|_{p_{j}}=0$, that is $C_{dt_{0}}$ passes through $p_{j}$. Hence, the first part of the claim is shown. The second part of the claim depends heavily on a geometric argument. Let $l_{t_{0}}$ be the tangent line to $Par_{t}$, at $t_{0}$, as given immediately before Remarks 4.5. As explained in Remarks 4.5, the curve $C_{dt_{0}}$ corresponds to a point $Q$ on $l_{t_{0}}$, and we denote by $O$, the point corresponding to $C_{t_{0}}$ in $Par_{t}$, hence $l_{t_{0}}=l_{OQ}$. By Remarks 4.5, the line $l_{t_{0}}$ defines the tangent line to the branch $\gamma_{0}^{t_{0}}$, corresponding to $t_{0}$, in the parametrisation $\Phi:P^{1}\rightarrow Par_{t}$.\\

Similarly to the argument in Remarks 4.5, let;\\

 $U=(\Phi^{-1}(NonSing(Par_{t}))\setminus\{t_{0}\})\subset P^{1}$\\

and consider the covers $V\subset U\times P^{{(n+1)(n+2)\over 2}}$ and $V^{*}\subset U\times P^{{(n+1)(n+2)\over 2}}$ defined by;\\

$V=\{(t,y):y\in l_{O\Phi(t)}\}$\\

$V^{*}=\{(t,\lambda):H_{\lambda}\supset l_{O\Phi(t)}\}$\\

Let $\bar V$ and $\bar V^{*}$ define the Zariski closure of these
two covers inside $P^{1}\times P^{{n(n+3)\over 2}}$. Using the method of Lemma 6.14 in \cite{depiro6}, one can show that the fibre $\bar V^{*}(t_{0})$ consists of parameters for forms $H_{\lambda}$ containing the tangent line $l_{t_{0}}$. Using the duality argument in Lemma 5.4 below, one then deduces that the fibre $\bar V(t_{0})$ defines the tangent line $l_{t_{0}}$. The details are left to the reader. We now have that $\bar V(t_{0},Q)$ holds, hence, as $t_{0}$ is regular for the cover $(\bar V/P^{1})$, given generic $t_{0}'\in{\mathcal V}_{t_{0}}\cap P^{1}$, we can find $Q'\in{\mathcal V}_{Q}\cap P^{n(n+3)\over 2}$, such that $\bar V(t_{0}',Q')$, hence, $Q'$ belongs to $l_{O\Phi(t_{0}')}$. Denote by $D_{Q'}$ the plane curve of order $n$ corresponding to the point $Q'$ in $Par_{n}$ We now consider the pencil of curves $\mathcal{P}$ defined by the line $l_{O\Phi(t_{0}')}\subset Par_{n}$. By construction, this pencil contains the plane curves $\{C_{t_{0}},C_{t_{0}'},D_{Q'}\}$. As $t_{0}'\in Par_{t}$, the plane curve $C_{t_{0}'}$ is an infinitesimal translation of $C_{t_{0}}$ in the direction defined by the tangent lines $\{\tau_{1},\ldots,\tau_{m}\}$. By the assumption at the beginning of the proof, for a given node $p_{j}$, this direction does not coincide with either of its two distinct tangent lines $\{l_{\gamma_{p_{j}}^{1}},l_{\gamma_{p_{j}}^{2}}\}$. We are, therefore, able to apply Theorem 1.13 of \cite{depiro8}, to conclude that $C_{t_{0}}\cap C_{t_{0}'}\cap{\mathcal V}_{p_{j}}$ consists of exactly two distinct points $\{p_{j}^{1},p_{j}^{2}\}$, situated on the branches $\{\gamma_{p_{j}}^{1},\gamma_{p_{j}}^{2}\}$ of $C_{t_{0}}$, and moreover these intersections are transverse. By elementary facts on linear systems, $\{p_{j}^{1},p_{j}^{2}\}$ are base points for the pencil defined by ${\mathcal P}$, and, moreover, $C_{t_{0}}\cap C_{t_{0}'}=C_{t_{0}}\cap D_{Q'}$. It follows immediately that $C_{t_{0}}\cap D_{Q'}\cap{\mathcal V}_{p_{j}}$ also consists of exactly the two points $\{p_{j}^{1},p_{j}^{2}\}$. Again, by elementary facts on linear systems, see \cite{depiro6} for more details, these intersections are also transverse. We have, therefore, shown, using the non-standard definition of intersection multiplicity, given in \cite{depiro5} or \cite{depiro6}, that $I_{p_{j}}(C_{t_{0}},C_{dt_{0}})=2$. Hence, Claim 4 is shown.\\

We now complete the proof of Theorem 4.3. Combining Claims 1,2,3 and 4, the total multiplicity of intersection between $C$ and $C_{dt_{0}}$ is $n+m+2d$. By construction, $deg(C)deg(C_{dt_{0}})=n.n=n^{2}$. Hence, Bezout's Theorem gives that $n+m+2d=n^{2}$, as required. By the results of \cite{depiro5}, Plucker's formula also holds for the algebraic definition of intersection multiplicity.\\

\end{proof}

\begin{rmk}

The proof that we have given follows Plucker's original geometrical idea and the presentation of Severi in \cite{Sev}. Although long, the methods used are almost entirely geometrical, and adapt easily to handle cases where the singularities of $C$ are more complicated. The reader is invited to extend the approach to these situations. One can find algebraic proofs in the literature, for example in \cite{GH}. Unfortunately, these proofs generally fail to give a precise calculation of intersection multiplicity for $C$ and some other curve $C'$, a problem that we observed in Remarks 1.12 of \cite{depiro8}.

\end{rmk}

\end{section}

\begin{section}{The Transformation of Branches by Duality}

The purpose of this section is to give a general account of the
theory of duality and to develop the connection with the theory of
branches given in \cite{depiro6}. We first give a brief account of
Grassmannians on $P^{w}$. We define;\\

$G_{w,j}=\{P_{j}\subset P^{w}\}$, $(0\leq j\leq w)$\\

where $P_{j}$ is a plane of dimension $j$. It follows from
classical well known arguments that $G_{w,j}$ may be given the
structure of a smooth algebraic variety, see for example \cite{GH}
p193. (These arguments begin with the observation that $G_{w,j}$
can be identified with the set of $j+1$-dimensional subspaces of a
$w+1$-dimensional vector space over $L$.) The fact that $G_{w,j}$
defines a projective algebraic variety follows from the Plucker
embedding;\\

$\rho:G_{w,j}\rightarrow P(\wedge^{j+1}L^{w+1})=P^{C^{w+1}_{j+1}-1}$\\

given by sending a $j+1$ dimensional subspace of $L^{w+1}$, with
basis $\{v_{1},\ldots,v_{j+1}\}$, to the multivector
$v_{1}\wedge\ldots\wedge v_{j+1}$.  It is easily checked that the
Plucker map $\rho$ defines a morphism of algebraic varieties, is
injective on points and its differential $d\rho_{x}$ has maximal
rank, for $x\in G_{w,j}$. In the case when the underlying field
$L=\mathcal{C}$, one can then use the Immersion Theorem and Chow's
Theorem to show that $Image(\rho)$ has the structure of a
projective algebraic variety. In general, one can show;\\

The image of the Grassmannian $G_{w,j}$ under the Plucker
embedding $\rho$ is defined by a linear system of quadrics.\\

For want of a convenience reference, we leave the reader to check
that the proof given of this result in \cite{GH} holds for
arbitrary characteristic. We now observe the following duality
between $G_{w,j}$ and $G_{w,w-(j+1)}$, for $0\leq j\leq w-1$;

\begin{lemma}

If $P_{j}\subset P^{w}$ is a $j$-dimensional plane, then;\\

$\{\lambda\in P^{w*}:H_{\lambda}\supset P_{j}\}$ $(1)$\\

determines a $w-(j+1)$-dimensional plane $P^{*}_{w-(j+1)}$ in
$P^{w*}$. Conversely, if $P^{*}_{w-(j+1)}\subset P^{w*}$ is a
$w-(j+1)$-dimensional plane, then;\\

$\{x\in P^{w}:x\in\bigcap_{\lambda\in P^{*}_{w-(j+1)}}H_{\lambda}\}$ $(2)$\\

determines a $j$-dimensional plane $P_{j}$ in $P^{w}$. Moreover,
these correspondences are inverse and determine a closed
projective variety;\\

$I\subset G_{w,j}\times G_{w,w-(j+1)}$\\

\end{lemma}

\begin{proof}
The proof is quite elementary. If $P_{j}\subset P^{w}$ is a
$j$-dimensional plane, then one can find an independent sequence
$\{\bar a_{0},\ldots,\bar a_{j}\}$ defining it,$(*)$, see Section
1 of \cite{depiro6}. The condition that a hyperplane
$H_{\lambda}$ contains $P_{j}$ is then given by the conditions;\\

$a_{0i}\lambda_{0}+a_{1i}\lambda_{1}+\ldots+a_{wi}\lambda_{w}=0$,
$(0\leq i\leq j)$\\

It is elementary linear algebra, using $(*)$, to see that these
conditions determine a plane of codimension $(j+1)$ in $P^{w*}$.
For the converse direction, if $P^{*}_{w-(j+1)}\subset P^{w*}$ is
a $w-(j+1)$-dimensional plane, then one can find an independent
sequence $\{\bar b_{0},\ldots,\bar b_{w-(j+1)}\}$ defining it,
$(**)$. The condition that $x\in P^{w}$ is contained in the
intersection of the planes defined by $P^{*}_{w-(j+1)}$ is then
given by the conditions;\\

$x_{0}b_{0i'}+x_{1}b_{1i'}+\ldots+x_{w}b_{wi'}=0$, $(0\leq i'\leq
w-(j+1))$\\

By the same elementary linear algebra argument, using $(**)$,
these conditions determine a plane of codimension $(w-j)$ in
$P^{w}$, that is a plane of dimension $j$. The fact that these
correspondences are inverse follows immediately from the
relations;\\

$a_{0i}b_{0i'}+a_{1i}b_{1i'}+\ldots+a_{wi}b_{wi'}=0$, $(0\leq
i\leq j,0\leq i'\leq w-(j+1))$ $(\dag)$\\

and an elementary dimension argument. The final part of the lemma
follows by checking that the relations $(\dag)$ can be defined by
matrix multiplication on representatives of $G_{w,j}$ and
$G_{w,w-(j+1)}$. If $\{U_{I}\}$ and $\{U_{J}\}$ define the
standard open affine covers of these Grassmannians, as given on
p193 of \cite{GH}, then these relations clearly define closed
algebraic subvarieties $I_{I,J}\subset U_{I}\times U_{J}$. Using
standard patching arguments, we then obtain a closed algebraic
subvariety $I\subset G_{w,j}\times G_{w,w-(j+1)}$ determining the
duality correspondence.\\
\end{proof}

\begin{rmk}

The duality correspondence $I$ clearly induces bijective
maps, in the sense of model theory;\\

$*:G_{w,j}\rightarrow G_{w,w-(j+1)}$\\

$*^{-1}:G_{w,w-(j+1)}\rightarrow G_{w,j}$\\

When the underlying field $L$ has characteristic $0$, it follows that they define isomorphisms in the sense of algebraic geometry. However, in non-zero characteristic,  it is difficult to determine whether these maps define morphisms or
are seperable, therefore, whether the maps are inverse in the sense of algebraic geometry. As we will be concerned with the behaviour of algebraic curves under
duality, we will be able to deal with this problem using arguments
on Frobenius that we have already seen in \cite{depiro6}. The
reader should also note that $*$ is \emph{not}
canonical, even as a set theoretic map, for it depends on a
particular identification of $P^{w}$ and $P^{w*}$. We will,
henceforth, denote the (set theoretic) inverse $*^{-1}$ by $*$.
This is motivated by the fact that we can identify $P^{w}$
naturally with $P^{w**}$, in which case the relation $(2)$,
defining $*^{-1}$ in the previous lemma, becomes an instance of
the relation $(1)$.

\end{rmk}

As an application of the above, we have;\\

\begin{lemma}{Tangent Variety of a Projective Algebraic Curve}\\

Let $C\subset P^{w}$ be \emph{any} projective algebraic curve and let;\\

$V\subset NonSing(C)\times P^{w}$ be $\{(x,y):x\in
NonSing(C)\wedge y\in l_{x}\}$\\

Then $V$ defines an irreducible algebraic variety and, if
${\overline V}\subset C\times P^{w}$ defines its Zariski closure,
then, for a singular point $p$, which is the origin of branches
$\{\gamma_{p}^{1},\ldots,\gamma_{p}^{m}\}$, the fibre ${\overline
V}(p)$ consists exactly of the finite union $\bigcup_{1\leq j\leq
m}l_{\gamma_{p}^{j}}$ of the tangent lines to the $m$ branches at
$p$.\\

\end{lemma}

\begin{proof}
 The fact that $V$ defines an irreducible algebraic variety follows
easily from arguments given in Section 1 of \cite{depiro6}. Let
$G_{w,1}$ be the Grassmannian of lines in $P^{w}$, let $G_{w,w-2}$
be the Grassmannian of planes of codimension $2$ in $P^{w*}$ and
let $I$ be the duality correspondence between $G_{w,1}$ and
$G_{w,w-2}$, as defined in Lemma 5.1. Without loss of generality,
we can find algebraic forms $\{G_{1},\ldots,G_{w-1}\}$ defining
$NonSing(C)$, see Section 1 of \cite{depiro6}. For each $G_{j}$,
with $1\leq j\leq w-1$, the differential $dG_{j}$ determines a morphism;\\

$dG_{j}:NonSing(C)\rightarrow P^{w*}:x\mapsto dG_{j}(x)=({\partial
G_{j}\over\partial X_{0}}(x):\ldots:{\partial G_{j}\over\partial
X_{n}}(x))$\\

We then obtain a morphism;\\

$\Phi_{1}:NonSing(C)\rightarrow G_{w,w-2}:x\mapsto
(dG_{1}(x),\ldots,dG_{w-1}(x))=P_{x}$\\

Using the duality correspondence $I$ and the observation that, for\\
$x\in NonSing(C)$, the tangent line $l_{x}$ is determined by the
intersection of the hyperplanes determined by $dG_{j}(x)$, for
$1\leq j\leq w-1$, see Section 1 of \cite{depiro6},
we obtain a morphism;\\

$\Phi_{2}:NonSing(C)\rightarrow G_{w,1}:x\mapsto l_{x}$\\

By construction, we must clearly have the duality
$I(l_{x},P_{x})$, whenever $x\in NonSing(C)$, $(\dag)$. Let
$C_{1}=\overline {Im(\Phi_{1})}$ and
$C_{2}=\overline{Im(\Phi_{2})}$. Let $\Gamma_{\Phi_{1}}\subset
C\times C_{1}$ and $\Gamma_{\Phi_{2}}\subset C\times C_{2}$ be the
irreducible correspondences defined by
$\overline{Graph(\Phi_{1})}$ and $\overline{Graph(\Phi_{2})}$. Let
$\Gamma_{\Phi_{2}}^{*}\subset C\times G_{w,w-2}$ be the dual
correspondence to
$\Gamma_{\Phi_{2}}$ defined by;\\

$\Gamma_{\Phi_{2}}^{*}(x,y)\equiv\exists
z(\Gamma_{\Phi_{2}}(x,z)\wedge I(z,y))$\\

We then have that $\Gamma_{\Phi_{2}}^{*}$ defines a closed
irreducible projective variety. By $(\dag)$,
$\Gamma_{\Phi_{2}}^{*}$ agrees with $\Gamma_{\Phi_{1}}$ on the
open subset obtained by restricting to $NonSing(C)$. Hence, we
have that $\Gamma_{\Phi_{1}}$ defines the dual correspondence to
$\Gamma_{\Phi_{2}}$. Let $W_{\Phi_{1}}\subset C\times C_{1}\times
P^{w*}$ and $W_{\Phi_{2}}\subset C\times C_{2}\times
P^{w}$ be defined by;\\

$W_{\Phi_{1}}=\{(y,P,z):\Gamma_{\Phi_{1}}(y,P)\wedge z\in P\}$\\

$W_{\Phi_{2}}=\{(y,l,z):\Gamma_{\Phi_{2}}(y,l)\wedge z\in l\}$\\

We have that $W_{\Phi_{1}}$ and $W_{\Phi_{2}}$ are closed
irreducible projective varieties. Let $p_{13}:C\times
C_{1,2}\times P^{w}\rightarrow C\times P^{w}$ be the projection
map. Let $V^{*}$ be defined as in Lemma 6.13 of \cite{depiro6}. We
have that $p_{13}(W_{\Phi_{1}})$ and $p_{13}(W_{\Phi_{2}})$ agree
with $V^{*}$ and $V$, restricted to $NonSing(C)$, hence
$p_{13}(W_{\Phi_{1}})=\bar V^{*}$ and $p_{13}(W_{\Phi_{2}})=\bar
V$. It follows that, for a singular point $p$ of $C$, the fibre
$\bar V(p)$ consists of a finite number of lines determined by the
planes appearing in the fibre $\bar V^{*}(p)$. The result of the
theorem then follows from the description of the fibre $\bar
V^{*}(p)$ given in Lemma 6.13 of \cite{depiro6}.
\end{proof}

One can also formulate the following "desingularised version" of Lemma 5.3;\\

\begin{lemma}

Let $C\subset P^{w}$ be any projective algebraic curve, with a choice of nonsingular model $C^{ns}\subset P^{w'}$ and birational morphism $\Phi:C^{ns}\rightarrow C$. Then;\\

Desingularised Version of Lemma 5.3:\\

Let $U=\Phi^{-1}(NonSing(C))\subset C^{ns}$ and let;\\

$V\subset U\times P^{w}$ be $\{(x,y):x\in U\wedge y\in l_{\Phi(x)}\}$\\

Then $V$ defines an irreducible algebraic variety and if ${\overline V}\subset C^{ns}\times P^{w}$ defines its Zariski closure, then, for $p_{j}\in C^{ns}$, corresponding to a branch $\gamma_{p}^{j}$ of $C$, the fibre ${\overline V}(p_{j})$ consists exactly of the tangent line $l_{\gamma_{p}^{j}}$.\\

\end{lemma}

\begin{proof}

The proof is merely a question of changing the parameter space from $C$ to $C^{ns}$ and adapting the argument of the previous lemma. The details are left to the reader.

\end{proof}

In a similar vein, we also have;\\

\begin{lemma}{Intuitive Construction of Tangent Lines}\\

Let $C\subset P^{w}$ be a projective algebraic curve and let $O\in C$ be a given fixed point, possibly singular. Let;\\

$V\subset(NonSing(C)\setminus \{O\})\times P^{w}$ be $\{(x,y):x\in
NonSing(C)\wedge y\in l_{Ox}\}$\\

Then $V$ defines an irreducible algebraic variety and, if
${\overline V}\subset C\times P^{w}$ defines its Zariski closure,
then, if $O$ is the origin of branches
$\{\gamma_{p}^{1},\ldots,\gamma_{p}^{m}\}$, the fibre ${\overline
V}(O)$ consists exactly of the finite union $\bigcup_{1\leq j\leq
m}l_{\gamma_{p}^{j}}$ of the tangent lines to the $m$ branches at
$O$.\\

\end{lemma}

\begin{proof}

The proof is the same as Lemma 5.3, except that we use the definition of $V^{*}$ given in Lemma 6.14 of \cite{depiro6} and the description of the fibre $\bar V^{*}(O)$ given in Lemma 6.14 of \cite{depiro6}.

\end{proof}

We also require the following result;\\

\begin{lemma}{Convergence of Intersections of Tangents}\\

Let $C\subset P^{2}$ be a plane projective algebraic curve, not equal to a line, let $O$ be a nonsingular point, which is the origin of a branch $\gamma_{O}$ of character $(1,k-1)$, $(k\geq 2)$, and suppose that $char(L)$ is zero or coprime to $k$. Let;\\

$V\subset(NSing(C)\setminus \{O\})\times P^{2}=\{(x,y):x\in NSing(C)\wedge y\in l_{O}\cap l_{x}\}$\\

Then $V$ defines a generically finite cover of $NonSing(C)$. Moreover, if $\bar V\subset C\times P^{2}$ defines its Zariski closure, then the fibre $\bar V(O)$ consists exactly of the point $O$.

\end{lemma}

\begin{proof}

The fact that $V$ defines a generically finite cover of $Nonsing(C)$ follows easily from the assumption that $C$ is not a line. The infinite fibres of the cover $V$ correspond to the, at most, finitely many points $\{p_{1},\ldots,p_{n}\}\subset (NonSing(C)\setminus\{O\})$ such that $l_{O}$ defines their tangent line. Let $U=(NonSing(C)\setminus\{O,p_{1},\ldots,p_{n}\})$ and let $V^{res}$ be the restriction of the cover $V$ to $U$. We clearly have that $\bar V(O)=\bar V^{res}(O)$ and that the fibres of $V^{res}$ consist of a unique intersection. In particular, the Zariski closure $\bar V^{res}$ is irreducible. Using the fact that $O$ is nonsingular and an important property of Zariski structures, given in Theorem 3.3 of \cite{Z}, it follows that the fibre $\bar V^{res}(O)$, hence the fibre $\bar V(O)$ also consists of a unique point $O'$. In order to show that $O=O'$, we employ an argument using infinitesimals;\\

We may assume that $O$ lies at the origin $(0,0)$ of the affine coordinate system $(x,y)$ and that the tangent line $l_{O}$ corresponds to the line $y=0$ in this coordinate system. Using Lemma 3.7, we can find a parametrisation of the branch $\gamma_{O}$, in the sense of Theorem 6.1 of \cite{depiro6}, of the form $(t,y(t))$, where $y(t)$ is an algebraic power series. By the assumption on the tangent line $l_{O}$, the definition of a parametrisation in Theorem 6.1 of \cite{depiro6} and the assumption on the character of the branch, we must have that $ord_{t}(y(t))=k$, so we can assume that $y(t)=t^{k}u(t)$, for the given $k\geq 2$, where $u(t)\in L[[t]]$ is a unit. Let $F(x,y)$ define $C$ in the coordinate system $(x,y)$, then we have that;\\

$F(t,y(t))=0$ and $F_{x}|_{(t,y(t))}.1+F_{y}|_{(t,y(t))}.y'(t)=0$, $(*)$\\

as a formal identity in the power series ring $L[[t]]$, see \cite{depiro6} for similar calculations. We now work in the nonstandard model $K=L[[\epsilon]]^{alg}$. It follows easily from the paper \cite{depiro3} that we can interpret $\epsilon$ as an infinitesimal in ${\mathcal V}_{0}$. By construction, we clearly have that the identity $(*)$ holds, replacing $t$ by $\epsilon$, $(**)$. By definition of the specialisation map, given in \cite{depiro3}, and $(**)$, we then have that $(\epsilon,y(\epsilon))\in NonSing(C)\cap{\mathcal V}_{(0,0)}$. Using $(**)$ again, we also have that the equation of the tangent line $l_{(\epsilon,y(\epsilon))}$ in the nonstandard model $K$ is given by;\\

$(y-y(\epsilon))=y'(\epsilon)(x-\epsilon)$ $(***)$\\

We now compute the intersection of $l_{(\epsilon,y(\epsilon))}$ with $y=0$. By $(***)$, we obtain that;\\

$l_{(0,0)}\cap l_{(\epsilon,y(\epsilon))}=(x_{\epsilon},0)$, where $x_{\epsilon}={(\epsilon y'(\epsilon)-y(\epsilon))\over y'(\epsilon)}$\\

Using the assumption on $char(L)$, it is easy to calculate that $ord_{t}(ty'(t)-y(t))\geq k$ and $ord_{t}(y'(t))=k-1$. It follows that $ord_{t}(x_{t})\geq 1$, considered as an element of $L[[t]]$, hence, by the definition of the specialisation map in \cite{depiro3}, we must have that $x_{\epsilon}\in{\mathcal V}_{0}$ and $(x_{\epsilon},0)\in{\mathcal V}_{(0,0)}={\mathcal V_{O}}$. By construction, we have that $\bar V^{res}((\epsilon,y(\epsilon)),(x_{\epsilon},0))$, hence, by specialisation, we have that $\bar V^{res}(O,O)$. It follows that the fibre $\bar V^{res}(O)$ consists exactly of the point $O$ as required.
\end{proof}

We now make the following definition;\\

\begin{defn}
Let $C\subset P^{2}$ be a plane projective curve. We define a nonsingular point $p$, which is the origin of a branch $\gamma_{p}$ having character $(1,r)$, for $r\geq 2$, to be a flex. We define $p$ to be an ordinary flex, if $r=2$.

\end{defn}

We now consider the duality construction applied to plane projective curves.
We claim the following;\\

\begin{lemma}

Let $C\subset P^{2}$ be a plane projective curve, defined in the coordinate system $\{X,Y,Z\}$ by the irreducible polynomial $F(X,Y,Z)$. Then, if $C$ is not a line, the differential;\\

$dF:NonSing(C)\rightarrow P^{2*}:x\mapsto l_{x}=({\partial F\over\partial X}(x):{\partial F\over\partial Y}(x):{\partial F\over\partial Z}(x))$\\

defines a morphism, with the property that $C^{*}=\overline{Im(dF)}\subset P^{2*}$ is also a plane projective curve. If $char(L)\neq 2$, then the following conditions are equivalent;\\

(i). $C$ has finitely many flexes.\\

(ii). $dF:C\leftrightsquigarrow C^{*}$ defines a birational morphism and $C=C^{**}$.\\

In particular, if $char(L)=0$, then $(ii)$ always holds, hence any plane projective curve over $L$, with $char(L)=0$, can only have finitely many flexes, and if $C$ has infinitely many flexes, the duality morphism $dF$ is inseperable.

\end{lemma}

\begin{proof}

The fact that $dF$ defines a morphism on $NonSing(C)$ follows easily from the fact that the partial derivatives $\{{\partial F\over \partial X},{\partial F\over\partial Y},{\partial F\over\partial Z}\}$ cannot all vanish at a nonsingular point of $C$. By  basic results in algebraic geometry, the image $Im(dF)\subset P^{2*}$ is constructible and irreducible. As $C$ is not a line, it must have infinite distinct tangent lines, hence, $C^{*}=\overline{Im(dF)}$ defines a plane projective algebraic curve. We first show;\\

$(i)\Rightarrow(ii)$\\

As $C$ has finitely many flexes, there exists an open subset $U\subset NonSing(C)$ with the property that, for $x\in U$, the corresponding branch $\gamma_{x}$ has character $(1,1)$. As $C^{*}$ is a projective algebraic curve, we can find a further open set $V\subset U$ such that $dF(V)\subset NonSing(C^{*})$. If $G(U,V,W)$ is a defining equation for $C^{*}$, then, using the first part of this Lemma and Lemma 5.1, in order to identify $P^{2**}$ with $P^{2}$, we obtain a morphism;\\

$dG:NonSing(C^{*})\rightarrow P^{2}$\\

We now claim that;\\

$(dG\circ dF)=Id_{V}:V\rightarrow P^{2}$ $(*)$\\

This implies immediately that $C^{**}=C$ and $dF:C\leftrightsquigarrow C^{*}$ is a birational map, with birational inverse $dG:C^{*}\leftrightsquigarrow C$. We now show $(*)$. Suppose that $x_{0}\in V$, with corresponding $y_{0}=dF(x_{0})\in NonSing(C^{*})$. As the branch $\gamma_{x_{0}}$ has character $(1,1)$, and $char(L)\neq 2$, we have that the result of Lemma 5.6 holds for $x_{0}$, (replacing O in the Lemma). That is, if $V_{1}\subset C\times P^{2}$ is the closed subvariety given in Lemma 5.6, then the fibre $V_{1}(x_{0})=x_{0}$ and, if $x\in C$, with $x\neq x_{0}$, then $V_{1}(x)=l_{x_{0}}\cap l_{x}$. As $y_{0}\in NonSing(C^{*})$, we can apply Lemma 5.5, and obtain a closed subvariety $V_{2}\subset C^{*}\times P^{2*}$ such that the fibre $V_{2}(y_{0})=l_{y_{0}}$ and, if $y\in C^{*}$, with $y\neq y_{0}$, then $V_{2}(y)=l_{y_{0}y}$. Shrinking $V$ if necessary, we can assume that if $x\in V$ with $x\neq x_{0}$, then $dF(x)\neq y_{0}$ and $l_{x}\cap l_{x_{0}}$ is a point, $(\dag)$. Now, define the closed relation $S\subset V\times P^{2}\times P^{2*}$ by;\\

$S(x,x',y)\equiv V_{1}(x,x')\wedge V_{2}(dF(x),y)$\\

By construction of $\{V_{1},V_{2}\}$, we have that if $x\in V$ with $x\neq x_{0}$, then the fibre $S(x)=\{(x',y):x'\in l_{x_{0}}\cap l_{x}, y\in l_{y_{0}dG(x)}\}$. Using the assumption $(\dag)$, the fibre $S(x)$ consists of a point and a line. Moreover, the point of intersection $l_{x_{0}}\cap l_{x}\in P^{2}$ is in dual correspondence with the line $l_{y_{0}dF(x)}\subset P^{2*}$, $(**)$. Using Lemma 5.1, this follows from determining the intersection of the pencil of lines parametrised by $l_{y_{0}dF(x)}=l_{dF(x_{0})dF(x)}$. By construction of $dF$, this is exactly the intersection of tangent lines $l_{x_{0}}\cap l_{x}$.\\

We now use $(**)$ and a simple limiting argument, to show that this duality must also hold for the fibre $S(x_{0})$. Let $R_{2}\subset P^{2}\times P^{2*}$ be the incidence relation, given by $R_{2}(u,v)$ iff $v\in H_{u}$, where $H_{u}$ is the linear form with coefficients given by $u$. By $(**)$, we have that;\\

$S(x)\subset R_{2}$ $(x\in V,x\neq x_{0})$ $(***)$\\

Using elementary model theoretic arguments, $(***)$ is a closed condition on $V$. Hence, we must have that $S(x_{0})\subset R_{2}$ as well, $(****)$.\\

Again, by construction of $\{V_{1},V_{2}\}$, we have that the fibre $S(x_{0})=\{(x_{0},y):y\in l_{y_{0}}\}$. By $(****)$, this shows that $x_{0}$ must be in dual correspondence with $l_{y_{0}}$. We, therefore, must have that $dG(y_{0})=x_{0}$, hence $(*)$ is shown.\\

$(ii)\Rightarrow (i)$.\\

There are two approaches to this problem. We will show later in the section that if $(ii)$ holds and $[(dF)^{-1}]^{*}$ is the induced bijection on branches guaranteed by Lemma 5.7 of \cite{depiro6}, then, for a given branch $\gamma$ of $C$ with character $(\alpha,\beta)$, the corresponding branch $\gamma^{*}$ of $C^{*}$ has character $(\beta,\alpha)$. If $C$ had infinitely many flexes, this would imply that $C^{*}$ had infinitely many singular points, which is impossible. Hence, $(i)$ must hold. For now, we will give a more direct algebraic proof.\\

Suppose that $(ii)$ holds and $C$ has infinitely many flexes. By Remarks 6.6 of \cite{depiro6} and using a similar argument to the above, we can assume that there exists an open $U\subset NonSing(C)$, with the property that $dF(U)\subset NonSing(C^{*})$ and every $p\in U$ is a flex. For ease of notation, we abbreviate the dual morphism $dF$ by $F^{dual}$. We will show directly that if $p$ is in $U$, then;\\

$dF^{dual}_{(a,b)}:T_{p,C}\rightarrow T_{F^{dual}(p),C^{*}}$ is identically zero, $(*)$.\\

where we have used the differential  and tangent space notation, given on p170 of \cite{Mum}. By standard algebraic considerations, see Section 1 of \cite{depiro6}, this implies that the morphism $F^{dual}$ is ramified in the sense of algebraic geometry at every point of $U$. Using results of \cite{depiro7}, see particularly Theorem 2.8, and a standard results about Zariski covers, that there can only exist finitely many ramification points, we conclude that the morphism $F_{dual}$ is inseperable. In particular, it cannot be birational as required.\\

In order to show $(*)$, let $\{X,Y,Z\}$ be a choice of coordinates for $P^{2}(L)$ and, without loss of generality, assume that $U\subset (Z\neq 0)$. Working in the coordinate system $\{x={X\over Z},y={Y\over Z}\}$,without loss of generality, we can assume that a given flex $O$ of $U$ is located at the origin $(0,0)$ of the coordinate system $(x,y)$ and its tangent line corresponds to $y=0$. Arguing as in Lemma 5.5, we can find a parametrisation of the branch $\gamma_{O}$ of the form $(t,y(t))$, where $ord_{t}y(t)\geq 3$. As in Lemma 5.6, we can consider $x_{\epsilon}=[\epsilon:y(\epsilon):1]$ as defining a point in $C\cap{\mathcal V}_{O}$, for an appropriate choice of non-standard model $K$. The coordinates of $F^{dual}(x_{\epsilon})$ in the dual space $P^{2*}$ are then given by taking the cross product $\phi(\epsilon)\times\phi'(\epsilon)$, where $\phi(\epsilon)=(\epsilon,y(\epsilon),1)$. We have;\\

$\phi(\epsilon)\times\phi'(\epsilon)=(\epsilon,y(\epsilon),1)\times (1,y'(\epsilon),0)=(-y'(\epsilon),1,\epsilon y'(\epsilon)-y(\epsilon))$\\

so $F^{dual}(x_{\epsilon})=[-y'(\epsilon):1:\epsilon y'(\epsilon)-y(\epsilon)]\in C^{*}\cap{\mathcal V}_{F^{dual}(O)}$, see also the corresponding calculation in Lemma 5.6. Now let $\{U,V,W\}$ be projective coordinates
for $P^{2*}(L)$ and let $\{u={U\over V},w={W\over V}\}$. Let $F^{dual}_{u}$ and $F^{dual}_{w}$ be the components of $F^{dual}$ with respect to the affine coordinate system $(u,w)$. Then;\\

$F^{dual}_{u}(\epsilon,y(\epsilon))=-y'(\epsilon)$\ \ \ \ \  $F^{dual}_{v}(\epsilon,y(\epsilon))=\epsilon y'(\epsilon)-y(\epsilon)$\\

Differentiating these expressions with respect to $\epsilon$, see \cite{depiro6} for the justification of such calculations in non-zero characteristic, we obtain;\\

${\partial F_{u}^{dual}\over \partial x}|_{(\epsilon,y(\epsilon))}.1+{\partial F_{u}^{dual}\over \partial y}|_{(\epsilon,y(\epsilon))}.y'(\epsilon)=-y''(\epsilon)$\\

${\partial F_{v}^{dual}\over \partial x}|_{(\epsilon,y(\epsilon))}.1+{\partial F_{v}^{dual}\over \partial y}|_{(\epsilon,y(\epsilon))}.y'(\epsilon)=\epsilon y''(\epsilon)$ $(**)$\\

Now setting $\epsilon=0$ in $(**)$ and using the fact that $ord_{t}y(t)\geq 3$, we obtain immediately the result $(*)$, hence this direction of the lemma is shown.\\

In order to finish the result, observe that if $char(L)=0$, one can use Lemma 5.6 directly and the argument of the first part of this Lemma $((i)\Rightarrow(ii))$ to show directly that the dual morphism $dF:C\leftrightsquigarrow C^{*}$ is birational. If $C$ has infinitely many flexes, then one can use the argument of the second part of the Lemma $((ii)\Rightarrow (i))$ to show that the dual morphism $dF:C\rightarrow C^{*}$ is inseperable.

\end{proof}

\begin{rmk}
For the remainder of this section, we will always assume that $char(L)\neq 2$ a given projective algebraic curve $C$ has \emph{finitely} many flexes. As the duality morphism is then birational, this will allow us to use the theory of branches that we have have developed in previous papers. The exceptional case that $C$ has infinitely many flexes was studied extensively in the paper \cite{Pardini}. We give a brief summary of the main results;\\

(Corollary 2.2) Let $C$ be a non-singular projective algebraic curve of degree $n$, with infinitely many flexes. Then, if $char(L)=p\neq 2$, we have that $p|n-1$.\\

(Proposition 3.7) Let $C$ be a non-singular projective algebraic curve of degree $p+1$, with infinitely many flexes, and $char(L)=p\neq 2$, then $C$ is projectively equivalent to the plane curve with equation;\\

$XY^{p}+YZ^{p}+ZX^{p}=0$\\

In particular, by direct calculation, the duality morphism $dF$ is purely inseperable, hence, biunivocal, and $C=C^{**}$.\\

(Corollary 4.3, Lemma 4.4 and Proposition 4.5)\\

Let $C$ be a generic non-singular projective algebraic curve of degree $dp+1$, with $(d>1)$, then the duality morphism $dF$ is biunivocal, but $C\neq C^{**}$.\\

The case of singular projective algebraic curves with infinitely many flexes is more difficult. For example, the graph of Frobenius, given by;\\

$YZ^{p-1}=X^{p}$\\

has degree $p$, infinitely many flexes and is singular. The dual curve $C^{*}$ is given by $X=0$, in particular $C^{**}$ is a point, hence $C\neq C^{**}$, and the duality morphism $dF$ is purely inseperable, therefore, biunivocal.\\

We will return to these examples and the case when $char(L)=2$ in the final section of this paper.

\end{rmk}

We now make the following definition;\\

\begin{defn}

Let $C\subset P^{2}$ be a plane projective curve. We define a multiple tangent line $l$ of $C$ to be a line which is tangent to \emph{at least} 2 branches of $C$.

\end{defn}

We now prove the following;\\

\begin{lemma}

Let $C$ be a projective algebraic curve, not equal to a line, with finitely many flexes, then every multiple tangent line of $C$ corresponds to a singularity of $C^{*}$, in particular, $C$ has finitely many multiple tangent lines.

\end{lemma}

\begin{proof}

By the hypotheses and Lemma 5.7, the duality morphism $dF$ is birational. By Lemma 5.7 of \cite{depiro6}, the duality morphism induces a bijection $[(dF)^{-1}]^{*}$ between the branches of $C$ and of $C^{*}$. We first claim the following;\\

Let $\gamma$ be a branch of $C$ with tangent line $l_{\gamma}$, and let $O_{\gamma}$ be the point defined by $l_{\gamma}$ in $P^{2*}$, then the corresponding branch $[(dF)^{-1}]^{*}(\gamma)$ of $C^{*}$ passes through $O_{\gamma}$. $(*)$\\

The claim $(*)$ is trivially true, by definition of the duality morphism, if $\gamma$ is centred at a non-singular point of $C$, $(**)$. Otherwise, we obtain the result by a straightforward limiting argument;\\

We use the notation and variety $\bar V$ of Lemma 5.4. Then define $S\subset C^{ns}\times P^{2}\times P^{2*}$ by;\\

$S(x,y,z)\equiv {\overline V}(x,y)\wedge \Gamma_{[\Phi\circ dF]}(x,z)$\\

where $\Phi\circ dF:C^{ns}\rightarrow C^{*}$ is a birational morphism. Let $R_{1}\subset P^{2}\times P^{2*}$ be the incidence relation defined by $R_{1}(u,v)$ iff $u\in H_{v}$, where $H_{v}$ is the linear form with coefficients given by $v$. If $x\in \Phi^{-1}(NonSing(C))$, then, by $(*)$ and Lemma 5.4, we have that the fibre $S(x)$ consists of the tangent line $l_{\Phi(x)}\subset P^{2}$ and its corresponding point of the dual space $dF\circ\Phi(x)\in P^{2*}$. In particular $S(x)\subset R_{1}$, $(***)$. As $S$ and $R_{1}$ are closed varieties, the relation $(***)$ holds for all
$x\in C^{ns}$. Now fix any branch $\gamma$ of $C$ with corresponding point $Q_{\gamma}\in C^{ns}$, then, by $(***)$, the fibre $S(Q_{\gamma})$ consists of the tangent line $l_{\gamma}$ and its corresponding point $O_{\gamma}\in P^{2*}$. From the definition of $S$ we have that $\Phi\circ dF(Q_{\gamma})=O_{\gamma}$. Hence, the corresponding branch $[(dF)^{-1}]^{*}(\gamma)$ of $C^{*}$ must be centred at $O_{\gamma}$, see Lemma 5.7 of \cite{depiro6}. Therefore, the result $(*)$ is shown.\\

Now, if $l$ is a multiple tangent line to $C$, by the previous definition, there exist at least $2$ distinct branches $\{\gamma_{1},\gamma_{2}\}$ of $C$ such that $l=l_{\gamma_{1}}=l_{\gamma_{2}}$. By the claim $(*)$, if $O_{l}$ is the corresponding point of $P^{2*}$, then the corresponding branches $\{[(dF)^{-1}]^{*}(\gamma_{1}),[(dF)^{-1}]^{*}(\gamma_{2})\}$ of $C^{*}$ both pass through $O_{l}$. By Lemma 5.4 of \cite{depiro6}, $O_{l}$ must be a singular point of $C^{*}$. Hence, the Lemma follows immediately from the fact that a plane projective curve can only have finitely many singular points.

\end{proof}

\begin{lemma}

Let $C$ be a projective algebraic curve, with finitely many flexes, then $Cl(C)$, as defined in Definition 4.1, is the same as $deg(C^{*})$ and $deg(C)$ is the same as $Cl(C^{*)}$. In particular, if $C$ has at most nodes as singularities, then;\\

$deg(C^{*})=n(n-1)-2d$\\

\end{lemma}

\begin{proof}

By Definition 1.12 of \cite{depiro6}, $deg(C^{*})$ is given by the number of distinct intersections $\{p_{1},\ldots,p_{n}\}$ of $C^{*}$ with a generic line $l\subset P^{2*}$. Let $O_{l}$ be the corresponding generic point of $P^{2}$ and let $\{L_{p_{1}},\ldots,L_{p_{n}}\}$ be the corresponding lines of $P^{2}$. We clearly have that each line $L_{p_{j}}$ passes through $O_{l}$, for $1\leq j\leq n$. If $\{U_{[dF]},V_{[dF]}\}$ are the canonical open subsets of $\{C,C^{*}\}$, with respect to the birational morphism $dF$, see Section 1 of \cite{depiro6}, we can assume that all the intersections $\{p_{1},\ldots,p_{n}\}$ lie inside $V_{[dF]}$. Let $\{q_{1},\ldots,q_{n}\}$ be the corresponding non-singular points of $C$. By the definition of the duality map $dF$, the tangent line $l_{q_{j}}$ is exactly $L_{p_{j}}$, for $1\leq j\leq n$. Hence, by Definition 4.1, we must have that $Cl(C)\geq n$. If $Cl(C)\geq n+1$, we could, without loss of generality, find a further $q_{n+1}$, distinct from $\{q_{1},\ldots,q_{n}\}$, lying inside $U_{[dF]}$, witnessing the class of $C$, see Definition 4.1. In this case, the tangent line $l_{q_{n+1}}$ would pass through $O$, hence, again just using the definition of the duality map $dF$, the corresponding point $p_{n+1}$ of $P^{2*}$ would lie on $C^{*}\cap l\cap V_{[dF]}$. This would give at least $n+1$ intersections of $C$ and $l$, which is a contradiction. Hence, $Cl(C)=deg(C^{*})=n$ as required. The claim that $deg(C)=Cl(C^{*})$ follows from the same argument, replacing $C$ by $C^{*}$ and $C^{*}$ by $C^{**}$, and using the fact that $C=C^{**}$. Finally, the relation $deg(C^{*})=n(n-1)-2d$ follows immediately from the Plucker formula, proved in Theorem 4.3, and the relation $Cl(C)=deg(C^{*})$, which we have just shown.

\end{proof}

We now show the following important result;\\

\begin{Theorem}{Transformation of Branches by Duality}\\

Let $C$ be a plane projective algebraic curve, with finitely may flexes, then, if $\gamma$ is a branch of $C$ with character $(\alpha,\beta)$, such that $\{\alpha,\alpha+\beta\}$ are coprime to $char(L)=p$, the corresponding branch of $C^{*}$ has character $(\beta,\alpha)$.

\end{Theorem}

\begin{rmk}
One can find algebraic "proofs" of this result in the literature. These proofs use a local parametrisation of the branch and an analysis of the resulting parametrisation after applying the duality morphism. Unfortunately, such proofs fail to give the correct answer in the case when $C$ has infinitely many flexes. The explanation of this discrepancy is that in such cases, the resulting "parametrisations" fail to give parametrisations in the sense of Theorem 6.1 of \cite{depiro6}. One requires the fact that the duality morphism is birational in order to show this stronger claim. This suggests that a geometric proof of this result is required. The one that we give is based on Severi's methods.
\end{rmk}
\begin{proof}
Let $\gamma^{0}$ be a given branch of $C$, with character $(\alpha,\beta)$, as in the statement of the theorem, centred at $O$, and let $l_{\gamma^{0}}$ be its tangent line. By $(*)$ of Lemma 5.11, the corresponding branch $\gamma^{0*}$ of $C^{*}$ is centred at the corresponding point $O'=O_{l_{\gamma}^{0}}$ of the dual space $P^{2*}$. Let $l$ be a generic line through $O'$, and let $A_{l}$ be the corresponding point of the dual space $P^{2}$. As $O_{l_{\gamma}^{0}}\in l$, we have, by duality, that $A_{l}\in l_{\gamma^{0}}$, and, by genericity, that $A_{l}\neq O$. The line $l$ then parametrises the pencil of lines passing through $A_{l}$. Now pick a further generic point $B$ belonging to $l$ and let $l_{B}$ be the corresponding line in the dual space $P^{2}$. Again, as $B\in l$, we obtain, by duality that $A_{l}\in l_{B}$, and, by genericity, that $l_{B}\neq l_{\gamma^{0}}$. Now choose a parametrisation $\Phi:P^{1}\rightarrow l_{B}$ such that $\Phi(0)=A_{l}$. We will denote the corresponding pencil of lines in $P^{2*}$ by $\{l_{t}\}_{t\in P^{1}}$. By construction, we have that $l_{0}=l$ and that the pencil
$\{l_{t}\}_{t\in P^{1}}$ is centred at $B$. Now, considering $P^{2*}$ as parametrising the set of lines in $P^{2}$, it defines a $g_{n}^{2}$ on $C$, where $n=deg(C)$. Each line $l_{t}\subset P^{2*}$, then defines a $g_{n}^{1}\subset g_{n}^{2}$ on $C$. We will denote this $g_{n}^{1}$ by $g_{n}^{1,t}$.\\

By a suitable choice of coordinates, we may, without loss of generality, assume that the line $l_{B}$ corresponds to the line $l_{\infty}$ in the affine coordinate system $\{x={X\over Z},y={Y\over Z}\}$, the point $\Phi(0)=A_{l}$ corresponds to $[0:1:0]$, the point $O$ corresponds to $[0:0:1]$ and the point $\Phi(\infty)$ corresponds to $[1:0:0]$. By choice of $l_{B}$, we can assume that the line $l_{\infty}$ intersects $C$ transversely in ordinary simple points $\{p_{1},\ldots,p_{n}\}$. Let $t_{j}=\Phi^{-1}(p_{j})$ and let $U=P^{1}\setminus\{t_{1},\ldots,t_{n},\infty\}$. By the choice of $A_{l}$, we can assume that $0\in U$. Now, for $t\in U$, the corresponding $g_{n}^{1,t}$, defined above, has no fixed branch contribution. We then have that $g_{n}^{1,t}=(h_{t})$, where $h_{t}$ is the non-constant rational function on $C$ defined by $h_{t}(x,y)=x-ty$, see Lemma 2.4. In particular, we have that $g_{n}^{1,0}$ is defined by $(x=0)$. Now let $U_{[dF]}$ be the canonical set associated to the birational morphism $dF:C\leftrightsquigarrow C^{*}$, and let $V\subset U_{[dF]}$ be obtained by removing;\\

(i). All the finitely many flexes from $U_{[dF]}$.\\

(ii). All the finitely many points on $C\cap l_{\infty}$ from $U_{[dF]}$.\\

(iii). All the finitely many non-singular points of $U_{[dF]}$, whose\\
 \indent \ \ \ \ \ \ \ tangent line is parallel to the $y$-axis.\\

Now, using the method before Lemma 3.9, we can, for $t\neq\infty$,  define the rational function ${dh_{t}\over dx}$. By the calculation there, if $C$ is defined by $f(x,y)=0$ in the coordinate system $(x,y)$, we have that;\\

${dh_{t}\over dx}={h_{x}f_{y}-h_{y}f_{x}\over f_{y}}=1+t{f_{x}\over f_{y}}$\\

Then, using the notation of Lemma 2.4, for $t\in {U\setminus\{0\}}$, $({dh_{t}\over dx}=0)$ is the same as $(-{f_{y}\over f_{x}}=t)$, considered as weighted sets on $C$. As $-{f_{y}\over f_{x}}$ is a non-constant rational function on $C$, then, by Lemma 2.4, we can associate a $g_{1}^{m}=(-{f_{y}\over f_{x}})$ to it.\\

Let $J_{t}=Jac(g_{n}^{1,t})$, see Definition 3.2 and Definition 3.22. Then $J_{t}$ consists of a weighted set of branches $\{\alpha_{1}\gamma_{1}^{t},\ldots,\alpha_{r}\gamma_{r}^{t}\}$. We claim the following;\\

$(J_{t}\cap V)=(g_{1}^{m}\cap V)$ for $(t\in {U\setminus\{0\}})$ $(\dag)$\\

In order to show $(\dag)$, let $\gamma$ be a branch of $V$. By the definition of $V$,
 $\gamma$ is centred at the finite position $(a,b)$, has character $(1,1)$ and its tangent line $l_{\gamma}$ is not parallel to the $y$-axis. Let $(x,y(x))$ be a parametrisation of $\gamma$, in the form given by Lemma 3.7. If $\gamma$ belongs to $J_{t}$, then $ord_{\gamma}(h_{t})=2$, $val_{\gamma}(h_{t})<\infty$ and $h_{t}$ determines an algebraic power series at the branch $\gamma$;\\

$h_{t}=\lambda+(x-a)^{2}\psi(x-a)$ with $\psi(0)\neq 0,\lambda<\infty$\\

We then have that;\\

${dh_{t}\over dx}=(x-a)[2\psi(x-a)]+[(x-a)\psi'(x-a)]$ ($char(L)\neq 2$)\\

At $x=a$, the expression in brackets reduces to $2\psi(0)\neq 0$, hence $ord_{\gamma}({dh_{t}\over dx})=1$ and $val_{\gamma}({dh_{t}\over dx})=0$. This shows that $\gamma$ is counted once in the weighted set $({dh_{t}\over dx}=0)$, and, therefore, once in the weighted set $W_{t}$ of the corresponding $g_{m}^{1}$ defined above. If $\gamma$ belongs to the $g_{m}^{1}$, then, $val_{\gamma}({dh_{t}\over dx})=0$, for some $t\in U\setminus\{0\}$. Reversing the above argument, using the fact that $val_{\gamma}(h_{t})<\infty$, we obtain that $ord_{\gamma}({dh_{t}\over dx})=2$, hence $\gamma$ belongs to $J_{t}$ as required. Therefore, $(\dag)$ is shown. Note that an almost identical argument to the above was carried out in Lemma 3.9.\\

We now consider the behaviour of the $g_{m}^{1}$ at the branch $\gamma^{0}$. We claim that $\gamma^{0}$ is counted $\beta$ times in the weighted set $W_{0}$ of this $g_{m}^{1}$, $(\dag\dag)$. Let $(x(s),y(s))$ be a parametrisation of the branch at $(0,0)$. Using the definition of a parametrisation in Theorem 6.1 of \cite{depiro6} and the fact that the tangent line $l_{\gamma_{0}}$ is given by $x=0$, we obtain immediately that $ord_{s}x(s)=\alpha+\beta$ and $ord_{s}y(s)=\alpha$. Using the fact that $f(x(s),y(s))=0$, we obtain that;\\

$f_{x}|_{(x(s),y(s))}x'(s)+f_{y}|_{(x(s),y(s))}y'(s)=0$\\

If either $f_{x}$ or $f_{y}$ is identically zero on $C$, then the image of the map $dF$ is contained in a line, which implies, by Lemma 5.8, that $C$ has infinitely many flexes. Hence, we must have that $x'(s)\equiv 0$ iff $y'(s)\equiv 0$. If both $x'(s)\equiv 0$ and $y'(s)\equiv 0$, we can find algebraic power series $\{x_{1}(s),y_{1}(s)\}$ such that $x(s)=x_{1}(s^{p})$ and $y(s)=y_{1}(s^{p})$. This would contradict the construction of a parametrisation, as defined in Theorem 6.1 of \cite{depiro6}, see also the method used in \cite{depiro7}, Remarks 2.3. Hence, we have $x'(s)\neq 0$, and $y'(s)\neq 0$, and;\\

$-{f_{y}\over f_{x}}|_{(x(s),y(s))}={x'(s)\over y'(s)}$\\

By the assumption on $\{\alpha,\beta\}$ in the statement of the theorem, we obtain that $ord_{s}(-{f_{y}\over f_{x}}|_{(x(s),y(s))})=(\alpha+\beta-1)-(\alpha-1)=\beta$. By definition of the $g_{m}^{1}$, this implies $(\dag\dag)$.\\

Now consider the $g_{m'}^{1}$ on $C^{*}$, given by the pencil $\{l_{t}\}_{t\in P^{1}}$. By Theorem 1.14, and the fact that $C$ and $C^{*}$ are birational, this transfers to a $g_{m'}^{1}$ on $C$. By choice of $B$, the $g_{m'}^{1}$ on $C$ and $C^{*}$ have no base branches. We claim that $g_{m}^{1}=g_{m'}^{1}$, $(\dag\dag\dag)$. Let $\{V,V'\}$ be canonical sets associated to the birational map $dF:C\leftrightsquigarrow C^{*}$, where $V$ was defined above. Then, using Lemma 2.17 of \cite{depiro6}, there exists an open subset $U'\subset U\subset P^{1}$ such that, for $t\in U'$, the corresponding weighted sets $\{W_{t},W_{t}'\}$ of the $\{g_{m}^{1}, g_{m'}^{1}\}$ on $C$, consist of branches(points) which are simple for $\{W_{t},W_{t}'\}$ respectively, based inside $V$ or $V'$. In this case, a point $p\in W_{t}'$ corresponds to a transverse intersection between $l_{t}$ and $C^{*}$. The line $l_{p}$ in the dual space $C$ is then a tangent line to a non-singular point $p'$ of $C$, having character $(1,1)$. We, therefore, have that $I_{p'}(C,l_{p})=2$, hence $p'$ is counted once for $Jac(g_{n}^{1,t})$. Combining this with the result $(\dag)$, we obtain that $W_{t}=W_{t}'$, hence $(\dag\dag\dag)$ is shown.\\

We now finish the proof of the theorem. By Theorem 1.12, the fact that the $g_{m}^{1}$ on $C$ has no fixed branches, and the result $(\dag\dag)$, for generic $t\in{\mathcal V}_{0}$, the weighted set $W_{t}\cap C\cap\gamma^{0}$ of the $g_{m}^{1}$ consists of $\beta$ distinct branches(points), centred in $U'\cap \gamma^{0}$. By $(\dag\dag\dag)$, these correspond to $\beta$ transverse intersections $C^{*}\cap l_{t}\cap\gamma^{0*}$. It follows immediately that $I_{\gamma^{0*}}(C^{*},l_{0})=\beta$, hence, as $l_{0}=l$ was chosen to be generic through $O'$, the order of the branch is $\beta$. We now follow through the same argument, replacing $C$ by $C^{*}$ and $C^{*}$ by $C^{**}=C$. It follows immediately that the class of the branch $\gamma^{0*}$ must be the order of the branch $\gamma^{0}$. Hence, the character of the branch $\gamma^{0*}$ is $(\beta,\alpha)$ as required.

\end{proof}

\begin{rmk}

The assumption that $\{\alpha,\alpha+\beta\}$ are co-prime to $char(L)=p$ is necessary. Consider the projective algebraic curve $C$ defined by;\\

$y=x^{p}+x^{p+2}$\\

In projective coordinates, this is defined by $YZ^{p+1}=X^{p}Z^{2}+X^{p+2}$ and the duality morphism $dF$ is given by;\\

$dF:[X:Y:Z]\mapsto [2X^{p+1}:-Z^{p+1}:2X^{p}Z-YZ^{p}]$\\

$\indent \ \ \ \ \ \ \ \ \ \ (x,y)\mapsto (-2x^{p+1},y-2x)$\\

The duality morphism is clearly seperable, hence $C$ must have finitely many flexes. $C$ is non-singular in the affine coordinate system $(x,y)$, and the character of the unique branch $\gamma_{0}$ of $C$ at the origin $(0,0)$ is $(1,p-1)$. The proof of the theorem shows that the order of the corresponding branch $\gamma_{0}^{*}$ of $C^{*}$ is $(p+1)\neq (p-1)$, for $char(L)\neq 2$.\\

One can easily construct further examples using appropriate rational parametrisations.

\end{rmk}

\begin{section}{A Generalised Plucker Formula}

The purpose of this section is to give a geometric proof of a generalisation of the Plucker formula of Section 4. We also discuss the question of representation of plane algebraic curves in greater detail. We first require the following definition;\\

\begin{defn}

We will define a plane projective algebraic curve $C$ to be normal if it is has finitely many flexes and all its branches $\gamma$ have character $(\alpha,\beta)$, with $\{\alpha,\alpha+\beta\}$ coprime to $char(L)=p$.

\end{defn}

We first show the following;\\

\begin{lemma}
Let $C$ be a normal plane projective curve with $class(C)=m$, as defined in Definition 4.1, $genus(C)=\rho$, as defined in Definition 3.33 and $order(C)=n$. Then;\\

$\rho={1\over 2}[m+\sum_{\gamma}(\alpha(\gamma)-1)]-(n-1)$\\

where the sum is taken over the finitely many singular branches and, for such a branch $\gamma$, $\alpha(\gamma)$ gives the order of the branch. In particular, if $C$ is a normal projective curve, having at most nodes as singularities, we obtain the formula shown earlier;\\

$\rho={m\over 2}-(n-1)$

\end{lemma}

\begin{proof}
By Remarks 4.2 and the fact that $C$ has finitely many flexes, we can suppose that, for generic $P$, the $m$ tangent lines of $C$ passing through $P$ are based at ordinary simple points. In particularly, $P$ does not lie on any of the finitely many tangent lines belonging to the singular branches of $C$. We consider the $g_{n}^{1}$ defined by the pencil of lines passing through $P$. We have that $Jac(g_{n}^{1})$ consists exactly of the $m$ ordinary branches witnessing the class of $C$ and the finitely many singular branches $\gamma$, each counted $\alpha(\gamma)-1$ times. In particular;\\

 $order(Jac(g_{n}^{1}))=m+\sum_{\gamma}(\alpha(\gamma)-1)$\\

  Now we obtain the first part of the lemma from the fact that $order(g_{n}^{1})=n$ and Definition 3.33 of the genus of $C$. If $C$ has at most nodes as singularities, then it has \emph{no} singular branches. Therefore, $(\alpha(\gamma)-1)=0$ for any branch $\gamma$ of $C$. The second part of the lemma then follows from the previous formula.

\end{proof}

Using duality, we have;\\

\begin{lemma}

Let $C$ be a normal plane projective algebraic curve, not equal to a line, with the invariants $\{m,n,\rho\}$ as defined above. Then;\\

$\rho={1\over 2}[n+\sum_{\gamma}(\beta(\gamma)-1)]-(m-1)$\\

where the sum is taken over the finitely many flexes, and, for such a branch $\gamma$, $\beta(\gamma)$ gives the class of the branch.

\end{lemma}

\begin{proof}
As $C$ is normal and not equal to a line, we may apply Lemma 5.8, Lemma 5.12 and Theorem 5.1. In particular, we have that $C^{*}$ has finitely many flexes, and, as in the previous lemma, for generic $P$, the tangent lines of $C^{*}$, passing through $P$, are based at ordinary simple points. We consider the $g_{n'}^{1}$ on $C^{*}$ defined by the pencil of lines passing through $P$. We have that $n'=order(C^{*})=class(C)=m$ and $class(C^{*})=order(C)=n$ by Lemma 5.12. We also have that $Jac(g_{m}^{1})$ consists exactly of the $n$ ordinary branches witnessing the class of $C^{*}$ and the finitely many singular branches, each counted $(\beta(\gamma)-1)$ times, by Theorem 5.1. Hence;\\

$order(Jac(g_{m}^{1}))=n+\sum_{\gamma}(\beta(\gamma)-1)$\\

where the sum is taken over the finitely many flexes of $C$, using the fact that $C$ and $C^{*}$ are birational, given in Lemma 5.8. We also have that $\rho=genus(C)=genus(C^{*})$ by Lemma 5.8 and Theorem 3.35. Hence, we obtain the first part of the theorem from the fact that $order(g_{m}^{1})=m$ and Definition 3.33 of the genus of $C^{*}$.

\end{proof}

\begin{theorem}{Generalised Plucker Formula}\\

Let $C$ be a normal plane projective algebraic curve, not equal to a line, with $\{m,n\}$ defined as in the previous lemmas. Then;\\

$3m-3n=\sum_{\gamma}(\beta(\gamma)-1)-\sum_{\gamma}(\alpha(\gamma)-1)$\\

where the sums are taken over the finitely many flexes and finitely many singular branches of $C$ respectively. In particular, if $C$ is a plane projective curve, having at most nodes as singularities and no flexes, then it is either a line or a smooth conic.

\end{theorem}

\begin{proof}
The first part of the theorem follows immediately by combining the first formulas
 given in Lemma 6.2 and Lemma 6.3. If $C$ is a smooth plane projective curve, having at most nodes as singularities and no flexes, then it must be normal as all its branches $\gamma$ have character $(1,1)$ ($char(L)\neq 2$). If $C$ is not a line, then we can apply the first part of the theorem, to obtain that $m=n$. Using the Plucker formula given in Theorem 4.3, we have that $m=n^{2}-n-2d$, where $d$ is the number of nodes of $C$. Hence, $n^{2}-n-2d=n$, which gives that $d={n(n-2)\over 2}$. If $n\geq 3$, this contradicts Theorem 3.37. Hence, $n=2$ and $d=0$. The second part of the theorem then follows.

\end{proof}

We now apply the above formulas to the study of normal plane projective curves, having at most nodes as singularities.

\begin{theorem}

Let $C$ be a normal plane projective curve, not equal to a line, having at most nodes as singularities, with the convention on summation of branches given above and $\{m,n,\rho,d\}$ as defined in the previous lemmas. Then, we obtain the class formula;\\

$3m-3n=\sum_{\gamma}(\beta(\gamma)-1)$ $(1)$\\

and the genus formula;\\

$6\rho+3n-6=\sum_{\gamma}(\beta(\gamma)-1)$ $(2)$\\

and the node formula;\\

$3n(n-2)-6d=\sum_{\gamma}(\beta(\gamma)-1)$ $(3)$\\

In particular, if $C$ has at most ordinary flexes, and $i$ is the number of these flexes, we obtain the class formula, referred to as Plucker III' in \cite{Sev};\\

$3m=3n+i$ $(4)$\\

and the genus formula;\\

$6\rho=i-3n+6$ $(5)$\\

and the node formula, referred to as Plucker III in \cite{Sev};\\

$6d=3n(n-2)-i$ $(6)$\\

\end{theorem}

\begin{proof}

The class formula $(1)$ follows from the Generalised Plucker formula and the fact that $(\alpha(\gamma)-1)=0$ for any branch $\gamma$ of $C$, as $C$ has at most nodes as singularities. The genus formula $(2)$ follows from $(1)$ and the formula $\rho={m\over 2}-(n-1)$, given in Lemma 6.2. The node formula $(3)$ follows from $(1)$ and the Plucker formula $m=n(n-1)-2d$, given in Theorem 4.3. The formulas $(4),(5),(6)$ all follow immediately from the corresponding formulas $(1),(2),(3)$ and the fact that;\\

$i=\sum_{\gamma}(\beta(\gamma)-1)$\\

as an ordinary flex has character $(1,2)$, hence, for the corresponding branch $\gamma$, $(\beta(\gamma)-1)=1$.

\end{proof}
\end{section}
\end{section}

\end{document}